\documentclass{article} 

\usepackage[english]{babel} 
\usepackage[utf8]{inputenc} 
\usepackage[T1]{fontenc}

\usepackage[a4paper,left=2.4cm,right=2.4cm,top=2.2cm,bottom=2.2cm]{geometry}  
\usepackage{tikz}
\usetikzlibrary[topaths] 

\usepackage{amsmath, mathrsfs, amssymb, amsthm, mathtools, bbm, bm} 
\usepackage{microtype}
\usepackage{booktabs,tabularx,array,longtable}
\usepackage{xcolor}
\usepackage[linesnumbered,ruled,vlined]{algorithm2e}

\usepackage{csquotes} 

\usepackage{caption,float} 
\usepackage{graphicx}
\usepackage{makecell} 
\usepackage{enumitem} 
\usepackage[numbers]{natbib}
\usepackage[hidelinks]{hyperref}
\usepackage{doi}
\usepackage[capitalise,noabbrev]{cleveref}

\newtheorem{theorem}{Theorem}[section]
 
\newtheorem{definition}[theorem]{Definition}
\newtheorem{lemma}[theorem]{Lemma}

\newtheorem{corollary}[theorem]{Corollary}
\newtheorem{proposition}[theorem]{Proposition}

\newtheorem*{theorem*}{Theorem}
\newtheorem*{example*}{Example} 
\newtheorem*{definition*}{Definition}
\newtheorem*{lemma*}{Lemma}
\newtheorem*{remark*}{Remark}
\newtheorem*{corollary*}{Corollary}
\newtheorem*{proposition*}{Proposition}
\newtheorem*{assumption*}{Assumption} 
\newtheorem*{claim*}{Claim} 

\crefname{conjecture}{Conjecture}{Conjectures}

\newtheoremstyle{TheoremNum}
    {\topsep}{\topsep}              
    {\itshape}                      
    {}                              
    {\bfseries}                     
    {.}                             
    { }                             
    {\thmname{#1}\thmnote{ \bfseries #3}}
\theoremstyle{TheoremNum}

\newtheoremstyle{LemmaNum}
    {\topsep}{\topsep}              
    {\itshape}                      
    {}                              
    {\bfseries}                     
    {.}                             
    { }                             
    {\thmname{#1}\thmnote{ \bfseries #3}}
\theoremstyle{LemmaNum}

\usepackage{float}
\usepackage{cleveref}

\usepackage{amsmath, mathrsfs, amssymb, amsthm, mathtools, bbm, bm} 

\usepackage{makecell} 

\usepackage{subfig}
\usepackage{float}

\usepackage{enumitem} 

\usepackage{algorithm} 
\usepackage{algcompatible} 

\newcommand{\F}{\mathcal{F}}

\newcommand{\e}{ \mathrm{ e } }

\newcommand{\R}{\mathbb{R}}
\newcommand{\Kop}{\mathcal{K}}
\newcommand{\Jop}{\mathcal{J}}

\newcommand{\Clust}{\operatorname{Clust}}

\newcommand{\argmax}{\operatorname*{arg\,max}}
\newcommand{\one}{\bm{1}}

\newcommand{\phimin}{\phi_{\!*}}
\newcommand{\floor}[1]{\left\lfloor #1\right\rfloor}
\newcommand{\ceil}[1]{\left\lceil #1\right\rceil}

\newcommand{\norm}[1]{\left\lVert#1\right\rVert}

\renewcommand{\[}{\left[ }
\renewcommand{\]}{\right] }

\renewcommand{\(}{\left( }
\renewcommand{\)}{\right) }

\begin{document} 

\title{Optimal Recursive Composition and Dyadic Phase Laws for Gradient Descent with Predetermined Stepsizes}

\author{
Yu Liu\textsuperscript{1}
\quad
Kang Chen\textsuperscript{1}
\quad
Rujun Jiang\textsuperscript{2,*}
\quad
Tianyu Wang\textsuperscript{1,*}
}

\date{}

\maketitle

\footnotetext[1]{
Shanghai Center for Mathematical Sciences, Fudan University.
Emails:
\href{mailto:yuliu22@m.fudan.edu.cn}{\nolinkurl{yuliu22@m.fudan.edu.cn}},
\href{mailto:kchen26@m.fudan.edu.cn}{\nolinkurl{kchen26@m.fudan.edu.cn}},
\href{mailto:wangtianyu@fudan.edu.cn}{\nolinkurl{wangtianyu@fudan.edu.cn}}.
}

\footnotetext[2]{
School of Data Science, Shanghai Key Laboratory for Contemporary Applied Mathematics,
Fudan University.
Email:
\href{mailto:rjjiang@fudan.edu.cn}{\nolinkurl{rjjiang@fudan.edu.cn}}.
}

\begingroup
\renewcommand{\thefootnote}{*}
\footnotetext{Corresponding authors.}
\endgroup

\begin{abstract}
Predetermined stepsize schedules featuring carefully chosen long steps have recently been shown to accelerate gradient descent (GD) on smooth convex functions. A prominent class of such schedules is built through recursive composition. In this paper, we characterize the convergence of these optimized recursive schedules, revealing a nonconstant log-periodic modulation across prescribed horizons. Specifically, we prove that for every $N \geq 1$, \textcolor{black}{the optimized ConPP/OBS-S schedules \citep{ZhangJiang,GrimmerShuWangMOR}} satisfy
\begin{align*}
    f(x_{N-1})-f_\ast \le \frac{1}{2N^p\Phi(\log_2N)-1} \frac{L}{2}\|x_0-x_\ast\|^2, \qquad p=\log_2(1+\sqrt2), 
\end{align*}
where $\Phi$ is a positive, Lipschitz, nonconstant $1$-periodic function. We derive this by proving that balanced splitting is optimal at every horizon for these constructions, resolving a conjecture of \textcolor{black}{\citet{ZhangJiang}}. Furthermore, for \textcolor{black}{their asymmetric counterparts} (ConPD/OBS-F schedules), we show that although optimal splits are not necessarily balanced, the same Silver exponent asymptotically persists alongside a distinct log-periodic modulation.

\end{abstract}



\section{Introduction}\label{sec:introduction}

Let $f:\R^d\to\R$ be convex with $L$-Lipschitz gradient and a nonempty minimizer set. Given a normalized schedule $h=(h_0,\ldots,h_{n-1})^\top\in\R^n$, gradient descent takes the form
\begin{align}\label{eq:gd}
    x_{t+1}=x_t-\frac{h_t}{L}\nabla f(x_t),
    \quad t=0,\ldots,n-1.
\end{align}
The schedule is fixed before optimization begins, so all algorithmic design lies in the predetermined scalar sequence $h$. This leads to a basic question: \emph{how much can gradient descent be accelerated by stepsize design alone?}


For $h_t\in(0,2)$, every nonstationary iteration decreases the objective. Tight worst-case analysis and carefully chosen varying steps can improve constants within this descent regime, while the best achievable order remains $n^{-1}$ \citep{DroriTeboulle,TeboulleVaisbourd}. One-dimensional Huber instances, in particular, rule out a uniform $o(n^{-1})$ guarantee for any schedule confined to this regime. Faster convergence therefore requires giving up monotonicity at the level of individual iterations. Long-step schedules trade this local monotonicity for progress over a larger block: some iterations may increase the objective, while their combined effect yields a stronger guarantee.


Recent work has made such long-step blocks systematic. Performance-estimation methods revealed numerically promising schedules \citep{TaylorHendrickxGlineur,DasGuptaVanParysRyu}, while subsequent constructions established first constant-factor and then polynomial improvements \citep{Grimmer2024,GrimmerShuWang2023}. A key structural device behind these advances is recursive composition: two shorter certified schedules are placed on either side of a specially chosen joining step, producing a longer certified schedule. Iterating this operation yields the Silver schedule of \citet{AltschulerParriloJACM,AltschulerParriloMP}, with objective-gap rate  $O(n^{-p})$ at horizons $n=2^k-1$, where $\rho\coloneqq1+\sqrt2$ and $p\coloneqq\log_2\rho\approx1.27155$. Later constructions sharpened its objective-gap constant and obtained the same exponent for the squared gradient norm \citep{GrimmerShuWangIJO}. At the dyadic horizons $n=2^k-1$, this recursion is exactly self-similar. The natural structural question is what replaces this symmetry at the horizons in between.

It is convenient to write $N=n+1$, so that the Silver horizons correspond exactly to $N=2^k$.  Away from these dyadic horizons, recursive composition becomes a split-selection problem: for each $N$, one must choose child sizes $m$ and $N-m$. Two concurrent frameworks formalize this optimization. \citet{ZhangJiang} introduced the ConPP and ConPD recursive constructions, while \citet{GrimmerShuWangMOR} independently developed the corresponding OBS-S and OBS-F families through composable certificates. The symmetric ConPP/OBS-S family extends the Silver construction to every horizon, whereas the asymmetric ConPD/OBS-F family is tailored to the final objective gap. Under matched split choices, the corresponding constructions generate the same numerical schedules \citep[Remark~5.9]{ZhangJiang}.

For the symmetric ConPP/OBS-S family, the joining rules reduce the optimization to choosing a split $m+(N-m)$. Exhaustive dynamic programming finds an optimal split for each $N$, while \citet[Conjecture~4.17]{ZhangJiang} conjectured a uniform structural rule: the balanced split is always optimal, with certain unbalanced splits attaining the same value. Establishing such a uniform rule is nontrivial because the join is symmetric but nonassociative: symmetry allows the two children to be exchanged, but gives no comparison between different hierarchical pairings.


We resolve this conjecture, including all unbalanced ties. The resulting balanced recurrence exposes an exact dyadic self-similarity: doubling $N$ multiplies the optimized scalar value by $\rho=1+\sqrt2$.  This scaling links the optimal composition structure to the behavior of the convergence bound across horizons. The Silver exponent $p=\log_2\rho$ determines the overall power-law decay, but leaves open whether the corresponding normalized coefficient approaches a single constant. We show instead that a nontrivial dependence on the fractional part of $\log_2 N$ remains. We call this quantity the \emph{dyadic phase}: it records the position of $N$ between consecutive powers of two. For the symmetric recursion, this phase dependence is exact at every finite horizon. For the asymmetric recursion, whose optimal splits need not be balanced, the normalized coefficients converge uniformly over dyadic blocks to a distinct nonconstant phase profile. Thus the Silver exponent together with the dyadic phase describes the leading-order horizon dependence of both optimized recursive families.

\noindent\textbf{Contributions.}
We establish three results on the structure and convergence of these optimized recursive schedules.

\emph{Optimal recursive structure.}
We prove that balanced splitting is optimal at every horizon, resolving the ConPP conjecture of \citet{ZhangJiang}, and characterize all unbalanced splits that attain the same value. The key mechanism is a sharp four-point rearrangement inequality that compares different hierarchical pairings and propagates through the composition tree. As a consequence, we characterize all optimal composition trees and obtain an $O(\log N)$-time, $O(1)$-workspace evaluator for a single optimized scalar value $U_N$ (Section~\ref{sec:balanced}).

\emph{Exact symmetric phase.}
The optimized values satisfy $U_N=N^p\Phi(\log_2N)$, where $\Phi$ is positive, Lipschitz, $1$-periodic, and nonconstant.
This yields an exact formula for the tight objective coefficient and the complete nondegenerate interval of normalized subsequential limits (Section~\ref{sec:phase}).

\emph{Asymmetric phase beyond balanced composition.}
For ConPD/OBS-F, optimal splits can be unbalanced, yet the normalized tight objective coefficients still converge uniformly over dyadic blocks to a distinct nonconstant phase profile. We further prove that this limiting profile stays strictly above the sharp homogeneous support bound obtained from the join map (Section~\ref{sec:obsf}).


\noindent\textbf{Tightness.}
The finite-horizon coefficients underlying both phase laws are exact worst-case objective coefficients. For ConPP/OBS-S, equality is attained by one-dimensional Huber instances; for ConPD/OBS-F, matching one-dimensional quadratic and Huber instances attain the objective bound.

\noindent\textbf{Positioning.}
Recent lower bounds identify the Silver exponent as the optimal polynomial exponent for predetermined nonnegative stepsizes \citep{ye2026silverratealmostoptimal}. A complementary line of work studies anytime acceleration, where a single infinite schedule is required to provide guarantees across all stopping times \citep{KornowskiShamir,ZhangLeeDuChen,TsaiEtAl}. Our results refine the finite-horizon picture for the ConPP/OBS-S and ConPD/OBS-F recursive families by characterizing their optimal composition structure and phase-dependent leading coefficients. 

\section{From Recursive Schedules to a Scalar Bellman Problem}\label{sec:scalar}\label{sec:scalar-conjugacy}


\paragraph{A common scalar representation.} 
The symmetric ConPP and OBS-S families admit the same scalar description. For a schedule $h$, define
\begin{align*}
    u(h):=1+\sum_i h_i.
\end{align*}
The empty schedule has value $u=1$. If two child schedules $a$ and $b$ have scalar values $u$ and $v$, respectively, both recursive constructions are governed by the same composition kernel
\begin{align}\label{eq:Kdef}
 \Kop(u,v)\coloneqq\frac{u+v+\sqrt{u^2+6uv+v^2}}2,\qquad u,v\ge0.
\end{align}
They insert the joining step $\Kop(u,v)-u-v+1$ between the two child schedules, producing
\begin{align}\label{eq:symmetric-join-main}
 \bigl(a^\top,\ \Kop(u,v)-u-v+1,\ b^\top\bigr)^\top,
\end{align}
whose scalar value is $\Kop(u,v)$. Thus, on any fixed composition tree, ConPP and OBS-S generate the same numerical schedule. Appendix~\ref{app:certificates-framework} gives the underlying certificate formulations and the change of variables establishing this correspondence.

Each internal node corresponds to one join and therefore inserts exactly one gradient step. Since a full binary tree with $N$ leaves has $N-1$ internal nodes, an ordered $N$-leaf composition tree represents an $(N-1)$-step schedule, matching the convention $N=n+1$ introduced in Section \ref{sec:introduction}. Starting from value $1$ at every leaf, recursively applying $\Kop$ determines the scalar value at each internal node; together with \eqref{eq:symmetric-join-main}, these values determine the inserted steps and hence the resulting numerical schedule. For example, joining two leaves gives the one-step schedule $(\sqrt2)$, while joining two copies of this block gives $(\sqrt2,2,\sqrt2)$. Maximizing the common scalar value within the symmetric recursive families is therefore equivalent to maximizing the root value over $N$-leaf composition trees.

\paragraph{The Bellman problem.}
Let $U_N$ be the largest root value over all composition trees with $N$ leaves. Since $\Kop$ is strictly increasing in each argument, the two subtrees of an optimal tree may themselves be chosen optimally. Hence
\begin{align}\label{eq:U-DP}
 U_1=1,\qquad U_N=\max_{1\le m<N}\Kop(U_m,U_{N-m})\quad(N\ge2).
\end{align}
Here $m$ and $N-m$ are the leaf counts of the two child blocks, while the root contributes their joining step. Equation \ref{eq:U-DP} is the scalar optimization problem underlying both symmetric recursive families.

\paragraph{Exact objective interpretation.}
Let $H_N$ denote the maximum total stepsize in the $(N-1)$-step ConPP family, and let $s_N$ denote the optimal OBS-S rate at the same horizon. Their relation to the common scalar value is
\begin{align}\label{eq:common-objective-coefficient}
 U_N=H_N+1=\frac1{s_N}.
\end{align}
Since the tight objective coefficients for ConPP and OBS-S are $1/(2H_N+1)$ and $s_N/(2-s_N)$, respectively, \eqref{eq:common-objective-coefficient} shows that both families have the same tight objective coefficient
\begin{align*}
    B_N\coloneqq\frac1{2U_N-1},
\end{align*}
and their optimized $(N-1)$-step schedules satisfy
\begin{align}\label{eq:objective-bound}
 f(x_{N-1})-f_*\le B_N\frac L2\norm{x_0-x_*}^2.
\end{align}
A one-dimensional Huber instance attains equality in this bound for each family; the corresponding certificate derivations are collected in Appendix~\ref{app:certificates-framework}. Thus $B_N$ is their exact worst-case objective coefficient, and maximizing $U_N$ is exactly equivalent to minimizing the worst-case objective gap.


It remains, therefore, to understand the maximization defining $U_N$ in \eqref{eq:U-DP}. For this structural question, the key properties of $\Kop$ are symmetry, positive homogeneity, and strict coordinatewise monotonicity, together with
\begin{align*}
    \Kop(0,u)=u,\qquad \Kop(u,u)=\rho u, \qquad \rho=1+\sqrt{2}.
\end{align*}
The central question is whether the balanced split always attains the maximum in \eqref{eq:U-DP}, and when other splits can tie.
\section{Balanced optimality and all optimal compositions trees}\label{sec:balanced}


The Bellman recursion \eqref{eq:U-DP} optimizes over every root split, whereas the Silver construction uses equal children at dyadic horizons. The next theorem shows that the same balancing principle governs every horizon: splitting the leaf count as evenly as possible always attains the optimum. It further characterizes exactly when an unbalanced split attains the same value.


\begin{theorem}\label{thm:balanced-main}
Extend the sequence by $U_0=0$. For every $N\ge2$,
\begin{align}\label{eq:balanced-recurrence}
 U_N=\Kop\bigl(U_{\lfloor N/2\rfloor},U_{\lceil N/2\rceil}\bigr).
\end{align}
Moreover, $(U_N)_{N\ge0}$ is strictly increasing and satisfies the supercomposition inequality
\begin{align*}
    \Kop(U_m,U_n)\le U_{m+n}, \qquad m,n\ge0.
\end{align*}
For positive integers $m,n$, equality holds exactly when
\begin{align}\label{eq:main-equality}
 \Kop(U_m,U_n)=U_{m+n}
 \quad\Longleftrightarrow\quad
 \begin{cases}
 m=n,\quad\text{or}\\
 \{m,n\}=\{2^kr,2^k(r+1)\}\text{ for some }k\ge0,\ r\ge1.
 \end{cases}
\end{align}
\end{theorem}

This resolves Conjecture~4.17 of \citet{ZhangJiang} and completely characterizes the maximizing root splits. In particular, every unbalanced tie is a dyadic rescaling of two consecutive leaf counts. For example, at $N=6$, the balanced split $3+3$ ties with the two ordered splits $2+4$ and $4+2$. If $N$ is a power of two, the equal split is the unique maximizing root split; recursively, the completely balanced ordered tree is therefore the unique optimal composition tree. More generally, an ordered composition tree is optimal if and only if every internal split satisfies one of the two patterns in \eqref{eq:main-equality}.

Appendix~\ref{app:optimal-structure} gives the complete root-split classification, translates it to the original step-count notation of \citet{ZhangJiang}, and counts the resulting optimal ordered trees; distinct optimal trees may still generate the same numerical schedule.

\subsection{A four-point rearrangement principle}\label{sec:four-point-rearrangement}
The obstacle is to compare compositions with different tree shapes, something symmetry alone does not do. We overcome it by examining two levels of a tree at once: four subtrees can be paired in three ways before their values are combined at the parent. The following inequality orders these three possibilities and, crucially, retains the equality information.
For $a,b,c,d\ge0$, write $Z_{ab\mid cd}=\Kop(\Kop(a,b),\Kop(c,d))$.

\begin{theorem}\label{thm:four-point-rearrangement}
For $0\le a\le b\le c\le d$,
\begin{align}\label{eq:four-point-chain}
 Z_{ab\mid cd}\le Z_{ac\mid bd}\le Z_{ad\mid bc}.
\end{align}
The equality cases are
\begin{align}
 Z_{ab\mid cd}=Z_{ac\mid bd}&\iff b=c\ \text{or}\ ad=bc,\label{eq:first-equality-case}\\
 Z_{ac\mid bd}=Z_{ad\mid bc}&\iff a=b\ \text{or}\ c=d.\label{eq:second-equality-case}
\end{align}
\end{theorem}

The proof exploits the quadratic identity
\begin{align*}
    \Kop(x,y)^2-(x+y)\Kop(x,y)-xy=0.
\end{align*}
For $u=\Kop(a,b)$ and $v=\Kop(c,d)$, eliminating the inner roots reduces the first comparison to the sign of
\begin{align*}
    abv^2+cdu^2-(ad+bc)uv=(av-cu)(bv-du).
\end{align*}
Homogeneity gives
\begin{align*}
    av=\Kop(ac,ad),\qquad cu=\Kop(ac,bc),
\end{align*}
and similarly
\begin{align*}
    bv=\Kop(bc,bd),\qquad du=\Kop(ad,bd).
\end{align*}
Monotonicity therefore makes the two factors have opposite signs. The same factorization identifies the equality case $ad=bc$, while $b=c$ corresponds to coincident pairings. The second comparison follows analogously. Appendix~\ref{app:four-point-algebra} gives the complete elimination and root-sign argument, including boundary cases. This turns a comparison of nested compositions into a local exchange rule that can be propagated through the tree.

\paragraph{From a local comparison to global optimality.}
To propagate the four-point rearrangement through arbitrary trees, define the balanced candidate sequence
\begin{align*}
    V_0=0,\qquad V_1=1, \qquad V_N=\Kop(V_{\lfloor N/2\rfloor},V_{\lceil N/2\rceil}).
\end{align*}
Monotonicity of $\Kop$ makes $(V_N)$ strictly increasing. The key inductive invariant is the supercomposition inequality $ \Kop(V_m,V_n)\le V_{m+n}$. 
Once this inequality holds, every Bellman split is bounded by $V_N$, while the balanced split attains $V_N$; hence $U_N=V_N$. 

For example, in the even–odd case $1\le a\le b$,
\begin{align*}
     \Kop(V_{2a},V_{2b+1})
 &=\Kop\bigl(\Kop(V_a,V_a),\Kop(V_b,V_{b+1})\bigr)\\
 &\le\Kop\bigl(\Kop(V_a,V_b),\Kop(V_a,V_{b+1})\bigr)\qquad \text{(rearrangement)}\\
 &\le\Kop(V_{a+b},V_{a+b+1}) \qquad \qquad\qquad\,\text{(induction)}\\&=V_{2a+2b+1}.
\end{align*}
The remaining parity cases use the same two ingredients: homogeneity handles even–even indices, while the full rearrangement chain handles unequal odd indices. Thus the supercomposition inequality holds for all $m,n$, and the preceding argument gives $U_N=V_N$ for every $N$.

Tracking equality through the same parity recursion yields exactly the dyadic-consecutive pattern in \eqref{eq:main-equality}: common powers of two peel off, mixed parity permits equality only for consecutive indices, and unequal odd indices are strict. Strict monotonicity then forces both children of every optimal node to be optimal for their respective sizes, giving the complete tree characterization in Theorem \ref{thm:balanced-main}. Appendix~\ref{app:abstract-balanced} gives the abstract balanced-composition induction, and Appendix~\ref{app:optimal-structure} completes the equality and optimal-tree classification.

The balanced recurrence therefore replaces the exhaustive Bellman search by a deterministic binary structure, which we exploit next to evaluate $U_N$ without tabulating all smaller horizons.

\subsection{Binary recursion and single-horizon evaluation}
\label{subsec:binary_recursion}



Theorem \ref{thm:balanced-main} reduces the exhaustive Bellman search to the balanced recurrence, so the table $U_1,\ldots,U_N$ can be computed in $O(N)$ scalar operations. For a single prescribed horizon, the recurrence has a stronger binary structure. Indeed,
\begin{align*}
    (U_{2r},U_{2r+1})=(\rho U_r,\Kop(U_r,U_{r+1})),  \quad (U_{2r+1},U_{2r+2})=(\Kop(U_r,U_{r+1}),\rho U_{r+1}).
\end{align*}
Thus an adjacent pair $(U_r,U_{r+1})$ determines the pair associated with either binary extension $r\mapsto2r$ or $r\mapsto2r+1$. Starting from $(U_0,U_{1})=(0,1)$, read the binary digits of $N$ from left to right. After a prefix representing $r$ has been read, maintain $(a,b)=(U_r,U_{r+1})$. A next bit $0$ sends this pair to $(U_{2r},U_{2r+1})$, while a next bit $1$ sends it to $(U_{2r+1},U_{2r+2})$. This gives the following evaluator.
\begin{algorithm}[H]
\caption{Binary evaluation of $U_N$}
\label{alg:bitwise}
\DontPrintSemicolon
\SetAlgoNoLine

\KwIn{An integer $N\ge 1$, read in binary from the most significant bit.}

$(a,b)\leftarrow(0,1)$\;

\For{each binary digit $\xi$ of $N$}{
    $c\leftarrow\Kop(a,b)$\;
    \lIf{$\xi=0$}{$(a,b)\leftarrow(\rho a,c)$}
    \lElse{$(a,b)\leftarrow(c,\rho b)$}
}
\Return{$a$}\;
\end{algorithm}
Each binary digit requires one evaluation of $\Kop$ and one multiplication by $\rho$. Hence Algorithm \ref{alg:bitwise} evaluates a single optimized scalar value $U_N$ in $O(\log N)$ scalar operations using $O(1)$ scalar working storage. Constructing the full $(N-1)$-step schedule requires linear output work. Appendix~\ref{app:algorithm-proof} proves the adjacent-pair invariant and the operation count.

The same binary recursion also reveals the exact scaling $U_{2r}=\rho U_r$. Its iteration is the dyadic self-similarity underlying the phase law developed next.
\section{Exact dyadic phase law for the symmetric family}\label{sec:phase}


The exact scaling from Section \ref{subsec:binary_recursion} iterates to
\begin{align}\label{eq:dyadic-scaling}
 U_{2^kN}=\rho^kU_N,\qquad s_{2^kN}=\rho^{-k}s_N,\qquad N\ge1,\ k\ge0.
\end{align}
Since $\rho=2^p$, the normalized value $U_N/N^p$ is invariant under dyadic dilation. Hence its remaining dependence on the horizon is determined entirely by the position of $N$ within its dyadic block. Write $x_N:=N/2^{\lfloor\log_2N\rfloor}\in[1,2)$, $ \theta_N:=\log_2x_N=\{\log_2N\}\in[0,1)$, and identify $\theta_N$ with its class in $\mathbb T:=\mathbb R/\mathbb Z$. Here $x_N$ is the dyadic mantissa, and $\theta_N$ is the corresponding logarithmic phase. We write $\operatorname{Clust}(a_N)$ for the set of subsequential limits of a sequence $\{a_N\}$ along $N\to\infty$. The next theorem shows that all normalized optimal values lie exactly on a single continuous profile of the dyadic phase.

\begin{theorem}\label{thm:symmetric-main}
    There exists a unique positive Lipschitz function $\Phi:\mathbb T\to(0,\infty)$ such that, for every $N\ge1$, 
    \begin{align}\label{eq:exact-phase}
        U_N=N^p\Phi(\theta_N).
    \end{align}
    Moreover, $\Phi$ is nonconstant and satisfies
    \begin{align}\label{eq:phase-extrema}
        0<\phimin\coloneqq\min_{t\in\mathbb T}\Phi(t)<1,
        \qquad \max_{t\in\mathbb T}\Phi(t)=\Phi(0)=1.
    \end{align}
\end{theorem}
Through the identity $B_N=(2U_N-1)^{-1}$, the exact phase law for $U_N$ induces an exact phase law for the tight objective coefficient.
\begin{corollary}
    \label{cor:sysmmetic-main-objective}
    The optimized symmetric schedules satisfy
    \begin{align}
        B_N&=\frac1{2N^p\Phi(\theta_N)-1}.\label{eq:B-exact}
    \end{align}
    Consequently,
    \begin{align}
        B_N&=\frac{N^{-p}}{2\Phi(\theta_N)}+O(N^{-2p}),\label{eq:B-asymptotic}
    \end{align}
    where the remainder is uniform over the dyadic phase, and
    \begin{align}\label{eq:main-objective-cluster}
        \Clust(N^pB_N)=\left[\frac12,\frac1{2\phimin}\right].
    \end{align}
\end{corollary}
Thus the dyadic phase enters the leading $N^{-p}$ coefficient itself. In particular, the normalized tight objective bound has no single asymptotic constant. Two fixed dyadic rays already exhibit distinct limits:
\begin{align*}
 \lim_{k\to\infty}(2^k)^pB_{2^k}=\frac12,
 \qquad
 \lim_{k\to\infty}(3\cdot2^k)^pB_{3\cdot2^k}
 =\frac{3^p}{2U_3}>\frac12.
\end{align*}
Continuity of $\Phi$ then fills the entire interval in \eqref{eq:main-objective-cluster}.

\paragraph{From dyadic self-similarity to a phase profile.}
Equation \ref{eq:dyadic-scaling} collapses all horizons related by powers of two onto the single reference block $[1,2]$. For example, the horizons $3,6,12,\ldots$ all have dyadic mantissa $3/2$ and satisfy $\rho^{-1}U_3=\rho^{-2}U_6=\rho^{-3}U_{12}=\cdots$. More generally, for a dyadic point $x=j/2^k\in[1,2]$, define $F(x):=\rho^{-k}U_j$. The scaling $U_{2j}=\rho U_j$ makes this definition independent of the representation of $x$.

The balanced recurrence is compatible with dyadic refinement. If $x=j/2^k$ and $y=(j+1)/2^k$ are adjacent grid points, their midpoint corresponds to the odd index $2j+1$, and
\begin{align*}
F\!\left(\frac{x+y}{2}\right)
=
\rho^{-1}\Kop(F(x),F(y)).
\end{align*}
Thus the same recurrence recursively fills in the values of $F$ on finer dyadic grids.

\paragraph{From the dyadic grid to a continuous phase.}
The adjacent ratios $R_N:=U_{N+1}/U_N$ satisfy a binary recursion contracting toward $1$, and in particular $R_N-1=O(N^{-1})$ (Lemma~\ref{lem:ratio-contraction}). Lemma~\ref{lem:Lipschitz-F} converts this estimate into uniform upper and lower bounds on the slopes of the dyadic piecewise-linear interpolants. Hence $F$ extends uniquely to a strictly increasing bi-Lipschitz map $F:[1,2]\to[1,\rho]$.

Writing $x=2^t$, we have $x^p=\rho^t$, and therefore $U_N/N^p=F(x_N)/x_N^p$. This gives the phase profile
\begin{align*}
\Phi(t):=\rho^{-t}F(2^t), \qquad 0\le t\le 1.
\end{align*}
Since $F(1)=1$ and $F(2)=\rho$, we have $\Phi(0)=\Phi(1)=1$, so $\Phi$ extends continuously and periodically to $\mathbb T$.

\paragraph{Why the phase is nonconstant.}
Let $q=1/p$. The sharp kernel inequality $\Kop(x,y)\le (x^q+y^q)^{1/q}$ implies $U_N\le N^p$, with equality exactly at powers of two. Hence $\Phi(0)=1$, whereas $\Phi(\log_2(3/2))=U_3/3^p<1$. Thus $\Phi$ is nonconstant. Uniqueness follows from the density of the integer phases in $\mathbb T$.

This proves Theorem~\ref{thm:symmetric-main}. Corollary~\ref{cor:sysmmetic-main-objective} follows from $B_N=(2U_N-1)^{-1}$; Appendix~\ref{app:objective-phase-proof} gives the uniform expansion and the complete cluster interval.
\section{Limiting Dyadic phase of the asymmetric Recursion}\label{sec:obsf}

The ConPD/OBS-F recursion has a different geometry. Each join combines an optimized symmetric prefix with an objective-specific continuation, and the optimal split may be unbalanced. This breaks the exact doubling identity that underpins the symmetric phase law in Section~\ref{sec:phase}. Nevertheless, after normalization by $N^p$, the tight objective coefficients converge uniformly across dyadic blocks to a nonconstant phase profile. The proof proceeds by comparing the ordered recursion with its symmetric counterpart and then analyzing the equality structure of a sharp homogeneous support bound.

\subsection{Ordered Bellman Recursion and Symmetric Comparison}\label{subsec:obsf-bellman}
Let $\eta_N^F$ denote the tight objective-gap coefficient of the optimized
ConPD/OBS-F schedule with $N-1$ steps, and set
$W_N := (\eta_N^F)^{-1}$.
The composition rules underlying ConPD and OBS-F yield the ordered Bellman
recursion
\begin{align}\label{eq:obsf-reciprocal-dp}
 W_1=1,\qquad W_N=\max_{1\le m<N}\mathcal A(U_m,W_{N-m}),
\end{align}
where $\mathcal A(u,w)=(4u+w+\sqrt{w^2+8uw})/2$. Here $U_m$ represents an optimized symmetric prefix with $m$ leaves,
whereas $W_{N-m}$ represents the optimized objective-specific continuation.
The two arguments therefore play different roles: the recursion is ordered
rather than symmetric, placing it outside the four-point rearrangement
mechanism of Section~\ref{sec:balanced}. Appendix~\ref{app:asymmetric-interface} gives the certificate derivation and the
correspondence between the ConPD and OBS-F constructions.

The corresponding optimized schedule satisfies the tight objective guarantee
\begin{align}\label{eq:obsf-objective-bound}
 f(x_{N-1})-f_*\le\eta_N^{\mathrm F}\frac L2\norm{x_0-x_*}^2.
\end{align}
Matching one-dimensional quadratic and Huber instances attain equality, so
$\eta_N^F$ is the exact worst-case objective coefficient for the optimized
recursive schedule; see Appendix~\ref{app:asymmetric-interface}.

Despite their different kernels, the symmetric and asymmetric recursions remain
quantitatively coupled. Local comparisons between $\mathcal{A}$ and $\Kop$, propagated
through the Bellman recursions, give
\begin{align}\label{eq:obsf-sandwich}
 2U_N-1\le W_N\le\rho U_N.
\end{align}
Equivalently, $1/(\rho U_N) \le \eta_N^F \le 1/(2U_N-1)=B_N$.
Hence the asymmetric objective coefficient is at most the symmetric one at every
horizon. Together with the exact symmetric phase law
$U_N=N^p\Phi(\theta_N)$, \eqref{eq:obsf-sandwich} gives
$\eta_N^F=\Theta(N^{-p})$. The exponent is therefore fixed; what remains is
the normalized coefficient $C_N:=N^p\eta_N^F$. We next show that $C_N$
retains a nonconstant dyadic-phase dependence, converging uniformly over dyadic
blocks to a limiting profile.

\subsection{Uniform phase limit and strict support  separation}\label{subsec:obsf-phase-construction}
For each dyadic block $2^k\le N\le2^{k+1}$, let $\Psi_k:[0,1]\to\mathbb R$ be the piecewise-affine interpolation of the points $(\log_2(N/2^k),C_N)$. Thus every dyadic block is represented by a curve on the same phase domain. In the symmetric setting, exact doubling makes the corresponding grid values consistent across scales. Here the interpolants vary with $k$, and the central question is whether they converge to a common profile.

A second ingredient is the sharp homogeneous support bound for the asymmetric kernel. With $q=1/p$, define
\begin{align*}
 B_{\mathrm{sup}}\coloneqq\sup_{u>0}
 \frac{\mathcal A(u,1)^q-1}{u^q},\qquad
 c_*\coloneqq B_{\mathrm{sup}}^{-p}.
\end{align*}
By homogeneity, $\mathcal A(u,w)^q\le B_{\rm sup}u^q+w^q$, and combining this with the symmetric power bound $U_m^q\le m$ gives $C_N\ge c_*$. The key question is whether an optimizing recursion can asymptotically attain this local barrier.

\begin{theorem}\label{thm:asymmetric-main}
There is a unique positive Lipschitz profile $\Psi_F:\mathbb T\to(0,\infty)$ such that
$\|\Psi_k-\Psi_F\|_\infty\to0$. Equivalently,
\begin{align}\label{eq:obsf-uniform-phase}
 \sup_{2^k\le N\le2^{k+1}}|N^p\eta_N^{\mathrm F}-\Psi_F(\theta_N)|\longrightarrow0.
\end{align}
Hence $\eta_N^{\mathrm F}=N^{-p}\Psi_F(\theta_N)+o(N^{-p})$, with this dyadic-block uniformity, and
\begin{align}\label{eq:obsf-cluster-set}
 \Clust(N^p\eta_N^{\mathrm F})=[\min\Psi_F,\max\Psi_F].
\end{align}
The profile is nonconstant and satisfies
\begin{align}\label{eq:obsf-final-enclosure}
 c_*<\min\Psi_F<\max\Psi_F\le c_*/\phimin.
\end{align}
\end{theorem}


Equation \ref{eq:obsf-uniform-phase} is the asymmetric counterpart of the exact phase law in \eqref{eq:exact-phase}: every dyadic block approaches the same limiting profile uniformly, although the optimizing splits may vary with the horizon. Thus both optimized recursive families share the Silver exponent $p$, while their leading objective coefficients retain nonconstant dyadic-phase dependence. Moreover, the strict inequality $c_*<\min\Psi_F$ shows that global recursive compatibility enforces a positive gap above the sharp local support barrier. Figure \ref{fig:phases} visualizes the two objective-coefficient phases.


\begin{figure}[!htbp]
 \centering
 \begin{minipage}[t]{0.49\linewidth}
   \centering
   \includegraphics[width=\linewidth]{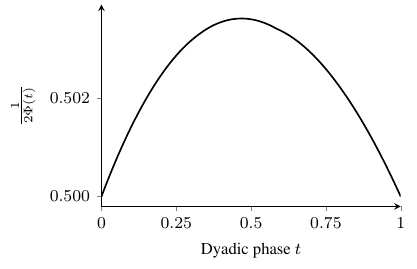}\\[-0.15em]
   {\footnotesize (a) Symmetric: exact phase law}
 \end{minipage}\hfill
 \begin{minipage}[t]{0.49\linewidth}
   \centering
   \includegraphics[width=\linewidth]{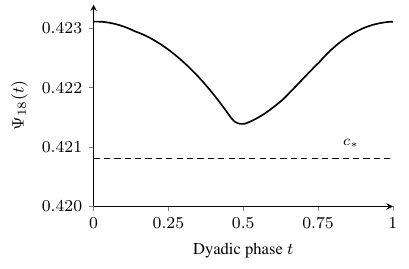}\\[-0.15em]
   {\footnotesize (b) Asymmetric: finite-scale interpolant}
 \end{minipage}
 \caption{Dyadic phase dependence of the objective-gap leading coefficients.
 Left: the symmetric leading-coefficient phase $1/(2\Phi(t))$, given exactly by Corollary \ref{cor:sysmmetic-main-objective}.
 Right: the finite-scale asymmetric interpolant $\Psi_{18}(t)$; Theorem~\ref{thm:asymmetric-main} establishes the uniform limit $\Psi_k\to\Psi_{F}$.
 The dashed line marks the sharp support barrier $c_*$.
 }
 \label{fig:phases}
\end{figure}

\textbf{Proof sketch.} \emph{Uniform phase limit.}
For each fixed $m$, the inequality $\eta^F_{2N}\le \rho^{-1}\eta^F_N$ makes
$C_{m2^k}=(m2^k)^p\eta^F_{m2^k}$ nonincreasing in $k$, while \eqref{eq:obsf-sandwich} and
$U_N\le N^p$ give $C_{m2^k}\ge 1/\rho$. Hence every fixed dyadic ray has a
positive limit. Moreover, $|C_{N+1}-C_N|=O(N^{-1})$, while adjacent logarithmic
phases are separated by $\Theta(N^{-1})$, so the dyadic-block interpolants
$\{\Psi_k\}$ are equi-Lipschitz. Thus every subsequence has a uniformly
convergent subsequence. Any such limit is already determined on the dense set
of dyadic phases by the raywise limits, so all subsequential limits coincide.
Therefore $\Psi_k\to\Psi_F$ uniformly.

\emph{Strict separation from $c_*$.}
We next show that $\Psi_F$ stays strictly above the support barrier
$c_*=B_{\rm sup}^{-p}$. Suppose instead that $\Psi_F(t_0)=c_*$. Choose
$N_k\to\infty$ with $\{\log_2 N_k\}\to t_0$. Uniform phase convergence gives
$N_k^p\eta^F_{N_k}\to c_*$, equivalently
$E_{N_k}:=B_{\rm sup}N_k-W_{N_k}^q=o(N_k)$.

For every optimal split $N=m+r$,
\begin{align*}
E_N
&=
B_{\rm sup}(m-U_m^q)+E_r
+\Bigl(B_{\rm sup}U_m^q+W_r^q-\mathcal A(U_m,W_r)^q\Bigr),
\end{align*}
and all three terms are nonnegative. Thus near-contact with $c_*$ requires the
symmetric bound, the continuation bound, and the sharp support inequality to
become asymptotically tight together.

A concentration argument based on this deficit identity shows that near-contact
forces an optimal prefix $m_k$ to satisfy $m_k/N_k\to\alpha_*$, where
$\alpha_*\in(1/2,3/4)$ is the unique fraction selected by equality in the sharp
support inequality. The same near-contact also forces the symmetric bound to be
asymptotically tight: $U_{m_k}^q/m_k\to1$. By the exact symmetric phase law,
\begin{align*}
\frac{U_{m_k}^q}{m_k}
&=
\Phi(\{\log_2 m_k\})^q,
\end{align*}
and $\Phi(t)=1$ only at phase $0$. Hence $\{\log_2 m_k\}\to0$. Since
$\log_2 m_k=\log_2 N_k+\log_2(m_k/N_k)$, while
$\{\log_2 N_k\}\to t_0$ and $m_k/N_k\to\alpha_*$, we obtain
\begin{align*}
t_0+\log_2\alpha_*
&=0 \pmod{1}.
\end{align*}
The continuation satisfies $r_k/N_k\to1-\alpha_*$ and still has vanishing
relative deficit. Its phase therefore tends to
$t_1=t_0+\log_2(1-\alpha_*)\pmod{1}$. Applying the same argument to this
continuation gives
\begin{align*}
t_1+\log_2\alpha_*
&=0 \pmod{1}.
\end{align*}
Subtracting the two phase identities yields
\begin{align*}
\log_2(1-\alpha_*)
\in\mathbb Z,
\qquad
1/4<1-\alpha_*<1/2,
\end{align*}
a contradiction. Therefore $\Psi_F(t)>c_*$ for every $t\in\mathbb T$.

\emph{Nonconstancy and upper enclosure.}
Finally, if $\Psi_F$ were constant, the limiting Bellman inequality would become
scale-invariant, and the sharp support geometry would force the only possible
constant to be $c_*$, contradicting the strict separation above. The upper
enclosure $\Psi_F\le c_*/\phi_*$ follows from the same limiting Bellman
inequality together with the symmetric lower bound $U_m/m^p\ge\phi_*$.

Thus the asymmetric recursion changes the leading objective coefficient while retaining a nonconstant dyadic dependence.
\section{Discussion and conclusion}\label{sec:conclusion}

Optimized recursive long-step schedules therefore admit a sharper description than the usual $O(N^{-p})$ rate. For the symmetric ConPP/OBS-S family, the four-point rearrangement principle solves the Bellman split problem at every horizon: balanced splitting is optimal, all optimal ties are characterized, and the resulting recursion yields exact dyadic self-similarity together with an $O(\log N)$ evaluator. This self-similarity induces an exact log-periodic law for the tight objective coefficient, in which the Silver exponent $p=\log_2(1+\sqrt{2})$ determines the power-law envelope while a nonconstant phase function captures the leading dependence on the dyadic position of $N$. For the asymmetric ConPD/OBS-F family, horizon-dependent unbalanced optimizers produce a distinct limiting phase profile, while the same Silver exponent persists and the leading coefficient remains strictly above the sharp local support barrier.

Taken together, these results show that the recursive mechanism underlying accelerated long-step schedules also imprints a persistent dyadic structure on their convergence guarantees: the exponent determines the scale, while the phase determines the leading finite-horizon coefficient. This exponent–phase description sharpens the usual rate-based view of predetermined stepsize acceleration and suggests a broader question of whether analogous phase phenomena arise beyond recursive constructions. Extending this viewpoint to minimax-optimal predetermined schedules, and resolving the exact geometry of the asymmetric phase together with the asymptotic structure of its optimizing splits, are natural directions for future work.

\newpage
\section*{AI use statement}

The authors identified the phase-dependent variation of the leading coefficients that motivated this work, formulated the underlying mathematical questions and structural ideas, and led the development and presentation of the results. The authors substantially revised, simplified, reorganized, and verified the proofs, and carefully refined the manuscript. Generative AI tools, including ChatGPT, Codex, and AI-assisted features in Cursor, were used under the authors’ direction to propose candidate proof strategies and proof arguments, assist with Lean 4 tactics and declarations, and help implement numerical checking code. AI-generated mathematical arguments were critically examined, modified where needed, and independently verified by the authors before inclusion. All submitted Lean and numerical code was reviewed and executed by the authors. The authors take responsibility for the final text, mathematical claims, and submitted artifacts.


\section*{Reproducibility statement}

The paper is primarily theoretical. The
\hyperref[app:proof-roadmap]{appendix proof roadmap} locates the complete
mathematical proofs, and Appendix~\ref{app:certificates-framework} states the
imported certificate interfaces and prior results. The supplementary materials, which are publicly available at
\url{https://github.com/ck14527/GD_stepsize_supplementary_materials}, 
provide the scalar evaluator, reference Bellman recursions, figure data, and
reproduction commands. \texttt{NUMERICAL\_REPRODUCIBILITY.md} distinguishes
exploratory floating-point computations from rigorous interval-arithmetic
certificates. The phase-minimum and dyadic-corner certificates are specified in
Appendices~\ref{app:certificates} and~\ref{app:proof-strict-corner},
respectively, and the reported level-17 phase-minimum enclosure is verified by
interval arithmetic. The supplementary materials also include Lean~4 proofs of the scalar and
combinatorial results identified in \texttt{lean/docs/THEOREMS.md}.
\texttt{LEAN\_FORMALIZATION.md} documents the pinned environment,
verification commands, and formalization scope. The smooth-convex
gradient-descent certificate interfaces and tightness results imported from
prior work are documented in Appendix~\ref{app:certificates-framework} and
remain outside the Lean formalization.

\bibliographystyle{plainnat} 
\bibliography{references}

@article{AltschulerParriloJACM,
  author    = {Altschuler, J. M. and Parrilo, P. A.},
  title     = {Acceleration by stepsize hedging: Multi-step descent and the {Silver} Stepsize Schedule},
  journal   = {Journal of the ACM},
  volume    = {72},
  number    = {2},
  pages     = {Article 12},
  year      = {2025},
  doi       = {10.1145/3708502}
}

@article{AltschulerParriloMP,
  author    = {Altschuler, J. M. and Parrilo, P. A.},
  title     = {Acceleration by stepsize hedging: {Silver} Stepsize Schedule for smooth convex optimization},
  journal   = {Mathematical Programming},
  volume    = {213},
  pages     = {1105--1118},
  year      = {2025},
  doi       = {10.1007/s10107-024-02164-2}
}

@article{DasGuptaVanParysRyu,
  author    = {Das Gupta, S. and Van Parys, B. P. G. and Ryu, E. K.},
  title     = {Branch-and-bound performance estimation programming: A unified methodology for constructing optimal optimization methods},
  journal   = {Mathematical Programming},
  volume    = {204},
  pages     = {567--639},
  year      = {2024},
  doi       = {10.1007/s10107-023-01973-1}
}

@article{DroriTeboulle,
  author    = {Drori, Y. and Teboulle, M.},
  title     = {Performance of first-order methods for smooth convex minimization: A novel approach},
  journal   = {Mathematical Programming},
  volume    = {145},
  pages     = {451--482},
  year      = {2014},
  doi       = {10.1007/s10107-013-0653-0}
}

@misc{ye2026silverratealmostoptimal,
      title={Silver Rate Is (Almost) Optimal for Gradient Descent Acceleration}, 
      author={Yuhan Ye and Kaizhao Liu},
      year={2026},
      eprint={2609.09152},
      archivePrefix={arXiv},
      primaryClass={math.OC},
      url={https://arxiv.org/abs/2609.09152}, 
}

@inproceedings{KornowskiShamir,
  author    = {Kornowski, G. and Shamir, O.},
  title     = {Open Problem: Anytime Convergence Rate of Gradient Descent},
  booktitle = {Proceedings of the 37th Conference on Learning Theory},
  series    = {Proceedings of Machine Learning Research},
  volume    = {247},
  pages     = {5335--5339},
  year      = {2024},
  publisher = {PMLR},
  url       = {https://proceedings.mlr.press/v247/kornowski24a.html}
}

@article{Grimmer2024,
  author    = {Grimmer, B.},
  title     = {Provably faster gradient descent via long steps},
  journal   = {SIAM Journal on Optimization},
  volume    = {34},
  number    = {3},
  pages     = {2588--2608},
  year      = {2024},
  doi       = {10.1137/23M1588408}
}

@misc{GrimmerShuWang2023,
  author    = {Grimmer, B. and Shu, K. and Wang, A. L.},
  title     = {Accelerated gradient descent via long steps},
  year      = {2023},
  note      = {arXiv preprint arXiv:2309.09961},
  url       = {https://arxiv.org/abs/2309.09961}
}

@article{GrimmerShuWangIJO,
  author    = {Grimmer, B. and Shu, K. and Wang, A. L.},
  title     = {Accelerated objective gap and gradient norm convergence for gradient descent via long steps},
  journal   = {INFORMS Journal on Optimization},
  volume    = {7},
  pages     = {156--169},
  year      = {2025},
  doi       = {10.1287/ijoo.2024.0057}
}

@article{GrimmerShuWangMOR,
  author    = {Grimmer, B. and Shu, K. and Wang, A. L.},
  title     = {Composing optimized stepsize schedules for gradient descent},
  journal   = {Mathematics of Operations Research},
  year      = {2025},
  note      = {Published online 4 November 2025},
  doi       = {10.1287/moor.2024.0764}
}

@article{TaylorHendrickxGlineur,
  author    = {Taylor, A. B. and Hendrickx, J. M. and Glineur, F.},
  title     = {Smooth strongly convex interpolation and exact worst-case performance of first-order methods},
  journal   = {Mathematical Programming},
  volume    = {161},
  pages     = {307--345},
  year      = {2017},
  doi       = {10.1007/s10107-016-1009-3}
}

@article{TeboulleVaisbourd,
  author    = {Teboulle, M. and Vaisbourd, Y.},
  title     = {An elementary approach to tight worst case complexity analysis of gradient based methods},
  journal   = {Mathematical Programming},
  volume    = {201},
  pages     = {63--96},
  year      = {2023},
  doi       = {10.1007/s10107-022-01899-0}
}

@misc{TsaiEtAl,
  author    = {Tsai, C.-E. and Fatkhullin, I. and Zhang, L. and He, N.},
  title     = {Lower bounds for anytime acceleration of gradient descent},
  year      = {2026},
  eprint    = {2607.02053},
  archivePrefix = {arXiv},
  url       = {https://arxiv.org/abs/2607.02053}
}

@inproceedings{ZhangLeeDuChen,
  author    = {Zhang, Z. and Lee, J. D. and Du, S. S. and Chen, Y.},
  title     = {Anytime acceleration of gradient descent},
  booktitle = {Proceedings of COLT 2025},
  series    = {Proceedings of Machine Learning Research},
  volume    = {291},
  pages     = {5991--6013},
  year      = {2025},
  publisher = {PMLR},
  url       = {https://proceedings.mlr.press/v291/zhang25a.html}
}

@article{ZhangJiang,
  author    = {Zhang, Z. and Jiang, R.},
  title     = {Accelerated gradient descent by concatenation of stepsize schedules},
  journal   = {SIAM Journal on Optimization},
  volume    = {36},
  number    = {2},
  pages     = {1182--1210},
  year      = {2026},
  doi       = {10.1137/25M173898X}
}

\appendix
\section*{Appendix: proof roadmap}\phantomsection\label{app:proof-roadmap}
The main text states the structural and phase results and explains their proof mechanisms.
The appendices retain all supporting certificate definitions, intermediate results, and complete proofs.
\begin{center}
\begin{tabular}{p{0.43\linewidth}p{0.49\linewidth}}
\toprule
Main-text material & Supporting material and proof ingredients\\
\midrule
Section~\ref{sec:scalar}: schedules and scalar values &
Appendix~\ref{app:certificates-framework}: primitive and s-composable certificates, joins, and objective guarantees.\\
Theorem~\ref{thm:balanced-main}: optimal compositions &
Appendices~\ref{app:abstract-balanced} and~\ref{app:optimal-structure}: propagation, equality, and optimal trees.\\
Theorem~\ref{thm:four-point-rearrangement}: rearrangement &
Appendix~\ref{app:four-point-algebra}: sign identities and algebraic comparison.\\
Algorithm~\ref{alg:bitwise}: one-horizon evaluation &
Appendix~\ref{app:algorithm-proof}: adjacent-pair invariant and operation count.\\
Theorem~\ref{thm:symmetric-main}: exact symmetric phase &
Appendices~\ref{app:silver-power}--\ref{app:exact-symmetric-phase}: power envelope, ratio contraction, and canonical interpolation.\\
Corollary~\ref{cor:sysmmetic-main-objective}: tight objective coefficient &
Appendix~\ref{app:objective-phase-proof}: exact conversion, uniform expansion, and cluster set.\\
Theorem~\ref{thm:asymmetric-main}: asymmetric phase &
Appendix~\ref{app:obsf-details}: sandwich, compactness, support deficit, and phase rigidity.\\
\bottomrule
\end{tabular}
\end{center}

\section{Certificate frameworks and scalar recursions}\label{app:certificates-framework}
The composition theorems used in the main text are due to \citet{ZhangJiang} and
\citet{GrimmerShuWangMOR}. We record their certificate interfaces, the changes of
variables, and the resulting connection to the gradient-descent guarantees.
Throughout this appendix, $L=1$; applying the dimensionless schedules through
\eqref{eq:gd} restores the factor $L$.

\subsection{Certificate interfaces and recursive schedule classes}

Let $\F_1$ be the class of differentiable convex functions with $1$-Lipschitz gradient and a nonempty minimizer set. For a schedule $h=(h_0,\ldots,h_{n-1})^\top$ and the associated gradient-descent trajectory, write $f_i=f(x_i)$ and $g_i=\nabla f(x_i)$, and fix a minimizer $x_*$ with $f_*=f(x_*)$. Smooth convex interpolation \citep{TaylorHendrickxGlineur} gives
\begin{align*}
    Q_{i,*}\coloneqq f_i-f_*-\langle g_i,x_i-x_*\rangle+\frac12\norm{g_i}^2\le0.
\end{align*}

The two frameworks below use stronger inequalities involving the final iterate to certify recursive composition.

\subsubsection{Primitive schedules and ConPP}\label{app:primitive-interface}

We first recall the definition of a primitive schedule from \citet{ZhangJiang}.

\begin{definition}[Definition~2.7 in \citet{ZhangJiang}]\label{def:primitive}
    A schedule $h\in\R_+^n$ is primitive if, with $H\coloneqq\one^\top h$, every $f\in\F_1$ and every associated trajectory satisfy
    \begin{align*}
        \frac12\norm{x_0-x_*}^2+\sum_{i=0}^{n-1}h_iQ_{i,*}-H(f_n-f_*)
        \ge\frac12\norm{x_n-x_*}^2+\frac{H(H+1)}2\norm{g_n}^2.
    \end{align*}
    The empty schedule is primitive with $H=0$.
\end{definition}

This certificate also yields the tight objective bound
\begin{align} \label{eq:primitive-gap-imported}
    f_n-f_*\le\frac{1}{2H+1}\cdot\frac12\norm{x_0-x_*}^2,
\end{align}
with equality on a suitable one-dimensional Huber instance \citep[Theorem~2.6, Proposition~2.8, and Corollary~2.9]{ZhangJiang}. Thus, within this certificate class, maximizing the total stepsize $H$ is exactly equivalent to minimizing the certified objective coefficient.

If $a$ and $b$ are primitive, then their primitive--primitive concatenation
\begin{align*}
    \operatorname{ConPP}(a,b)
    \coloneqq\!\[a^\top,\varphi\!\(\one^\top a,\one^\top b\)\!,b^\top\]^\top
\end{align*}
is primitive, where
\begin{align*}
    \varphi(x,y)
    \coloneqq\frac{-x-y+\sqrt{(x+y+2)^2+4(x+1)(y+1)}}{2}.
\end{align*}
This is the ConPP theorem of \citet[Theorem~3.1]{ZhangJiang}. In particular, both the inserted step and the total stepsize of the joined schedule depend on each input schedule only through its total stepsize. We now record the optimized family in the notation of \citet[Definition~4.1]{ZhangJiang}.

\begin{definition}[Definition~4.1 in \citet{ZhangJiang}]\label{def:optimized-primitive}
    Set $h_\circ^{(0)}=[\,]^\top$. For $n\ge1$, choose
    \begin{align*}
        h_\circ^{(n)}
        \in\argmax_{h\in\mathcal H_\circ^{(n)}}\one^\top h,
        \quad
        \mathcal H_\circ^{(n)}
        \coloneqq
        \left\{
        \operatorname{ConPP}\!\(h_\circ^{(k)},h_\circ^{(n-k-1)}\):
        0\le k<n
        \right\}.
    \end{align*}
    If the maximizer is not unique, any maximizing schedule may be selected.
\end{definition}

Since $x+y+\varphi(x,y)$ is strictly increasing in each argument, this recursion implements the principle of optimality: $h_\circ^{(n)}$ has maximum total stepsize among all $n$-step schedules generated recursively from the empty schedule by ConPP.

\subsubsection{s-composable schedules and OBS-S}\label{app:scomposable-interface}

The s-composable framework of \citet{GrimmerShuWangMOR} uses a rate parameter $\eta$ whose defining certificate is attained by canonical quadratic and Huber instances. Its certificate is symmetric in the sense that it controls the final objective gap, distance, and gradient norm through the same parameter.

\begin{definition}[Definition~3 in \citet{GrimmerShuWangMOR}]\label{def:scomposable}
    A positive schedule $h=(h_0,\ldots,h_{n-1})^\top$ is s-composable with rate $\eta>0$ if for every $f\in\F_1$ and every associated trajectory, it holds that
        \begin{align*}
        \frac{1-\eta}{2}\norm{g_n}^2+\frac{\eta^2}{2}\norm{x_n-x_*}^2
        +(\eta-\eta^2)(f_n-f_*)
        \le\frac{\eta^2}{2}\norm{x_0-x_*}^2,
    \end{align*}
    and
    \begin{align*}
        \eta=\frac{1}{1+\sum_{i=0}^{n-1}h_i}
        =\prod_{i=0}^{n-1}(h_i-1).
    \end{align*}
\end{definition}

The sum--product identity determines $\eta$ from the total stepsize.
The corresponding tight objective bound is
\begin{align}\label{eq:s-gap-imported}
    f_n-f_*
    \le\frac{\eta}{2-\eta}\cdot\frac12\norm{x_0-x_*}^2,
\end{align}
with equality attained by a one-dimensional Huber instance \citep[Proposition~1]{GrimmerShuWangMOR}. Hence minimizing $\eta$ strengthens the certified objective guarantee.
\citet[Proposition~1]{GrimmerShuWangMOR} also gives a stronger inequality, attained by the quadratic instance, that contains an additional negative squared-norm term.

If $a$ and $b$ are s-composable with rates $\alpha$ and $\beta$, respectively, their s-join is
\begin{align*}
    a\bowtie b
    \coloneqq\!\left[a^\top,\mu(\alpha,\beta),b^\top\right]^\top,
    \quad
    \mu(\alpha,\beta)
    \coloneqq1+\frac{\sqrt{\alpha^2+6\alpha\beta+\beta^2}-(\alpha+\beta)}{2\alpha\beta}.
\end{align*}
The joined schedule is s-composable with rate
\begin{align*}
    \Jop(\alpha,\beta)
    \coloneqq\frac{2\alpha\beta}{\alpha+\beta+\sqrt{\alpha^2+6\alpha\beta+\beta^2}}
\end{align*}
by \citet[Theorem~3]{GrimmerShuWangMOR}. Repeated s-joins from the empty schedule generate the basic s-composable family.

\begin{definition}[Optimized Basic s-composable Schedule (OBS-S) in \citet{GrimmerShuWangMOR}]\label{def:obs-s}
    Set $h^{\mathrm{OBS\text{-}S}}(0)=[\,]^\top$ and $\eta^{\mathrm{OBS\text{-}S}}(0)=1$. For $N\ge2$, define
    \begin{align*}
        \eta^{\mathrm{OBS\text{-}S}}(N-1)
        \coloneqq\min_{1\le m<N}
        \Jop\!\left(
        \eta^{\mathrm{OBS\text{-}S}}(m-1),
        \eta^{\mathrm{OBS\text{-}S}}(N-m-1)
        \right).
    \end{align*}
    If $m_N$ is any minimizing index, set
    \begin{align*}
        h^{\mathrm{OBS\text{-}S}}(N-1)
        \coloneqq
        h^{\mathrm{OBS\text{-}S}}(m_N-1)
        \bowtie
        h^{\mathrm{OBS\text{-}S}}(N-m_N-1).
    \end{align*}
\end{definition}

Thus $\eta^{\mathrm{OBS\text{-}S}}(N-1)$ is the minimum rate among all basic s-composable schedules with $N-1$ steps. Different minimizing splits may yield different schedules, but the optimized rate is unambiguous. We next derive the common scalar Bellman recursion underlying the ConPP/OBS-S correspondence.

\subsection{The common scalar Bellman program}\label{app:scalar-identities}

We use the leaf-count convention to describe the structure of schedules. A schedule containing $n$ gradient steps can be represented by a full binary composition tree with $N=n+1$ leaves. Each leaf represents the empty schedule, and every internal node inserts one joining step between the schedules represented by its two children.

Define
\begin{align*}
    H_N\coloneqq\one^\top h_\circ^{(N-1)},
    \quad
    s_N\coloneqq\eta^{\mathrm{OBS\text{-}S}}(N-1).
\end{align*}

The empty schedule gives $H_1=0$ and $s_1=1$. In this notation, the two recursions become
\begin{align*}
    H_N=\max_{1\le m<N}
    \bigl(H_m+H_{N-m}+\varphi(H_m,H_{N-m})\bigr), \quad
    s_N=\min_{1\le m<N}\Jop(s_m,s_{N-m}).
\end{align*}

We now introduce the symmetric composition law
\begin{align*}
    \Kop(x,y)
    \coloneqq\frac{x+y+\sqrt{x^2+6xy+y^2}}{2},
    \quad x,y\ge0.
\end{align*}
Direct substitution gives the reciprocal-affine conjugacies
\begin{align}\label{eq:reciprocal-identities}
    \frac1{\Jop(1/x,1/y)}&=\Kop(x,y), \quad
    x+y+\varphi(x,y)+1=\Kop(x+1,y+1),
\end{align}
and
\begin{align*}
    \mu\!\left(\frac1{x+1},\frac1{y+1}\right)=\varphi(x,y).
\end{align*}
Here the first identity is stated for $x,y>0$, while the remaining two hold for $x,y\ge0$.
The last identity shows that corresponding ConPP and s-join nodes insert the same numerical step, not merely that their optimized scalar values agree.

\begin{proposition}\label{prop:scalar-equivalence}
    For every $N\ge1$,
    \begin{align*}
        U_N\coloneqq H_N+1=\frac1{s_N}.
    \end{align*}
    The common value satisfies
    \begin{align*}
        U_1=1,
        \quad
        U_N=\max_{1\le m<N}\Kop(U_m,U_{N-m})
        \quad(N\ge2).
    \end{align*}
\end{proposition}

\begin{proof}
    The identities in \eqref{eq:reciprocal-identities} show that both $1/s_N$ and $H_N+1$ satisfy the same recursion and initial condition. Induction then gives $H_N+1=1/s_N$ for all integers $N \geq 1$.

\end{proof}

Equivalently, assign value one to each leaf of an ordered full binary tree and apply $\Kop$ at every internal node. Then $U_N$ is the largest root value among all such trees with $N$ leaves. The objective-bound coefficients also coincide:
\begin{align*}
    \frac{1}{2H_N+1}
    =\frac{s_N}{2-s_N}
    =\frac{1}{2U_N-1}.
\end{align*}
Thus maximizing $U_N$ simultaneously optimizes the tight primitive and OBS-S objective coefficients over their respective recursively generated classes. Throughout the paper, this correspondence is used for the symmetric families generated from the empty schedule by ConPP and the s-join.


\subsection{Properties of the composition law}\label{app:kernel-properties}

We record the properties of $\Kop$ needed below.

\begin{proposition}\label{lem:K-properties}
    The map $\Kop:[0,+\infty)^2\to[0,+\infty)$ is continuous, symmetric, positively homogeneous, and strictly increasing in each coordinate. It satisfies
    \begin{align}\label{eq:K-special}
        \Kop(0,x)=x,
        \quad
        \Kop(x,x)=\rho x,
        \quad
        \Kop(x,y)>x+y>\max\{x,y\}
        \quad(x,y>0).
    \end{align}
    Moreover, for $x,y,x',y'>0$,
    \begin{align}\label{eq:ratio-sandwich}
        \min\!\left\{\frac{x'}x,\frac{y'}y\right\}
        \le\frac{\Kop(x',y')}{\Kop(x,y)}
        \le\max\!\left\{\frac{x'}x,\frac{y'}y\right\},
    \end{align}
    and both inequalities are strict if $x'/x\ne y'/y$.
\end{proposition}

\begin{proof}
    The first assertions follow directly from the definition of $\Kop$ in \eqref{eq:Kdef}. Strict monotonicity follows by differentiation in the positive quadrant and continuity at the boundary. The identities in \eqref{eq:K-special} can be verified directly from \eqref{eq:Kdef}. Let $\lambda$ and $\Lambda$ be the minimum and maximum of $x'/x$ and $y'/y$. Then
    \begin{align*}
        (\lambda x,\lambda y)\le(x',y')\le(\Lambda x,\Lambda y)
    \end{align*}
    coordinatewise. Monotonicity and homogeneity give \eqref{eq:ratio-sandwich}, with strictness when $\lambda<\Lambda$.
\end{proof}
\subsection{The asymmetric certificate and rate recursion}\label{app:asymmetric-interface}
\begin{definition}[Definition~1 in \citet{GrimmerShuWangMOR}]\label{def:fcomposable}
A positive $n$-step schedule $h=(h_0,\ldots,h_{n-1})^\top$ is \emph{f-composable with rate $\eta>0$} if, for every $f\in\F_1$, every minimizer $x_*$, and every associated gradient-descent trajectory,
\begin{align*}
    f(x_n)-f_*
    \le\frac\eta2\norm{x_0-x_*}^2,
\end{align*}
and its rate satisfies the sum--product identities
\begin{align}\label{eq:f-rate-identity}
    \eta
    =\frac1{1+2\sum_{i=0}^{n-1}h_i}
    =\prod_{i=0}^{n-1}(h_i-1)^2.
\end{align}
The empty schedule is f-composable with rate one.
\end{definition}
The identities in \eqref{eq:f-rate-identity} balance the matching quadratic and Huber instances in one dimension, both of which attain the bound with equality \citep[Definition~1 and Lemma~1]{GrimmerShuWangMOR}. Thus the rate is the exact worst-case objective coefficient for an f-composable schedule. For a general $L$-smooth convex objective, applying the dimensionless schedule through $x_{i+1}=x_i-(h_i/L)\nabla f(x_i)$ restores a factor $L$ in Definition \ref{def:fcomposable}.
The f-join theorem of \citet[Theorem~1]{GrimmerShuWangMOR} combines an s-composable prefix of rate $\alpha$ with an f-composable continuation of rate $\beta$ into an f-composable schedule of rate
\begin{align*}
    \mathcal T(\alpha,\beta)
    \coloneqq
    \frac{2\alpha\beta}
    {\alpha+4\beta+\sqrt{\alpha^2+8\alpha\beta}}.
\end{align*}
Starting from the empty schedule and iterating these joins generates the basic f-composable family. For $N\ge1$, let $h^{\mathrm{OBS\text{-}F}}(N-1)$ be an $(N-1)$-step member with minimum rate, and write the rate $\eta^{\mathrm{OBS\text{-}F}}(N-1)$ as $\eta^{\mathrm{F}}_N$.

With $n=N-1$ and matched optimal-split choices, the optimized ConPD schedule $h_\bullet^{(n)}$ coincides with $h^{\mathrm{OBS\text{-}F}}(N-1)$ \citep[Remark~5.9]{ZhangJiang}. Their common rate satisfies
\begin{align}
 \label{eq:conpd-obsf-rate}
    \eta_N^{\mathrm F}=\frac{1}{2\one^\top h_\bullet^{(N-1)}+1}.
\end{align}
All results below therefore apply to both constructions.

The principle of optimality gives the ConPD/OBS-F dynamic program
\begin{align}\label{eq:obsf-dp}
    \eta^{\mathrm{F}}_1=1,
    \quad
    \eta^{\mathrm{F}}_N=\min_{1\le m<N}\mathcal T(s_m,\eta^{\mathrm{F}}_{N-m}),
\end{align}
where a minimizing split selects an OBS-S prefix $h^{\mathrm{OBS\text{-}S}}(m-1)$ and an OBS-F continuation $h^{\mathrm{OBS\text{-}F}}(N-m-1)$. By Definition \ref{def:fcomposable}, after restoring a general smoothness constant $L$, this schedule satisfies the tight guarantee
\begin{align*}
    f(x_{N-1})-f_*
    \le \eta^{\mathrm{F}}_N\frac{L}{2}\norm{x_0-x_*}^2.
\end{align*}
Hence all subsequent asymptotic statements about $\eta^{\mathrm{F}}_N$ are statements about the exact worst-case coefficient in \eqref{eq:obsf-objective-bound} for both optimized ConPD and OBS-F schedules. In particular, minimizing \eqref{eq:obsf-dp} directly optimizes the certified final objective gap within the basic f-composable class, equivalently within the recursively generated ConPD family described above.

\section{Rearrangement and balanced optimality}\label{app:four-point-algebra}
The proof of Theorem~\ref{thm:four-point-rearrangement} is algebraic. We first
establish three sign identities and then use them to compare the pairings.
The final subsection propagates the four-point inequality to arbitrary trees.

\subsection{Sign beyond the inner roots}\label{app:sign-beyond-inner-roots}
The proof is based on an algebraic sign certificate. For $x,y\geq0$, define the two conjugate roots
\begin{align*}
    r_{\pm}(x,y) \coloneqq \frac{x+y\pm\sqrt{x^2+6xy+y^2}}{2}.
\end{align*}
They are the roots of
\begin{align*}
    r^2-(x+y)r-xy=0,
\end{align*}
with $r_+(x,y)=\mathcal K(x,y)$ and $r_-(x,y)\leq0$. For the pairing $(a,b)\mid (c,d)$, where $a,b,c,d \geq 0$, define
\begin{align}
    P_{ab\mid cd}(z) \coloneqq \prod_{\sigma,\tau\in\{+,-\}} \Bigl[z^2-\bigl(r_\sigma(a,b)+r_\tau(c,d)\bigr)z-r_\sigma(a,b)r_\tau(c,d)\Bigr].
    \label{eq:conjugate-polynomial}
\end{align}
The four quadratic factors account for every choice of the two inner conjugate roots. In particular, the distinguished nonnegative root associated with this pairing is $Z_{ab\mid cd}$.
The first lemma explains how the sign of the degree-eight polynomial $P_{ab\mid cd}(z)$ recovers the position of this root once $z>\max\{\mathcal K(a,b),\mathcal K(c,d)\}$.

\begin{lemma} \label{lem:sign-beyond-inner-roots}
Let $a,b,c,d\geq0$ and let $P_{ab\mid cd}$ be defined in \eqref{eq:conjugate-polynomial}. Define
\begin{align*}
    M_{ab\mid cd} \coloneqq \max\{\mathcal K(a,b),\mathcal K(c,d)\}.
\end{align*}
Then, for every $z>M_{ab\mid cd}$,
\begin{align*}
    \operatorname{sgn}P_{ab\mid cd}(z) = \operatorname{sgn}\bigl(z-Z_{ab\mid cd}\bigr).
\end{align*}
\end{lemma}
\begin{proof}
Write
\begin{align*}
    u&=r_+(a,b), \quad r_-(a,b)=-s,\\
    v&=r_+(c,d), \quad r_-(c,d)=-t,
\end{align*}
where $u,v,s,t\geq0$. Vieta's relations give
\begin{align*}
    u-s&=a+b, \quad us=ab,\\
    v-t&=c+d, \quad vt=cd.
\end{align*}
Consequently, $0\leq s\leq u$ and $0\leq t\leq v$.

The $(+,+)$ factor in \eqref{eq:conjugate-polynomial} is
\begin{align*}
    q_{++}(z)=z^2-(u+v)z-uv.
\end{align*}
Its two roots are the nonnegative root $Z_{ab\mid cd}$ and a nonpositive root, so its sign for $z>0$ is $\operatorname{sgn}(z-Z_{ab\mid cd})$.  If $z>\max\{u,v\}$, the remaining three factors satisfy
\begin{align*}
    q_{+-}(z)&=z^2-(u-t)z+ut=z(z-u)+t(z+u)>0,\\
    q_{-+}(z)&=z^2-(-s+v)z+sv=z(z-v)+s(z+v)>0,\\
    q_{--}(z)&=z^2+(s+t)z-st\geq z^2-st\geq z^2-uv>0.
\end{align*}
Thus the sign of the product in \eqref{eq:conjugate-polynomial} is exactly the sign of its $(+,+)$ factor, which finishes the proof.
\end{proof}
\subsection{Pairing-difference identity}\label{app:pairing-identities}
Let $e_j=e_j(a,b,c,d)$ denote the elementary symmetric polynomial of degree $j$ in the four inputs. For a scalar $E$, define the auxiliary quartic
\begin{align}
    Q_E(z)\coloneqq z^4-e_1z^3-(e_2+2E)z^2-e_3z+e_4.
    \label{eq:auxiliary-quartic}
\end{align}

The next lemma expresses the differences between pairings in terms of $Q_E$. A complete coefficient-level derivation is given in Appendix \ref{app:pairing-identities}.
\begin{lemma} \label{lem:pairing-differences}
For all $a,b,c,d\geq 0$,
\begin{align*}
    P_{ac\mid bd}(z)-P_{ab\mid cd}(z) & = 2z^2(d-a)(c-b)Q_{ad+bc}(z), \\
    P_{ad\mid bc}(z)-P_{ac\mid bd}(z) & = 2z^2(b-a)(d-c)Q_{ab+cd}(z).
\end{align*}
\end{lemma}
\begin{proof}
We first derive the conjugate polynomial for a generic pairing. Let
\begin{align*}
    A=a+b,\quad B=ab,\quad C=c+d,\quad D=cd,
\end{align*}
and let $u_\pm$ and $v_\pm$ be the respective roots of
\begin{align*}
    u^2-Au-B=0, \quad v^2-Cv-D=0.
\end{align*}
For fixed $u$, multiplication over the two choices of $v$ gives
\begin{align*}
    \prod_{\tau\in\{+,-\}}\bigl[z^2-(u+v_\tau)z-uv_\tau\bigr] & =(z^2-uz)^2-C(z+u)(z^2-uz)-D(z+u)^2\\
    &=L(z)u^2+M(z)u+N(z),
\end{align*}
where
\begin{align*}
    L(z)&=z^2+Cz-D,\\
    M(z)&=-2z(z^2+D),\\
    N(z)&=z^2(z^2-Cz-D).
\end{align*}
Since $u_++u_-=A$ and $u_+u_-=-B$, multiplication over $u_+$ and $u_-$ then yields
\begin{align}
    P_{A,B,C,D}(z) \coloneqq P_{ab \mid cd}(z) =B^2L^2-BM^2+N^2-ABLM+(A^2+2B)LN+AMN.
    \label{eq:compact-degree-eight}
\end{align}
Expanding \eqref{eq:compact-degree-eight} gives the monic degree-eight
polynomial
\begin{align}
    P_{A,B,C,D}(z)={}&z^8-2(A+C)z^7+(A^2+2AC+C^2-2B-2D)z^6\notag\\
    &+2(AB+CD)z^5\notag\\
    &+\bigl(-A^2C^2-2A^2D+2ABC+2ACD+B^2-2BC^2-12BD+D^2\bigr)z^4\notag\\
    &+2(AD^2+B^2C)z^3\notag\\
    &+\bigl(A^2D^2+2ABCD+B^2C^2-2B^2D-2BD^2\bigr)z^2\notag\\
    &-2BD(AD+BC)z+B^2D^2.
    \label{eq:generic-degree-eight}
\end{align}

We now expose the dependence of this polynomial on the chosen pairing.  Let
$e_j$ be the elementary symmetric polynomials in the four underlying inputs
and set
\begin{align*}
    S \coloneqq B+D.
\end{align*}
For every pairing of the same four inputs,
\begin{align*}
    e_1=A+C, \quad e_2=B+D+AC, \quad e_3=AD+BC, \quad e_4=BD.
\end{align*}
The coefficients other than that of $z^4$ simplify immediately from
\begin{align*}
    AB+CD&=e_1S-e_3,\\
    AD^2+B^2C&=e_3S-e_1e_4,\\
    A^2D^2+2ABCD+B^2C^2&=e_3^2,\\
    B^2D+B D^2&=e_4S.
\end{align*}
For the $z^4$ coefficient, put $T \coloneqq AC=e_2-S$.  Since
\begin{align*}
    A^2D+BC^2&=e_1e_3-TS,\\
    ABC+ACD&=TS,\\
    B^2+D^2-12BD&=S^2-14e_4,
\end{align*}
we obtain
\begin{align*}
  & -A^2C^2-2A^2D+2ABC+2ACD+B^2-2BC^2-12BD+D^2\\
  = \; & -T^2-2(e_1e_3-TS)+2TS+S^2-14e_4\\
  = \; & -e_2^2-2e_1e_3-14e_4+6e_2S-4S^2.
\end{align*}
Therefore \eqref{eq:generic-degree-eight} can be written solely in terms of $e_1,e_2,e_3,e_4$ and the pairing-dependent scalar $S$ as
\begin{align}
    P_S(z)=\;&z^8-2e_1z^7+(e_1^2-2S)z^6+2(e_1S-e_3)z^5\notag\\
    &+\bigl(-e_2^2-2e_1e_3-14e_4+6e_2S-4S^2\bigr)z^4\notag\\
    &+2(e_3S-e_1e_4)z^3+(e_3^2-2e_4S)z^2\notag\\
    &-2e_3e_4z+e_4^2.
    \label{eq:pairing-polynomial-symmetric}
\end{align}
This representation makes the two desired differences transparent. Define
\begin{align*}
    S_0\coloneqq ab+cd, \quad S_1\coloneqq ac+bd, \quad S_2\coloneqq ad+bc.
\end{align*}
Then
\begin{align*}
    e_2=S_0+S_1+S_2,
\end{align*}
and
\begin{align*}
    S_1-S_0&=-(d-a)(c-b),\\
    S_2-S_1&=-(b-a)(d-c).
\end{align*}

Let $F=(d-a)(c-b)$. Subtracting
\eqref{eq:pairing-polynomial-symmetric} at $S_0$ from the same formula at
$S_1$, and using $S_0+S_1=e_2-S_2$, gives
\begin{align*}
    P_{S_1}(z)-P_{S_0}(z) &=2Fz^6-2e_1Fz^5-2(e_2+2S_2)Fz^4 -2e_3Fz^3+2e_4Fz^2\\
    &=2Fz^2\bigl[z^4-e_1z^3-(e_2+2S_2)z^2-e_3z+e_4\bigr]\\
    &=2z^2(d-a)(c-b)Q_{ad+bc}(z).
\end{align*}
This is the first equality.

Likewise, let $G=(b-a)(d-c)$. Since $S_1+S_2=e_2-S_0$, subtraction at $S_1$ from the formula at $S_2$ yields
\begin{align*}
    P_{S_2}(z)-P_{S_1}(z)&=2Gz^6-2e_1Gz^5-2(e_2+2S_0)Gz^4-2e_3Gz^3+2e_4Gz^2\\
    &=2Gz^2\bigl[z^4-e_1z^3-(e_2+2S_0)z^2-e_3z+e_4\bigr]\\
    &=2z^2(b-a)(d-c)Q_{ab+cd}(z),
\end{align*}
which is the second equality.
\end{proof}
\subsection{Quartic evaluation}\label{app:quartic-evaluation}
\begin{lemma} \label{lem:quartic-evaluation}
Fix a pairing $(a,b)\mid(c,d)$, where $a,b,c,d\geq 0$, and set
\begin{align*}
    u=\mathcal K(a,b), \quad v=\mathcal K(c,d),\quad z=\mathcal K(u,v).
\end{align*}
If $u,v>0$, then, for every scalar $E$,
\begin{align*}
    Q_E(z)=\frac{2\bigl(abv^2+cdu^2-Euv\bigr)(uv+uz+vz)}{uv},
\end{align*}
where $Q_E$ is formed from the four inputs $a,b,c,d$.
\end{lemma}
\begin{proof}
Let
\begin{align*}
    A=a+b, \quad B = ab, \quad  C=c+d, \quad D = cd.
\end{align*}
The two inner quadratic equations and the assumption $u,v>0$ give
\begin{align}
    A=u-\frac Bu, \quad C=v-\frac Dv.
    \label{eq:inner-sums-in-uv}
\end{align}
The outer quadratic equation gives
\begin{align}
    z^2=(u+v)z+uv.
    \label{eq:outer-reduction}
\end{align}
Set $S=u+v$ and $T=uv$. Repeated use of
\eqref{eq:outer-reduction} yields
\begin{align*}
    z^3&=(S^2+T)z+ST,\\
    z^4&=(S^3+2ST)z+S^2T+T^2.
\end{align*}
For the four inputs $a,b,c,d$, their elementary symmetric polynomials satisfy
\begin{align*}
    e_1=A+C, \quad e_2=B+D+AC, \quad e_3=AD+BC, \quad e_4=BD.
\end{align*}
Substitution into \eqref{eq:auxiliary-quartic} reduces $Q_E(z)$ to a linear polynomial in $z$:
\begin{align*}
    Q_E(z)=R_1z+R_0,
\end{align*}
where
\begin{align*}
    R_0&=S^2T+T^2-e_1ST-(e_2+2E)T+e_4,\\
    R_1&=S^3+2ST-e_1(S^2+T)-(e_2+2E)S-e_3.
\end{align*}
Using \eqref{eq:inner-sums-in-uv}, we first compute
\begin{align*}
    S^2+T-e_1S-(e_2+2E)=\frac{2Bv}{u}+\frac{2Du}{v}-\frac{BD}{uv}-2E.
\end{align*}
Consequently,
\begin{align}
    R_0&=T\left(\frac{2Bv}{u}+\frac{2Du}{v}-\frac{BD}{uv}-2E\right)+BD\notag\\
    &=2\bigl(Bv^2+Du^2-Euv\bigr).
    \label{eq:R0-evaluation}
\end{align}
It remains to relate the coefficient $R_1$ to $R_0$. Directly from their definitions,
\begin{align}
    R_1-\frac ST R_0=ST-e_1T-e_3-\frac{Se_4}{T}.
    \label{eq:R1-R0-relation}
\end{align}
On the other hand, \eqref{eq:inner-sums-in-uv} gives
\begin{align*}
    T(S-e_1)&=Bv+Du,\\
    e_3&=\left(u-\frac Bu\right)D+B\left(v-\frac Dv\right)\\
    &=uD+Bv-\frac{BD}{u}-\frac{BD}{v},\\
    \frac{Se_4}{T}&=\frac{BD}{u}+\frac{BD}{v}.
\end{align*}
Substitution into \eqref{eq:R1-R0-relation} shows that $R_1=(S/T)R_0$. Combining this identity with \eqref{eq:R0-evaluation} and $T+Sz=uv+uz+vz$ gives
\begin{align*}
    Q_E(z)=R_0\left(1+\frac STz\right)=\frac{2\bigl(Bv^2+Du^2-Euv\bigr)(uv+uz+vz)}{uv},
\end{align*}
which finishes the proof.
\end{proof}
\subsection{Proof of the four-point rearrangement theorem}\label{app:rearrangement-proof}
\begin{proof}[Proof of Theorem \ref{thm:four-point-rearrangement}]
For brevity, write
\begin{align*}
    Z_0&\coloneqq Z_{ab\mid cd}
    =\mathcal K\bigl(\mathcal K(a,b),\mathcal K(c,d)\bigr),\\
    Z_1&\coloneqq Z_{ac\mid bd}
    =\mathcal K\bigl(\mathcal K(a,c),\mathcal K(b,d)\bigr),\\
    Z_2&\coloneqq Z_{ad\mid bc}
    =\mathcal K\bigl(\mathcal K(a,d),\mathcal K(b,c)\bigr).
\end{align*}
If $b=0$, then $a=0$, and $\mathcal K(0,x)=x$ gives $Z_0=Z_1=Z_2=\mathcal K(c,d)$. Both inequalities and their stated equality conditions therefore hold. Henceforth assume $b>0$, so all inner quantities appearing below are positive.

For the first comparison, set
\begin{align*}
    u=\mathcal K(a,b),\quad
    v=\mathcal K(c,d),\quad
    z_0=\mathcal K(u,v)=Z_0.
\end{align*}
The quadratic form in Lemma \ref{lem:quartic-evaluation} factors as
\begin{align}
    abv^2+cdu^2-(ad+bc)uv
    =(av-cu)(bv-du).
    \label{eq:first-quadratic-form}
\end{align}
By homogeneity of $\mathcal K$,
\begin{align*}
    av=\mathcal K(ac,ad), \quad
    cu=\mathcal K(ac,bc), \quad
    bv=\mathcal K(bc,bd), \quad
    du=\mathcal K(ad,bd).
\end{align*}
Thus, by strict coordinatewise monotonicity of $\Kop$, $av-cu$ has the sign of $ad-bc$, whereas $bv-du$ has the opposite sign. Consequently,
\begin{align*}
    abv^2+cdu^2-(ad+bc)uv\leq0,
\end{align*}
with equality if and only if $ad=bc$. Applying Lemma \ref{lem:quartic-evaluation} with $E=ad+bc$ yields $Q_{ad+bc}(z_0)\leq0$. Since $P_{ab\mid cd}(z_0)=0$, Lemma \ref{lem:pairing-differences} gives
\begin{align*}
    P_{ac\mid bd}(z_0)
    =2z_0^2(d-a)(c-b)Q_{ad+bc}(z_0)\leq0.
\end{align*}
Moreover,
\begin{align*}
    z_0>v=\mathcal K(c,d)
    \geq\max\{\mathcal K(a,c),\mathcal K(b,d)\}
    =M_{ac\mid bd}.
\end{align*}
The sign certificate in Lemma \ref{lem:sign-beyond-inner-roots} therefore implies $Z_0\leq Z_1$.

If $b=c$, the two pairings defining $Z_0$ and $Z_1$ coincide. If $ad=bc$, equality holds in \eqref{eq:first-quadratic-form}, hence $Q_{ad+bc}(z_0)=0$ and the same sign certificate gives $Z_0=Z_1$. Conversely, if $b<c$ and $ad\neq bc$, then \eqref{eq:first-quadratic-form} is strict and $(d-a)(c-b)>0$. Thus $P_{ac\mid bd}(z_0)<0$, which gives $Z_0<Z_1$. This proves \eqref{eq:first-equality-case}.

For the second comparison, set
\begin{align*}
    u'=\mathcal K(a,c),\quad
    v'=\mathcal K(b,d),\quad
    z_1=\mathcal K(u',v')=Z_1.
\end{align*}
The corresponding quadratic form factors as
\begin{align}
    ac(v')^2+bd(u')^2-(ab+cd)u'v'
    =(cv'-bu')(av'-du').
    \label{eq:second-quadratic-form}
\end{align}
Homogeneity and coordinatewise monotonicity give
\begin{align*}
    cv'&=\mathcal K(bc,cd)
    \geq\mathcal K(ab,bc)=bu',\\
    av'&=\mathcal K(ab,ad)
    \leq\mathcal K(ad,cd)=du'.
\end{align*}
Hence the right-hand side of \eqref{eq:second-quadratic-form} is nonpositive. By Lemma \ref{lem:quartic-evaluation} with $E=ab+cd$, we have $Q_{ab+cd}(z_1)\leq0$. Since $P_{ac\mid bd}(z_1)=0$, Lemma \ref{lem:pairing-differences} yields
\begin{align*}
    P_{ad\mid bc}(z_1)
    =2z_1^2(b-a)(d-c)Q_{ab+cd}(z_1)\leq0.
\end{align*}
Furthermore,
\begin{align*}
    z_1>v'=\mathcal K(b,d)
    \geq\max\{\mathcal K(a,d),\mathcal K(b,c)\}
    =M_{ad\mid bc}.
\end{align*}
Another application of Lemma \ref{lem:sign-beyond-inner-roots} gives $Z_1\leq Z_2$.

If $a=b$ or $c=d$, the pairings defining $Z_1$ and $Z_2$ agree up to symmetry. Conversely, if $a<b$ and $c<d$, strict monotonicity gives $cv'>bu'$ and $av'<du'$. Thus $Q_{ab+cd}(z_1)<0$. Since $(b-a)(d-c)>0$, we have $P_{ad\mid bc}(z_1)<0$. The sign certificate then gives $Z_1<Z_2$. This proves \eqref{eq:second-equality-case} and completes the proof.
\end{proof}
\subsection{An abstract balanced-composition principle}\label{app:abstract-balanced}
The propagation argument uses the order, homogeneity, and rearrangement properties
of the kernel, not its radical formula.
\begin{theorem} \label{thm:abstract-balanced}
    Let $\star:[0,\infty)^2\to[0,\infty)$ be symmetric, positively homogeneous, and strictly increasing in each coordinate. Assume that, for some $\lambda>1$,
    \begin{align*}
        0\star x=x,\quad x\star x=\lambda x,
    \end{align*}
    hold for all $x\ge0$, and that the four-point rearrangement inequalities
    \begin{align*}
        (a\star b)\star(c\star d)
        \le(a\star c)\star(b\star d)
        \le(a\star d)\star(b\star c)
    \end{align*}
    hold whenever $0\le a\le b\le c\le d$. Define the balanced sequence $\{V_N\}_{N\ge0}$ by
    \begin{align}\label{eq:V-balanced}
        V_0=0,
        \quad
        V_1=1,
        \quad
        V_N=V_{\floor{N/2}}\star V_{\ceil{N/2}}
        \;(N\ge2),
    \end{align}
    and define the Bellman sequence $\{W_N\}_{N\ge1}$ by
    \begin{align}\label{eq:W-bellman}
        W_1=1,
        \quad
        W_N=\max_{1\le m<N}W_m\star W_{N-m}
        \;(N\ge2).
    \end{align}
    Then $\{V_N\}_{N\ge0}$ is strictly increasing and satisfies the supercomposition inequality
    \begin{align*}
        V_m\star V_n\le V_{m+n}
        \quad(m,n\ge0).
    \end{align*}
    Consequently, $W_N=V_N$ for every $N\ge1$.
\end{theorem}

\begin{proof}
    Since $V_0<V_1<V_2$, and since $V_r<V_{r+1}$ implies
    \begin{align*}
        V_{2r}=V_r\star V_r
        <V_r\star V_{r+1}=V_{2r+1}
        <V_{r+1}\star V_{r+1}=V_{2r+2},
    \end{align*}
    strict monotonicity follows by induction. We prove the supercomposition inequality by strong induction on $m+n$, assuming $0<m\le n$; the case $m=0$ is immediate. For even--even indices, homogeneity and the induction hypothesis give
    \begin{align*}
        V_{2a}\star V_{2b}=\lambda(V_a\star V_b)\le\lambda V_{a+b}=V_{2a+2b}.
    \end{align*}

    The even--odd case follows from the first rearrangement inequality and the induction hypothesis:
    \begin{align*}
        V_{2a}\star V_{2b+1}
        =(V_a\star V_a)\star(V_b\star V_{b+1})
        \le(V_a\star V_b)\star(V_a\star V_{b+1})
        \le V_{a+b}\star V_{a+b+1}
        =V_{2a+2b+1}.
    \end{align*}

    For odd--even indices, $a+1\le b$, and the same inequality with induction hypothesis gives
    \begin{align*}
        V_{2a+1}\star V_{2b}
        =(V_a\star V_{a+1})\star(V_b\star V_b)
        \le(V_a\star V_b)\star(V_{a+1}\star V_b)
        \le V_{a+b}\star V_{a+b+1}
        =V_{2a+2b+1}.
    \end{align*}

    For two odd indices, $a=b$ gives equality. If $a<b$, the full rearrangement chain and the induction hypothesis give
    \begin{align*}
        V_{2a+1}\star V_{2b+1}
        &=(V_a\star V_{a+1})\star(V_b\star V_{b+1})
        \le(V_a\star V_{b+1})\star(V_{a+1}\star V_b)\\
        &\le V_{a+b+1}\star V_{a+b+1}
        =V_{2a+2b+2}.
    \end{align*}
    This proves supercomposition. Finally, if $W_j=V_j$ for $j<N$, every root split has value at most $V_N$, whereas the balanced split attains $V_N$ by~\eqref{eq:V-balanced}. Induction gives $W_N=V_N$.
\end{proof}
\begin{corollary}\label{cor:balanced-optimality}
    Extend the common scalar sequence by $U_0 \coloneqq 0$. Then $\{U_N\}_{N\ge0}$ is strictly increasing, and for all $m,n\ge0$,
    \begin{align*}
        \Kop(U_m,U_n)\le U_{m+n}.
    \end{align*}
    Moreover, every $N\ge2$ satisfies the balanced recurrence
    \begin{align*} 
        U_N=\Kop\bigl(U_{\floor{N/2}},U_{\ceil{N/2}}\bigr).
    \end{align*}
    In particular, the balanced root split solves the Bellman program~\eqref{eq:U-DP} at every leaf count.
\end{corollary}

\begin{proof}
    Note that \eqref{eq:W-bellman} is exactly the Bellman sequence defined by~\eqref{eq:U-DP}. Apply Theorem \ref{thm:abstract-balanced} with $\star=\Kop$ and $\lambda=\rho$ to obtain that $U_N = V_N$ for all $N\geq 0$, which finishes the proof.
\end{proof}
These statements give the balanced recurrence and supercomposition part of
Theorem~\ref{thm:balanced-main}. Equality and complete-tree optimality are proved next.

\section{Optimal splits, composition trees, and fast evaluation}\label{app:optimal-structure}
The equality conditions below complete the proof of Theorem~\ref{thm:balanced-main}.
All counts refer to ordered composition trees; distinct certificates need not
produce distinct numerical stepsize vectors.

\subsection{Equality in supercomposition and all maximizing root splits}\label{app:equality-classification}

The sharp rearrangement equality conditions determine whether an unbalanced split can tie. The next lemma identifies exactly when the inequality in Corollary \ref{cor:balanced-optimality} is tight. The exceptional pairs are those that become consecutive after removal of a common power of two; this is the source of every nonbalanced tie in the Bellman recursion.

\begin{lemma}\label{lem:equality-supercomposition}
    For positive integers $m,n$,
    \begin{align*}
        \Kop(U_m,U_n)=U_{m+n}
    \end{align*}
    if and only if either
    \begin{enumerate}[label=(\roman*)]
        \item $m=n$; or
        \item there exist integers $k\ge0$ and $r\ge1$ such that
        \begin{align*}
            \{m,n\}=\{2^kr,2^k(r+1)\}.
        \end{align*}
    \end{enumerate}
\end{lemma}

\begin{proof}
    We use strong induction on $m+n$ and assume $m\le n$ by symmetry. The base case $m=n=1$ follows from~\eqref{eq:balanced-recurrence}. If $m=2a$ and $n=2b$, then
    \begin{align*}
        \Kop(U_m,U_n)=\rho\Kop(U_a,U_b),
        \quad
        U_{m+n}=\rho U_{a+b}.
    \end{align*}
    Equality for $(m,n)$ is therefore equivalent to equality for $(a,b)$. By the induction hypothesis, either $a=b$, which gives $m=n$, or $\{a,b\}=\{2^kr,2^k(r+1)\}$. Conversely, either lower-level equality propagates to $(m,n)$ by the same identities.

    Next let $m=2a$ and $n=2b+1$. Here $a\ge1$ and $a\le b$. The supercomposition proof begins with
    \begin{align*}
        \Kop\bigl(\Kop(U_a,U_a),\Kop(U_b,U_{b+1})\bigr)\le\Kop\bigl(\Kop(U_a,U_b),\Kop(U_a,U_{b+1})\bigr).
    \end{align*}
    For the ordered quadruple $U_a\le U_a\le U_b\le U_{b+1}$, the first equality condition in Theorem \ref{thm:four-point-rearrangement} requires either $U_a=U_b$ or $U_aU_{b+1}=U_aU_b$. The second alternative is impossible because $U_a>0$ and $U_{b+1}>U_b$, while strict monotonicity of $\{U_j\}$ makes the first equivalent to $a=b$. Hence the supercomposition inequality is strict unless $(m,n)=(2a,2a+1)$. In that remaining case the two indices are consecutive, and equality follows directly from~\eqref{eq:balanced-recurrence}.

    Similarly, for $(m,n)=(2a+1,2b)$, it forces $b=a+1$, again giving consecutive indices. These equalities also follow directly from the balanced recurrence.

    Finally, for $(m,n)=(2a+1,2b+1)$, equality is immediate when $a=b$. If $a<b$, then
    \begin{align*}
        \Kop\bigl(\Kop(U_a,U_{a+1}),\Kop(U_b,U_{b+1})\bigr)
        & \le \Kop\bigl(\Kop(U_a,U_b),\Kop(U_{a+1},U_{b+1})\bigr)\\
        & <\Kop\bigl(\Kop(U_a,U_{b+1}),\Kop(U_{a+1},U_b)\bigr)\\
        & \le \Kop(U_{a+b+1},U_{a+b+1})
        =U_{2a+2b+2}.
    \end{align*}
    The middle inequality is strict by Theorem \ref{thm:four-point-rearrangement}. Thus unequal odd indices never tie, completing the induction in both directions.
\end{proof}

The optimizer set now follows from the $2$-adic factorization of the total leaf count.

\begin{corollary}\label{cor:all-splits}
    For $N\ge2$, define the set of maximizing left-child leaf counts by
    \begin{align*}
        \mathcal{M}_N
        \coloneqq \argmax_{1\le m<N}\Kop(U_m,U_{N-m}).
    \end{align*}

    Write $N=\nu \cdot 2^\ell$, where $\nu$ is odd. Then
    \begin{enumerate}[label=(\roman*)]
        \item If $\nu=1$, then
        \begin{align*}
            \mathcal{M}_N=\{N/2\}.
        \end{align*}
        \item If $\nu>1$ and $\ell=0$, then
        \begin{align*}
            \mathcal{M}_N=\{(N-1)/2,(N+1)/2\}.
        \end{align*}
        \item If $\nu>1$ and $\ell\ge1$, then
        \begin{align}\label{eq:three-splits}
            \mathcal{M}_N
            =\left\{(\nu-1)2^{\ell-1},\nu\cdot2^{\ell-1},(\nu+1)2^{\ell-1}\right\}.
        \end{align}
    \end{enumerate}

\end{corollary}

\begin{proof}
    By the Bellman recursion and Corollary \ref{cor:balanced-optimality}, a split is maximizing exactly when
    \begin{align*}
        \Kop(U_m,U_{N-m})=U_N=U_{m+(N-m)}.
    \end{align*}
    The equality characterization in Lemma \ref{lem:equality-supercomposition} therefore applies. The equal-child alternative gives the unique balanced split $m=N/2$ whenever $N$ is even. Under the scaled-consecutive alternative,
    \begin{align*}
        N=2^kr+2^k(r+1)=2^k(2r+1).
    \end{align*}
    Since $2r+1$ is odd, uniqueness of the factorization of $N$ into an odd part and a power of two forces $k=\ell$ and $2r+1=\nu$. Thus $r=(\nu-1)/2$, and the two child counts are
    \begin{align*}
        2^\ell r=2^{\ell-1}(\nu-1),
        \quad
        2^\ell(r+1)=2^{\ell-1}(\nu+1).
    \end{align*}
    If $\nu=1$, this alternative is excluded by $r\ge1$, leaving only the equal split. If $\nu>1$ and $\ell=0$, $N$ is odd and only the two orientations of the consecutive split remain. If $\nu>1$ and $\ell\ge1$, the equal split contributes $2^{\ell-1}\nu$ in addition to the two orientations of the scaled-consecutive split. This gives all three cases.
\end{proof}

We finish by translating leaf counts into the step-count notation used by \citet{ZhangJiang}. This makes explicit both the balanced assertion and all the additional ties in their conjecture.

\begin{theorem}[Conjecture~4.17 in \citet{ZhangJiang}]\label{cor:ZhangJiang-conjecture}
    For every step count $n\ge1$,

    \begin{align*}
        \one^\top h_\circ^{(n)}
        =\one^\top\operatorname{ConPP}\!\left(
        h_\circ^{(\floor{(n-1)/2})},
        h_\circ^{(\ceil{(n-1)/2})}
        \right).
    \end{align*}
    Moreover, if $n=\nu\cdot2^\ell-1$ with $\nu>1$ odd and $\ell\ge1$, then, for every $i\in\{-1,0,1\}$,
    \begin{align*}
        \one^\top h_\circ^{(n)}
        =\one^\top\operatorname{ConPP}\!\left(
        h_\circ^{((\nu+i)2^{\ell-1}-1)},
        h_\circ^{((\nu-i)2^{\ell-1}-1)}
        \right),
    \end{align*}
    and no other ConPP root split between two recursively optimized primitive schedules attains the same total stepsize.
\end{theorem}

\begin{proof}
    Note that $N=n+1$. The balanced leaf counts $\floor{N/2}$ and $\ceil{N/2}$ therefore become the child step counts $\floor{(n-1)/2}$ and $\ceil{(n-1)/2}$, which proves the first assertion by Proposition \ref{prop:scalar-equivalence} and Corollary \ref{cor:balanced-optimality}. If $N=\nu\cdot2^\ell$ with $\nu>1$ odd and $\ell\ge1$, by Corollary \ref{cor:all-splits}, the maximizing leaf counts are $(\nu+i)2^{\ell-1}$ for $i\in\{-1,0,1\}$. Completeness follows again from Corollary \ref{cor:all-splits}.
\end{proof}

By the treewise conjugacy established in Proposition \ref{prop:scalar-equivalence}, the same split rules characterize all optimal basic s-join trees: replace ConPP by the s-join and interpret maximization of $U_N$ as minimization of $s_N=U_N^{-1}$. It remains only to pass from the root optimizer set to complete trees.


\begin{corollary} \label{cor:optimal-trees}
    Let $T$ be an ordered full binary composition tree with $N$ leaves, assign value $1$ to every leaf, and evaluate each internal node using $\Kop$. Then $T$ attains the optimal root value $U_N$ if and only if, at every internal node with $j$ descendant leaves, its two child subtrees are optimal for their respective leaf counts and its left-child leaf count belongs to $\mathcal{M}_j$. Equivalently, $T$ is optimal if and only if every internal split satisfies one of the two equality patterns in Lemma \ref{lem:equality-supercomposition}.
\end{corollary}

\begin{proof}
    If the root subtrees have $m$ and $N-m$ leaves and values $X$ and $Y$, then $X\le U_m$ and $Y\le U_{N-m}$. Strict monotonicity of $\Kop$ shows that the root can attain $U_N$ only if $X=U_m$, $Y=U_{N-m}$, and $m\in\mathcal M_N$; these conditions are also sufficient by Corollary \ref{cor:all-splits}. Recursion gives necessity at every node, while upward induction gives sufficiency.
\end{proof}

The resulting multiplicity of optimal plane composition certificates is quantified in Appendix \ref{app:tree-multiplicity}.
\subsection{Optimal-tree multiplicity}\label{app:tree-multiplicity}

Although each internal node has at most three maximizing splits, independent choices across the hierarchy can produce many optimal composition certificates. Let $T_N$ be the number of \emph{plane} full binary trees with $N$ unit-valued leaves whose evaluation under $\Kop$ equals $U_N$; plane means that left and right children are ordered. These numbers count recursive certificates, which need not yield distinct numerical schedule vectors.

\begin{proposition} \label{prop:tree-recurrence}
    The sequence $\{T_N\}_{N\ge1}$ satisfies
    \begin{align}\label{eq:tree-recurrence}
        T_1=1,
        \quad
        T_N=\sum_{m\in\mathcal{M}_N}T_mT_{N-m}
        \quad(N\ge2).
    \end{align}
    More explicitly, write $N=\nu2^\ell$, where $\nu$ is odd. For $N\ge2$,
    \begin{enumerate}[label=(\roman*)]
        \item if $\nu=1$, then $T_N=T_{N/2}^2=1$;
        \item if $\nu>1$ and $\ell=0$, then
        \begin{align*}
            T_N=2T_{(N-1)/2}T_{(N+1)/2};
        \end{align*}
        \item if $\nu>1$ and $\ell\ge1$, setting
        \begin{align*}
            a=(\nu-1)2^{\ell-1},
            \quad
            b=\nu2^{\ell-1},
            \quad
            c=(\nu+1)2^{\ell-1},
        \end{align*}
        gives
        \begin{align}\label{eq:tree-three}
            T_N=2T_aT_c+T_b^2.
        \end{align}
    \end{enumerate}
\end{proposition}

\begin{proof}
    There is only one full binary tree with one leaf, so $T_1=1$. Fix $N\ge2$ and consider an optimal plane tree with $N$ leaves. If its left subtree has $m$ leaves, then its right subtree has $N-m$ leaves. By Corollary \ref{cor:optimal-trees}, optimality of the whole tree is equivalent to the conjunction of the following three conditions: $m\in\mathcal{M}_N$, the left subtree is optimal for $m$ leaves, and the right subtree is optimal for $N-m$ leaves.

    For a fixed ordered root index $m\in\mathcal{M}_N$, the two subtrees can be chosen independently. There are $T_m$ possibilities for the left subtree and $T_{N-m}$ possibilities for the right subtree, hence $T_mT_{N-m}$ optimal trees with that root index. Different values of $m$ specify different ordered root splits and therefore give disjoint classes of plane trees. Summing over all admissible root indices proves~\eqref{eq:tree-recurrence}.

    We next substitute the explicit optimizer sets from Corollary \ref{cor:all-splits}. If $\nu=1$, then $N$ is a power of two and $\mathcal{M}_N=\{N/2\}$, so
    \begin{align*}
        T_N=T_{N/2}^2.
    \end{align*}
    Repeatedly halving $N$ eventually reaches $1$; since $T_1=1$, induction on $\ell$ yields $T_N=1$. Thus the completely balanced plane tree is the unique optimal certificate at dyadic leaf counts.

    If $\nu>1$ and $\ell=0$, then $N$ is odd and
    \begin{align*}
        \mathcal{M}_N=\left\{\frac{N-1}{2},\frac{N+1}{2}\right\}.
    \end{align*}
    The two terms in~\eqref{eq:tree-recurrence} are equal after interchanging the ordered children, which gives
    \begin{align*}
        T_N
        &=T_{(N-1)/2}T_{(N+1)/2}
          +T_{(N+1)/2}T_{(N-1)/2}\\
        &=2T_{(N-1)/2}T_{(N+1)/2}.
    \end{align*}

    Finally, suppose $\nu>1$ and $\ell\ge1$. In the notation of the statement, Corollary \ref{cor:all-splits} gives
    \begin{align*}
        \mathcal{M}_N=\{a,b,c\}.
    \end{align*}
    Since $a+c=N$ and $2b=N$, the three contributions to~\eqref{eq:tree-recurrence} are
    \begin{align*}
        T_aT_c,
        \quad
        T_b^2,
        \quad
        T_cT_a,
    \end{align*}
    and their sum is exactly~\eqref{eq:tree-three}.
\end{proof}

In particular, if $A_k\coloneqq T_{3\cdot2^k}$, then
    \begin{align*}
        A_0=2,
        \quad
        A_{k+1}=A_k^2+2,
    \end{align*}
    and consequently
    \begin{align*}
        T_{3\cdot2^k}=\exp\bigl(\Theta(2^k)\bigr)=\exp\bigl(\Theta(N)\bigr).
    \end{align*}
Thus dyadic leaf counts have a unique optimal plane certificate, whereas the subsequence $N=3\cdot2^k$ has exponentially many. For reference,
\begin{align*}
    (T_1,\ldots,T_{12})&=(1,1,2,1,4,6,4,1,8,28,48,38),\\
    (T_{3\cdot2^k})_{k=0}^{5}&=(2,6,38,1446,2090918,4371938082726).
\end{align*}
\subsection{Correctness and complexity of the binary evaluator}\label{app:algorithm-proof}
For every $r\ge0$, the balanced recurrence and $\Kop(x,x)=\rho x$ give
\begin{align*}
(U_{2r},U_{2r+1})&=(\rho U_r,\Kop(U_r,U_{r+1})),\\
(U_{2r+1},U_{2r+2})&=(\Kop(U_r,U_{r+1}),\rho U_{r+1}).
\end{align*}
Indeed, after a binary prefix of value $r$ has been processed, the invariant is $(a,b)=(U_r,U_{r+1})$. The two identities above map this pair to $(U_{2r},U_{2r+1})$ or $(U_{2r+1},U_{2r+2})$ according to the next digit. Hence, at termination, $a=U_N$. Since $k=\floor{\log_2N}+1$, Algorithm \ref{alg:bitwise} uses exactly $k$ evaluations of $\Kop$, $k$ multiplications by $\rho$, and $\Theta(1)$ scalar memory. Thus $U_N$ is computed in $\Theta(\log N)$ scalar time for $N\ge2$, without forming $U_1,\ldots,U_{N-1}$. These complexity statements use the unit-cost scalar-arithmetic model, with evaluations of $\Kop$ counted as unit-cost operations and input storage excluded.


Under this model, computing the
whole prefix table costs $O(N)$ scalar operations, and explicitly returning an
$(N-1)$-step schedule requires $\Omega(N)$ output operations.

\section{Symmetric phase: proofs and further properties}\label{app:phase-regularity-proofs}
This appendix proves the exact symmetric phase law in Theorem~\ref{thm:symmetric-main} and derives the tight objective consequences in Corollary~\ref{cor:sysmmetic-main-objective}. It also records subdivision derivatives, a dyadic corner, certified minimum bounds, and marginal increments. Throughout, $q=1/p=\log_\rho2$.

\subsection{Sharp power envelope}\label{app:silver-power}
The balanced split of $2N$ is $(N,N)$. By Corollary \ref{cor:balanced-optimality} and $\Kop(x,x)=\rho x$, we have $U_{2N}=\rho U_N$. Iteration, together with $s_N=U_N^{-1}$, yields
\begin{align*}
    U_{2^kN}=\rho^kU_N, \quad s_{2^kN}=\rho^{-k}s_N, \quad N\ge1,\ k\ge0.
\end{align*}
\begin{lemma} \label{lem:silver-power}
For all $x,y\geq0$,
\begin{align*}
    \Kop(x,y)\le (x^q+y^q)^{1/q}.
\end{align*}
Equality holds if and only if $x=y$ or $xy=0$.
\end{lemma}
\begin{proof}
The equality is immediate if $xy=0$, so assume $x,y>0$. Write
\begin{align*}
    x=\e^{m-s}, \quad y=\e^{m+s}, \quad s\ge0,
\end{align*}
where symmetry allows the last restriction. Direct calculation gives
\begin{align*}
    \Kop(x,y)&=\e^m\exp\bigl(\operatorname{arsinh}(\cosh s)\bigr),\\
    (x^q+y^q)^{1/q}&=\e^m\bigl(2\cosh(qs)\bigr)^{1/q}.
\end{align*}
It therefore suffices to show
\begin{align*}
    G(s)\coloneqq\frac1q\log\bigl(2\cosh(qs)\bigr)
    -\operatorname{arsinh}(\cosh s)\ge0.
\end{align*}
Since $2^{1/q}=\rho$, one has $G(0)=0$, while the large-$s$ asymptotics give $G(s)\to0$ as $s\to\infty$. Moreover,
\begin{align*}
    G'(s)=\tanh(qs)-\frac{\sinh s}{\sqrt{1+\cosh^2s}}.
\end{align*}
Define
\begin{align*}
    \alpha(s)\coloneqq
    \operatorname{artanh}\left(\frac{\sinh s}{\sqrt{1+\cosh^2s}}\right).
\end{align*}
Then
\begin{align*}
    \alpha'(s)=\frac{\cosh s}{\sqrt{1+\cosh^2s}},
\end{align*}
which increases strictly from $1/\sqrt2$ to $1$. Also $1/\sqrt2<q<1$: the upper bound follows from $\rho>2$, and the lower bound from $\rho^3=7+5\sqrt2<16$. Hence $\alpha(s)-qs$ is strictly convex, initially decreasing, tends to $+\infty$, and has a unique positive zero $s_0$. Since $\tanh$ is strictly increasing, $G'$ is positive on $(0,s_0)$ and negative on $(s_0,\infty)$. Together with the endpoint values, this proves $G(s)>0$ for $s>0$. Equality therefore occurs only when $s=0$, equivalently $x=y$.
\end{proof}
\begin{proposition}\label{prop:sharp-growth}
For every $N\ge1$,
\begin{align*}
    U_N\le N^p.
\end{align*}
Equivalently, $s_N\ge N^{-p}$. Equality holds if and only if $N$ is a power of two.
\end{proposition}

\begin{proof}
We argue by induction on $N$. For $N=1$, we directly have $U_1 = 1 = 1^p$. For $N\ge2$, let $a=\floor{N/2}$ and $b=\ceil{N/2}$. Monotonicity of $\Kop$, the induction hypothesis, Lemma \ref{lem:silver-power}, and $pq=1$ give
\begin{align*}
    U_N=\Kop(U_a,U_b)
    \le \Kop(a^p,b^p)
    \le (a^{pq}+b^{pq})^{1/q}
    =(a+b)^p=N^p.
\end{align*}
If $N$ is odd, then $a\ne b$, so the second inequality is strict. If $N=2a$, equality holds precisely when $U_a=a^p$. Recursing proves that equality occurs exactly at powers of two.
\end{proof}
\subsection{Adjacent-ratio contraction}\label{app:proof-ratio-contraction}
The regularity of the phase profile is controlled by a one-dimensional dynamical system. For $N \geq 1$, define
\begin{align*}
    R_N\coloneqq \frac{U_{N+1}}{U_N}, \quad
    k(r)\coloneqq \Kop(1,r), \quad
    T_0(r)\coloneqq \frac{k(r)}\rho, \quad
    T_1(r)\coloneqq \frac{\rho r}{k(r)}.
\end{align*}
\begin{lemma} \label{lem:ratio-contraction}
For every $N\ge1$,
\begin{align}\label{eq:R-recursion}
    R_{2N}=T_0(R_N), \quad R_{2N+1}=T_1(R_N),
\end{align}
and
\begin{align}\label{eq:R-bound}
    1<R_N\le\rho, \quad
    0<R_N-1\le\frac{\rho-1}{2^{\floor{\log_2N}}}
    \le\frac{2(\rho-1)}N.
\end{align}
\end{lemma}
\begin{proof}
The adjacent-pair identities in Section \ref{sec:balanced} and the homogeneity of $\Kop$ give \eqref{eq:R-recursion}. Put
\begin{align*}
    \Delta(r)\coloneqq \sqrt{r^2+6r+1}.
\end{align*}
Then
\begin{align}\label{eq:AB-explicit}
    T_0(r)=\frac{1+r+\Delta(r)}{2\rho}, \quad T_1(r)=\frac\rho2\bigl(\Delta(r)-r-1\bigr),
\end{align}
where the second identity follows from $r/k(r)=k(r)-(1+r)$. Both maps fix $1$. Their derivatives are
\begin{align}\label{eq:AB-derivatives}
    T_0'(r)&=\frac1{2\rho}\left(1+\frac{r+3}{\Delta(r)}\right),\\
    T_1'(r)&=\frac\rho2\left(\frac{r+3}{\Delta(r)}-1\right).\notag
\end{align}
Since
\begin{align*}
    \left(\frac{r+3}{\Delta(r)}\right)'=-\frac8{\Delta(r)^3}<0,
\end{align*}
both derivatives are positive and at most $1/2$ on $[1,\rho]$, with equality at $r=1$. Furthermore,
\begin{align*}
    \rho=\Kop(0,\rho)<k(\rho)<\Kop(\rho,\rho)=\rho^2,
\end{align*}
so $1<T_0(\rho)<\rho$ and $1<T_1(\rho)<\rho$. Hence $T_0$ and $T_1$ map $[1,\rho]$ into itself and are $1/2$-Lipschitz on this interval, with common fixed point $1$.

For $N=1$, the claimed bounds follow immediately from $R_1=\rho$. Hence assume $N\ge2$ and write
\begin{align*}
    N=(1\epsilon_1\epsilon_2\cdots\epsilon_m)_2, \quad \epsilon_j\in\{0,1\}, \quad m=\floor{\log_2N},
\end{align*}
and let $n_j=(1\epsilon_1\cdots\epsilon_j)_2$ be the integer represented by the first $j+1$ digits, with $n_0=1$. Then
\begin{align*}
    n_j=2n_{j-1}+\epsilon_j, \quad 1\le j\le m, \quad n_m=N.
\end{align*}
The two identities in \eqref{eq:R-recursion} give, successively,
\begin{align*}
    R_{n_j}=T_{\epsilon_j}(R_{n_{j-1}}), \quad R_N=T_{\epsilon_m}\circ T_{\epsilon_{m-1}}\circ\cdots\circ T_{\epsilon_1}(R_1).
\end{align*}
Thus each digit after the leading digit selects the corresponding branch map $T_{\epsilon_j}$. For either $T\in\{T_0,T_1\}$, the fixed-point identity $T(1)=1$ and the derivative bound $0<T'(r)\le1/2$ on $[1,\rho]$ imply
\begin{align*}
    0\le T(r)-1=T(r)-T(1)\le\frac12(r-1), \quad 1\le r\le\rho.
\end{align*}
Because $R_1=U_2/U_1=\rho$ and both maps preserve $[1,\rho]$, induction along the binary prefixes yields
\begin{align*}
    0\le R_{n_j}-1\le2^{-j}(R_1-1)=2^{-j}(\rho-1), \quad 0\le j\le m.
\end{align*}
Taking $j=m=\floor{\log_2N}$ gives $R_N-1\le(\rho-1)2^{-m}$ and, in particular, $R_N\le\rho$. Moreover, $2^m\le N<2^{m+1}$ implies
\begin{align*}
    2^{-m}\le\frac2N,
\end{align*}
which proves both upper bounds in \eqref{eq:R-bound}. Finally, the strict increase of $\{U_N\}$ established in the preceding section gives $U_{N+1}>U_N>0$, and hence
\begin{align*}
    R_N=\frac{U_{N+1}}{U_N}>1.
\end{align*}
This proves the strict lower bound and completes the proof.
\end{proof}
\subsection{Canonical interpolation}\label{app:proof-bilipschitz}
For each $k\ge0$, define the level-$k$ dyadic grid and its intervals by
\begin{align*}
    \mathcal D_k \coloneqq \left\{\frac{j}{2^k}:2^k\le j\le2^{k+1}\right\}, \quad    I_{k,j} \coloneqq \left[\frac{j}{2^k},\frac{j+1}{2^k}\right], \; 2^k\le j<2^{k+1}.
\end{align*}
The elements of $\mathcal D_k$ are called level-$k$ nodes, the sets $I_{k,j}$ are called level-$k$ intervals, and $\mathcal D \coloneqq \bigcup_{k\ge0}\mathcal D_k$ is the set of dyadic rationals in $[1,2]$. For $x=j/2^k\in\mathcal D_k$, define
\begin{align}\label{eq:F-dyadic}
    F(x) \coloneqq \rho^{-k}U_j.
\end{align}
This is well defined since the definition is independent of the chosen level by \eqref{eq:dyadic-scaling}. Let $F_k$ be the continuous function that agrees with $F$ on $\mathcal D_k$ and is affine on every $I_{k,j}$. For the endpoints $x<y$ of any level-$k$ interval, the balanced recurrence gives the nonlinear midpoint rule
\begin{align}\label{eq:midpoint-rule}
    F\left(\frac{x+y}{2}\right)
    =\frac1\rho\Kop\bigl(F(x),F(y)\bigr).
\end{align}
\begin{lemma}\label{lem:Lipschitz-F}
The function in \eqref{eq:F-dyadic} extends uniquely to a strictly increasing bi-Lipschitz map
\begin{align*}
    F:[1,2]\to[1,\rho]
\end{align*}
that satisfies $F(1)=1$, $F(2)=\rho$, and \eqref{eq:midpoint-rule} on every dyadic interval. More precisely,
\begin{align}\label{eq:bilipschitz}
    c_-(y-x)\le F(y)-F(x)\le c_+(y-x),
    \quad 1\le x<y\le2,
\end{align}
where
\begin{align*}
    c_-\coloneqq(\rho-1)\left(1-\frac1{2\rho}\right),
    \quad
    c_+\coloneqq(\rho-1)\exp\left(\frac1{2\rho}\right).
\end{align*}
In particular, one may take $L_F\coloneqq c_+$ as a Lipschitz constant for $F$.
\end{lemma}
For $k\ge0$ and $2^k\le j<2^{k+1}$, recall that
\begin{align*}
    I_{k,j}=
    \left[\frac{j}{2^k},\frac{j+1}{2^k}\right].
\end{align*}
For $r>1$, define the subdivision weights
\begin{align}\label{eq:w-def}
    w_0(r)\coloneqq\frac{2(T_0(r)-1)}{r-1},
    \quad
    w_1(r)\coloneqq\frac{2(r-T_0(r))}{r-1},
\end{align}
and set $w_0(1)=w_1(1)=1$. Thus $w_0+w_1=2$.

\begin{proof}
Let $S_{k,j}$ be the slope of $F_k$ on $I_{k,j}$. Write $a \coloneqq F(j/2^k)$. Since the right endpoint value is $aR_j$, the midpoint rule, the homogeneity of $\Kop$, and the definition of $T_0$ give
\begin{align*}
    F\left(\frac{2j+1}{2^{k+1}}\right)=\frac1\rho\Kop(a,aR_j)=\frac a\rho\Kop(1,R_j)=aT_0(R_j).
\end{align*}
The left and right children of $I_{k,j}$ are $I_{k+1,2j}$ and $I_{k+1,2j+1}$. Their increments are $a(T_0(R_j)-1)$ and $a(R_j-T_0(R_j))$, respectively, whereas the increment over $I_{k,j}$ is $a(R_j-1)$. Therefore
\begin{align*}
    S_{k,j}&=2^ka(R_j-1),\\
    S_{k+1,2j}&=2^{k+1}a(T_0(R_j)-1)=\frac{2(T_0(R_j)-1)}{R_j-1}S_{k,j},\\
    S_{k+1,2j+1}&=2^{k+1}a(R_j-T_0(R_j))=\frac{2(R_j-T_0(R_j))}{R_j-1}S_{k,j}.
\end{align*}
Thus
\begin{align}\label{eq:slope-subdivision}
    S_{k+1,2j}=w_0(R_j)S_{k,j}, \quad S_{k+1,2j+1}=w_1(R_j)S_{k,j}.
\end{align}
Differentiating \eqref{eq:AB-derivatives} once more gives
\begin{align*}
    T_0''(r)=-\frac4{\rho(r^2+6r+1)^{3/2}}.
\end{align*}
Thus $T_0$ is strictly concave and $T_0'(1)=1/2$. On $[1,\rho]$,
\begin{align}\label{eq:M-correct}
    0\le-T_0''(r)\le M, \quad M \coloneqq \frac1{4\rho(\rho-1)},
\end{align}
because $-T_0''$ is decreasing on $[1,\rho]$ and $-T_0''(1)=1/[4\rho(\rho-1)]$. For $r>1$, the strict increase of $T_0$ gives $w_0(r)>0$, while concavity and $T_0(1)=1$, $T_0'(1)=1/2$ give
\begin{align*}
    T_0(r)-1\le T_0'(1)(r-1)=\frac12(r-1),
\end{align*}
so $w_0(r)\le1$. Moreover, Taylor's theorem yields some $\xi_r\in(1,r)$ such that
\begin{align*}
    T_0(r)=T_0(1)+T_0'(1)(r-1)+\frac12T_0''(\xi_r)(r-1)^2=1+\frac12(r-1)+\frac12T_0''(\xi_r)(r-1)^2.
\end{align*}
Hence, using $w_0+w_1=2$ and \eqref{eq:M-correct},
\begin{align*}
    0\le1-w_0(r)=w_1(r)-1=-T_0''(\xi_r)(r-1)\le M(r-1).
\end{align*}
The definitions at $r=1$ give the same bounds there. Thus, for every $r\in[1,\rho]$,
\begin{align}\label{eq:w-errors}
    0<w_0(r)\le1\le w_1(r), \quad 0\le1-w_0(r)=w_1(r)-1\le M(r-1).
\end{align}

Fix $I_{k,J}$ and, for $0\le\ell\le k$, set
\begin{align*}
    j_\ell \coloneqq \floor{\frac{J}{2^{k-\ell}}}.
\end{align*}
Then $I_{\ell,j_\ell}$ is the unique level-$\ell$ interval containing $I_{k,J}$. The nested sequence $I_{0,j_0}\supset I_{1,j_1}\supset\cdots\supset I_{k,j_k}$ is called the ancestor chain of $I_{k,J}$. Here $j_0=1$, $j_k=J$, and for each $0\le\ell<k$ there is a unique $\delta_\ell\in\{0,1\}$ such that $j_{\ell+1}=2j_\ell+\delta_\ell$; the values $0$ and $1$ select the left and right child, respectively. Set $r_\ell \coloneqq R_{j_\ell}$ and
\begin{align*}
    e_\ell \coloneqq 1-w_0(r_\ell)=w_1(r_\ell)-1, \quad \sigma_\ell \coloneqq 2\delta_\ell-1\in\{-1,1\}.
\end{align*}
Then $w_{\delta_\ell}(r_\ell)=1+\sigma_\ell e_\ell$. Since $I_{0,1}=[1,2]$ has slope $\rho-1$, iteration of \eqref{eq:slope-subdivision} along this chain gives
\begin{align*}
    S_{k,J}=(\rho-1)\prod_{\ell=0}^{k-1}\bigl(1+\sigma_\ell e_\ell\bigr).
\end{align*}
Since $2^\ell\le j_\ell<2^{\ell+1}$, \eqref{eq:R-bound} and \eqref{eq:w-errors} imply
\begin{align*}
    0\le e_\ell\le M(r_\ell-1)\le M(\rho-1)2^{-\ell}, \quad 0\le\ell<k.
\end{align*}
Consequently,
\begin{align*}
    E \coloneqq \sum_{\ell=0}^{\infty}M(\rho-1)2^{-\ell}=2M(\rho-1)=\frac1{2\rho}<1.
\end{align*}
In particular, $0\le e_\ell<1$. Since $1-e_\ell\le1+\sigma_\ell e_\ell\le1+e_\ell$, the elementary inequalities $\prod_{\ell=0}^{k-1}(1-e_\ell)\ge1-\sum_{\ell=0}^{k-1}e_\ell$, proved by induction, and $1+t\le \e^t$ for $t\ge0$ yield
\begin{align*}
    S_{k,J}&\ge(\rho-1)\prod_{\ell=0}^{k-1}(1-e_\ell)\ge(\rho-1)\left(1-\sum_{\ell=0}^{k-1}e_\ell\right)\ge(\rho-1)(1-E)=c_-,\\
    S_{k,J}&\le(\rho-1)\prod_{\ell=0}^{k-1}(1+e_\ell)\le(\rho-1)\exp\left(\sum_{\ell=0}^{k-1}e_\ell\right)\le(\rho-1)\e^E=c_+.
\end{align*}
Thus every slope of every $F_k$ lies in $[c_-,c_+]$.

If $x=m/2^k<y=n/2^k$ are in the same grid $\mathcal D_k$, then
\begin{align*}
    F(y)-F(x)=2^{-k}\sum_{j=m}^{n-1}S_{k,j}.
\end{align*}
The slope bounds therefore give \eqref{eq:bilipschitz} for $x,y\in\mathcal D$, after passing to a common level when necessary. The upper bound makes $F$ uniformly continuous on the dense set $\mathcal D$, so $F$ has a unique continuous extension to $[1,2]$. Taking dyadic limits gives \eqref{eq:bilipschitz} for all $1\le x<y\le2$. The lower bound makes the extension strictly increasing, and $F(1)=U_1=1$, $F(2)=U_2=\rho$ show that its range is $[1,\rho]$. Finally, \eqref{eq:midpoint-rule} remains valid on every dyadic interval because its two endpoints and midpoint belong to $\mathcal D$.
\end{proof}
\subsection{Exact scalar phase and cluster sets}\label{app:exact-symmetric-phase}
\begin{theorem} \label{thm:phase}
For $0\le t\le1$, define
\begin{align*}
    \Phi(t)\coloneqq \rho^{-t}F(2^t),
\end{align*}
and extend $\Phi$ as a $1$-periodic function. Let $\mathbb T \coloneqq \R/\mathbb Z$. Then $\Phi: \mathbb T \to \mathbb R$ is positive, Lipschitz, and nonconstant, and
\begin{align*}
    U_N=N^p\Phi\bigl(\{\log_2N\}\bigr) \quad (N\ge1),
\end{align*}
where $\{\cdot\}$ takes the fractional part of a real. It is the unique continuous $1$-periodic function with this property. Furthermore,
\begin{align*}
    0<\phimin\coloneqq\min_{t\in[0,1]}\Phi(t)<1,
    \quad
    \max_{t\in[0,1]}\Phi(t)=1.
\end{align*}
One valid Lipschitz constant on the phase circle is $L_\Phi\coloneqq2(\log 2)c_++\rho\log\rho$.
\end{theorem}

\begin{proof}
Let $k=\floor{\log_2N}$, $x=N/2^k\in[1,2)$, and $t=\log_2x=\{\log_2N\}$. By \eqref{eq:F-dyadic}, $F(x)=\rho^{-k}U_N$, while $N^p=\rho^{k+t}$. Hence
\begin{align*}
    U_N=\rho^kF(2^t)=N^p\rho^{-t}F(2^t)=N^p\Phi(t),
\end{align*}
which proves \eqref{eq:exact-phase}. Since $\Phi(0)=F(1)=1$ and $\Phi(1)=\rho^{-1}F(2)=1$, the periodic extension is continuous. For $s,t\in[0,1]$, \eqref{eq:bilipschitz} gives
\begin{align*}
    |\Phi(t)-\Phi(s)| \le c_+|2^t-2^s|+\rho|\rho^{-t}-\rho^{-s}| \le\bigl(2(\log 2)c_++\rho\log\rho\bigr)|t-s|.
\end{align*}
The same estimate across the identified endpoints gives the stated Lipschitz constant. Positivity follows from $F>0$.

For $x=j/2^m\in\mathcal D$, Proposition \ref{prop:sharp-growth} gives
\begin{align*}
    F(x)=\rho^{-m}U_j\le\rho^{-m}j^p=x^p.
\end{align*}
Density extends this bound to $[1,2]$, so $\Phi\le1$. At $t_*=\log_2(3/2)$, the strict part of Proposition \ref{prop:sharp-growth} gives
\begin{align*}
    \Phi(t_*)=\frac{U_3}{3^p}<1.
\end{align*}
Thus \eqref{eq:phase-extrema} holds and $\Phi$ is nonconstant. Finally, the phases $\{\{\log_2N\}:N\ge1\}$ are dense: for any $\theta\in[0,1)$, take $N_k=\floor{2^{k+\theta}}$. Any other continuous periodic profile satisfying \eqref{eq:exact-phase} agrees with $\Phi$ on this dense set and hence everywhere.
\end{proof}
As in Section~\ref{sec:phase}, $\Clust(a_N)$ denotes the set of subsequential limits of $a_N$ as $N\to\infty$.
\begin{corollary}\label{cor:cluster-sets}
The normalized scalar and rate sequences satisfy
\begin{align*}
    \Clust\left(\frac{U_N}{N^p}\right)=[\phimin,1],
    \quad
    \Clust(N^ps_N)=\left[1,\frac1{\phimin}\right].
\end{align*}
In particular, neither sequence converges.
\end{corollary}

\begin{proof}
By \eqref{eq:exact-phase}, $U_N/N^p=\Phi(\theta_N)$ with $\theta_N=\{\log_2N\}$. The sequence $\{\theta_N\}$ is dense on the phase circle, and every phase is approached along integers tending to infinity. Continuity of $\Phi$ and \eqref{eq:phase-extrema} therefore give
\begin{align*}
    \Clust\left(\frac{U_N}{N^p}\right)
    =\Phi([0,1])=[\phimin,1].
\end{align*}
Since $N^ps_N=(U_N/N^p)^{-1}$ and $\Phi$ is positive, taking reciprocals proves the second identity.
\end{proof}
\subsection{Tight objective-gap phase}\label{app:objective-phase-proof}

We now translate the scalar phase into worst-case guarantees for gradient descent with recursive primitive schedules and OBS-S. Restore a general smoothness constant $L$, and apply a dimensionless schedule $h=(h_0,\ldots,h_{N-2})^\top$ through
\begin{align*}
    x_{i+1}=x_i-\frac{h_i}{L}\nabla f(x_i).
\end{align*}
For an optimized primitive schedule of $N-1$ steps, Proposition \ref{prop:scalar-equivalence} gives $H_N=U_N-1$. Hence the tight Huber certificate \eqref{eq:primitive-gap-imported} becomes
\begin{align*}
    f(x_{N-1})-f_*
    \le B_N\frac{L}{2}\norm{x_0-x_*}^2,
    \quad
    B_N\coloneqq\frac1{2U_N-1}.
\end{align*}
The same coefficient governs the objective consequence of OBS-S. Indeed, its optimized s-composable rate is $\eta=s_N=1/U_N$, and~\eqref{eq:s-gap-imported} gives
\begin{align*}
    \frac{\eta}{2-\eta}
    =\frac{1/U_N}{2-1/U_N}
    =\frac{1}{2U_N-1}=B_N.
\end{align*}
For primitive schedules, equality in~\eqref{eq:objective-bound} is attained on the one-dimensional Huber instance of \citet[Theorem~2.6 and Corollary~2.9]{ZhangJiang}; for OBS-S, a matching one-dimensional Huber instance is supplied by \citet[Proposition~1]{GrimmerShuWangMOR}. Thus $B_N$ is the exact worst-case objective coefficient for the corresponding optimized schedules, not merely an upper estimate produced by the scalar analysis.

It is useful to distinguish $B_N$ from the s-composable rate itself. Their exact relation is
\begin{align}\label{eq:B-vs-s}
    B_N=\frac{s_N}{2-s_N},
     \quad
    B_N-\frac{s_N}{2}=\frac{s_N^2}{2(2-s_N)}>0.
\end{align}
Consequently, the objective coefficient is asymptotically one half of the OBS-S rate, while the positive correction in~\eqref{eq:B-vs-s} records the finite-horizon effect of the term $-1$ in the denominator $2U_N-1$.

\begin{corollary}\label{cor:objective-phase}
Let $\theta_N = \{\log_2N\}$. For every $N\ge1$,
\begin{align*}
    B_N=\frac1{2N^p\Phi(\theta_N)-1}.
\end{align*}
As $N \to \infty$,
\begin{align*}
    B_N
    =\frac{N^{-p}}{2\Phi(\theta_N)}
    +O(N^{-2p}),
\end{align*}
where the $O(N^{-2p})$ remainder is uniform in the dyadic phase $\theta_N$.
Furthermore,
\begin{align*}
    \Clust(N^pB_N)
    =\left[\frac12,\frac1{2\phimin}\right].
\end{align*}
Consequently, the normalized tight objective coefficient has no single asymptotic constant.
\end{corollary}

\begin{proof}
Substituting \eqref{eq:exact-phase} into the definition of $B_N$ gives \eqref{eq:B-exact}.  Since $\Phi(\theta_N)\ge\phimin>0$,
\begin{align*}
    B_N-\frac{N^{-p}}{2\Phi(\theta_N)}
    =\frac1{\bigl(2N^p\Phi(\theta_N)-1\bigr)
        2N^p\Phi(\theta_N)}.
\end{align*}
Moreover, $N^p\Phi(\theta_N)=U_N\ge1$, so $2N^p\Phi(\theta_N)-1\ge N^p\Phi(\theta_N)$. The displayed remainder is therefore bounded above by $N^{-2p}/(2\phimin^2)$, proving \eqref{eq:B-asymptotic}. After multiplication by $N^p$,
\begin{align*}
    N^pB_N=\frac1{2\Phi(\theta_N)}+O(N^{-p}).
\end{align*}
Density of the integer phases, continuity of $\Phi$, and \eqref{eq:phase-extrema} now yield
\begin{align*}
    \Clust(N^pB_N)
    =\left\{\frac1{2\Phi(t)}:0\le t\le1\right\}
    =\left[\frac12,\frac1{2\phimin}\right].
\end{align*}
\end{proof}

The oscillation occurs in the leading term and cannot be absorbed into the $O(N^{-2p})$ remainder.
This proves the exact formula, uniform expansion, and cluster-set statement in Corollary~\ref{cor:sysmmetic-main-objective}. 

\subsection{Subdivision derivatives and binary-path products}\label{app:proof-subdivision-derivative}

\begin{proposition} \label{prop:subdivision-derivative}
There exists a bounded function $D:[1,2]\to[c_-,c_+]$ such that
\begin{align}\label{eq:F-integral}
    F(x)=1+\int_1^xD(u)\,\mathrm{d} u.
\end{align}
The function $D$ is continuous at every nondyadic point $x\in(1,2)$, and $F'(x)=D(x)$ there. At every interior dyadic point, both one-sided derivatives of $F$ exist.

For a binary path $\epsilon=\{\epsilon_i\}_{i\ge1}\in\{0,1\}^{\mathbb N}$, set $r_0=\rho$ and recursively define
\begin{align}\label{eq:r-bit}
    r_i=
    \begin{cases}
        T_0(r_{i-1}),&\epsilon_i=0,\\
        T_1(r_{i-1}),&\epsilon_i=1.
    \end{cases}
\end{align}
For $k\ge1$, define the partial product
\begin{align*}
    P_k(\epsilon)
    \coloneqq(\rho-1)\prod_{i=1}^{k}w_{\epsilon_i}(r_{i-1}).
\end{align*}
Then $P_k$ converges uniformly on the path space $\{0,1\}^{\mathbb N}$ to
\begin{align}\label{eq:D-product}
    P(\epsilon)\coloneqq
    (\rho-1)\prod_{i=1}^{\infty}w_{\epsilon_i}(r_{i-1})
\end{align}
in the precise sense that $\sup_\epsilon|P_k(\epsilon)-P(\epsilon)|\to0$. More precisely, with $M=1/[4\rho(\rho-1)]$,
\begin{align*}
    \sup_{\epsilon\in\{0,1\}^{\mathbb N}}
    \sum_{i=k+1}^{\infty}
    \left|\log w_{\epsilon_i}(r_{i-1})\right|
    \le4M(\rho-1)2^{-k}.
\end{align*}
If $x\in(1,2)$ is nondyadic and $x-1=0.\epsilon_1\epsilon_2\cdots$, then $D(x)=F'(x)=P(\epsilon)$. If an interior dyadic point $x_0$ has the two binary expansions
\begin{align*}
    x_0-1
    &=0.\epsilon_1\cdots\epsilon_m000\cdots\\
    &=0.\epsilon_1\cdots\epsilon_{m-1}0111\cdots,
    \quad \epsilon_m=1,
\end{align*}
then
\begin{align*}
    F'_+(x_0)&=P(\epsilon_1,\ldots,\epsilon_m,0,0,\ldots),\\
    F'_-(x_0)&=P(\epsilon_1,\ldots,\epsilon_{m-1},0,1,1,\ldots).
\end{align*}
\end{proposition}

\begin{proof}
For step-function bookkeeping, define the half-open level-$k$ cells
\begin{align*}
    I^{\mathrm h}_{k,j}\coloneqq \left[\frac{j}{2^k},\frac{j+1}{2^k}\right), \quad 2^k\le j<2^{k+1}.
\end{align*}
For $x\in I^{\mathrm h}_{k,j}$, set $S_k(x)\coloneqq S_{k,j}$, where $S_{k,j}$ is the slope of $F_k$ on $I_{k,j}$, and set $S_k(2)\coloneqq S_{k,2^{k+1}-1}$. Thus $S_k$ is a step function on $[1,2]$ whose value on each level-$k$ cell is the corresponding slope of $F_k$.

Fix $x\in I^{\mathrm h}_{k,j}$, and let $\delta\in\{0,1\}$ be determined by $x\in I^{\mathrm h}_{k+1,2j+\delta}$. By \eqref{eq:slope-subdivision}, $S_{k+1}(x)=w_\delta(R_j)S_k(x)$. Since $S_k(x)\le c_+$, $\floor{\log_2j}=k$, and \eqref{eq:R-bound} and \eqref{eq:w-errors} give $|w_\delta(R_j)-1|\le M(R_j-1)\le M(\rho-1)2^{-k}$, it follows that
\begin{align*}
    |S_{k+1}(x)-S_k(x)|&=S_k(x)|w_\delta(R_j)-1|\le c_+M(\rho-1)2^{-k}.
\end{align*}
The same estimate holds at $x=2$ by using the rightmost level-$k$ cell. Hence, for $m>k$,
\begin{align*}
    \norm{S_m-S_k}_\infty&\le\sum_{\ell=k}^{m-1}\norm{S_{\ell+1}-S_\ell}_\infty\le c_+M(\rho-1)\sum_{\ell=k}^{m-1}2^{-\ell}\le2c_+M(\rho-1)2^{-k}.
\end{align*}
Thus $\{S_k\}$ is uniformly Cauchy and converges uniformly to a function $D:[1,2]\to[c_-,c_+]$, because every $S_k$ takes values in the closed interval $[c_-,c_+]$.

The function $F_k$ is continuous and piecewise affine, has slope $S_k$ away from its finitely many level-$k$ nodes, and satisfies $F_k(1)=1$. Therefore
\begin{align*}
    F_k(x)=1+\int_1^xS_k(u)\,\mathrm{d} u, \quad 1\le x\le2.
\end{align*}
Uniform convergence of $S_k$ implies uniform convergence of the right-hand side to $G(x)\coloneqq 1+\int_1^xD(u)\,du$. If $x\in\mathcal D_m$, then $F_k(x)=F(x)$ for every $k\ge m$, so $G=F$ on the dense set $\mathcal D$. Since both functions are continuous, $G=F$ on $[1,2]$, proving \eqref{eq:F-integral}.

Now let $x\in(1,2)$ be nondyadic and write $x-1=0.\epsilon_1\epsilon_2\cdots$. Define the indices of its successive containing cells by
\begin{align*}
    j_0\coloneqq 1, \quad j_i\coloneqq 2j_{i-1}+\epsilon_i=2^i+\floor{2^i(x-1)}, \quad i\ge1.
\end{align*}
Then $x\in I^{\mathrm h}_{i,j_i}$ and $I^{\mathrm h}_{i,j_i}$ is the child of $I^{\mathrm h}_{i-1,j_{i-1}}$ selected by $\epsilon_i$. Moreover, $r_0=\rho=R_{j_0}$, and induction using \eqref{eq:R-recursion} and \eqref{eq:r-bit} gives $r_i=R_{j_i}$. Iterating \eqref{eq:slope-subdivision} along these containing cells therefore yields
\begin{align*}
    S_k(x)=(\rho-1)\prod_{i=1}^kw_{\epsilon_i}(r_{i-1}).
\end{align*}
Letting $k\to\infty$ and using $S_k(x)\to D(x)$ proves \eqref{eq:D-product}.

The same construction applies to any infinite binary path $\{\epsilon_i\}_{i\ge1}$. Since $j_{i-1}$ is a level-$(i-1)$ index, \eqref{eq:R-bound} and \eqref{eq:w-errors} give
\begin{align*}
    \left|w_{\epsilon_i}(r_{i-1})-1\right|\le M(r_{i-1}-1)\le M(\rho-1)2^{-(i-1)}=:a_i.
\end{align*}
Here $a_i\le a_1=M(\rho-1)=1/(4\rho)<1/2$. The elementary estimate $|\log(1+t)|\le2|t|$ for $|t|\le1/2$ consequently gives, uniformly over all binary paths,
\begin{align*}
    \sum_{i=k+1}^{\infty}\left|\log w_{\epsilon_i}(r_{i-1})\right|\le2M(\rho-1)\sum_{i=k+1}^{\infty}2^{-(i-1)}=4M(\rho-1)2^{-k}.
\end{align*}
This path-independent tail bound is precisely uniform convergence of the logarithms of the partial products. Since all partial products lie in the common compact interval $[c_-,c_+]$, exponentiation gives $\sup_\epsilon|P_k(\epsilon)-P(\epsilon)|\to0$. Thus the uniformity is with respect to the binary-path parameter indexing the family of products.

To prove continuity, use for every $y\in[1,2)$ its canonical binary expansion, namely the expansion that does not end in recurring ones. For dyadic $y$, this is the terminating expansion, and the preceding slope-product argument still gives $D(y)$ as the corresponding infinite product. Fix a nondyadic $x$ and an integer $k$. Since $x$ lies in the interior of its half-open level-$k$ cell, all sufficiently close $y$ lie in the same cell and hence share the first $k$ binary digits with $x$. The preceding tail estimate then gives
\begin{align*}
    |\log D(y)-\log D(x)|\le8M(\rho-1)2^{-k}.
\end{align*}
First choosing $k$ large and then $y$ sufficiently close to $x$ proves that $D$ is continuous at $x$. The integral representation \eqref{eq:F-integral} and the fundamental theorem of calculus now yield $F'(x)=D(x)$.

Finally, let $x_0\in(1,2)$ be dyadic. Its two binary expansions have the form
\begin{align*}
    x_0-1=0.\epsilon_1\cdots\epsilon_m000\cdots, \quad \epsilon_m=1,\quad x_0-1=0.\epsilon_1\cdots\epsilon_{m-1}0111\cdots.
\end{align*}
The uniform logarithmic tail estimate shows that the infinite products associated with these two paths are well defined; denote them by $D_+(x_0)$ and $D_-(x_0)$, respectively. For every $n\ge m$, the level-$n$ cells selected by the first $n$ digits of these paths are $[x_0,x_0+2^{-n})$ and $[x_0-2^{-n},x_0)$, respectively. The common-prefix argument therefore gives
\begin{align*}
    \lim_{x\downarrow x_0}D(x)=D_+(x_0), \quad \lim_{x\uparrow x_0}D(x)=D_-(x_0).
\end{align*}
For $h>0$, \eqref{eq:F-integral} implies
\begin{align*}
    \frac{F(x_0+h)-F(x_0)}h&=\frac1h\int_{x_0}^{x_0+h}D(u)\,\mathrm{d} u,\quad \frac{F(x_0)-F(x_0-h)}h=\frac1h\int_{x_0-h}^{x_0}D(u)\,\mathrm{d} u.
\end{align*}
Passing to the limit $h\downarrow0$ gives $F'_+(x_0)=D_+(x_0)$ and $F'_-(x_0)=D_-(x_0)$. Hence the terminating expansion gives the right derivative and the expansion ending in recurring ones gives the left derivative.
\end{proof}

\subsection{A certified dyadic corner}\label{app:proof-strict-corner}

\begin{corollary} \label{cor:strict-corner}
The two one-sided derivatives at $x=3/2$ satisfy
\begin{align}\label{eq:strict-corner}
    F'_-(3/2)>F'_+(3/2).
\end{align}
Consequently, $F$ is not differentiable at $3/2$ and is not continuously differentiable on $[1,2]$.
\end{corollary}

\begin{proof}
The two binary expansions
\begin{align*}
    \frac32-1=\frac12=0.0111\cdots=0.1000\cdots
\end{align*}
and \eqref{eq:D-product} give
\begin{align*}
    F'_-(3/2) &=(\rho-1)w_0(\rho) \prod_{i\ge0}w_1\bigl(T_1^i(T_0(\rho))\bigr),\\
    F'_+(3/2) &=(\rho-1)w_1(\rho) \prod_{i\ge0}w_0\bigl(T_0^i(T_1(\rho))\bigr).
\end{align*}
For $m\ge0$, set
\begin{align*}
    x_m&\coloneqq T_1^m(T_0(\rho)), \quad P_m^-\coloneqq (\rho-1)w_0(\rho)\prod_{i=0}^{m-1}w_1(x_i),\\
    y_m&\coloneqq T_0^m(T_1(\rho)), \quad P_m^+\coloneqq (\rho-1)w_1(\rho)\prod_{i=0}^{m-1}w_0(y_i),
\end{align*}
where an empty product equals one. Thus $P_6^-$ and $P_6^+$ contain the initial factor and exactly six tail factors.

We now give a rational interval certificate for these two finite products. For positive rational intervals $X=[a,b]$ and $Y=[c,d]$, with $c>0$, use
\begin{gather*}
    X+Y=[a+c,b+d], \quad X-Y=[a-d,b-c],\\
    XY=[ac,bd], \quad \frac XY=\left[\frac ad,\frac bc\right].
\end{gather*}
An enclosure $\sqrt X\subseteq[s_-,s_+]$ is certified by the rational inequalities $s_-^2\le a\le b\le s_+^2$. For $G\in\{T_0,T_1,w_0,w_1\}$, let $G^\square(X)$ denote the interval obtained by applying these rules to the explicit formulas in \eqref{eq:AB-explicit} and \eqref{eq:w-def}; every square-root endpoint is chosen by the preceding squaring test. Finally, let $\langle X\rangle_{12}$ be the smallest interval with endpoints in $10^{-12}\mathbb Z$ that contains $X$.

Start from
\begin{align*}
    \boldsymbol\rho\coloneqq \left[2.414213562373095048801688724209,2.414213562373095048801688724210\right],
\end{align*}
which contains $\rho=1+\sqrt2$ because the squares of the two endpoints after subtracting one lie below and above $2$, respectively. Define
\begin{align*}
    \mathbf x_0&\coloneqq \left\langle T_0^\square(\boldsymbol\rho)\right\rangle_{12}, \quad \mathbf P_0^-\coloneqq \left\langle(\boldsymbol\rho-1)w_0^\square(\boldsymbol\rho)\right\rangle_{12},\\
    \mathbf y_0&\coloneqq \left\langle T_1^\square(\boldsymbol\rho)\right\rangle_{12}, \quad \mathbf P_0^+\coloneqq \left\langle(\boldsymbol\rho-1)w_1^\square(\boldsymbol\rho)\right\rangle_{12},
\end{align*}
and, for $0\le m<6$, recursively set
\begin{align*}
    \mathbf x_{m+1}&\coloneqq \left\langle T_1^\square(\mathbf x_m)\right\rangle_{12}, \quad \mathbf P_{m+1}^-\coloneqq \left\langle\mathbf P_m^-w_1^\square(\mathbf x_m)\right\rangle_{12},\\
    \mathbf y_{m+1}&\coloneqq \left\langle T_0^\square(\mathbf y_m)\right\rangle_{12}, \quad \mathbf P_{m+1}^+\coloneqq \left\langle\mathbf P_m^+w_0^\square(\mathbf y_m)\right\rangle_{12}.
\end{align*}
Because rational interval operations preserve inclusion, induction gives $x_m\in\mathbf x_m$, $y_m\in\mathbf y_m$, $P_m^-\in\mathbf P_m^-$, and $P_m^+\in\mathbf P_m^+$. Six exact rational interval iterations give the terminal enclosures
\begin{align*}
 x_6&\in[1.006713216536,1.006713216564],&
 P_6^-&\in[1.419687938587,1.419687949396],\\
 y_6&\in[1.006668059132,1.006668059134],&
 P_6^+&\in[1.419604771595,1.419604772437].
\end{align*}
In particular,
\begin{align*}
 1.4196879<P_6^-<1.4196880,
 \quad
 1.4196047<P_6^+<1.4196048.
\end{align*}
By \eqref{eq:w-errors}, every omitted factor in the left product is at least one, whereas every omitted factor in the right product is at most one. Therefore
\begin{align*}
    F'_-(3/2)\ge P_6^->P_6^+\ge F'_+(3/2),
\end{align*}
which proves \eqref{eq:strict-corner}.
\end{proof}

\subsection{Nested certificates for the phase minimum}\label{app:certificates}

The qualitative phase theorem is established analytically. We next give rigorously convergent finite-grid certificates for its minimum. For $k\ge0$, define
\begin{align}\label{eq:phase-cert-def}
    \underline\phi_k
    &\coloneqq
    \min_{2^k\le j<2^{k+1}}\frac{U_j}{(j+1)^p},\\
    \overline\phi_k
    &\coloneqq
    \min_{2^k\le j\le2^{k+1}}\frac{U_j}{j^p}.\notag
\end{align}

\begin{proposition} \label{prop:phase-cert}
For every $k\ge0$,
\begin{align*}
    \underline\phi_k\le\phimin\le\overline\phi_k,
    \quad
    0\le\overline\phi_k-\underline\phi_k\le\frac{p}{2^k}.
\end{align*}
The sequence $\{\underline\phi_k\}$ is nondecreasing, the sequence $\{\overline\phi_k\}$ is nonincreasing, and both converge to $\phimin$. All quantities through level $k$ are computable in $O(2^k)$ scalar operations.
\end{proposition}

\begin{proof}
Writing $x=2^t$ gives
\begin{align*}
    \phimin=\min_{x\in[1,2]}\frac{F(x)}{x^p}.
\end{align*}
If $x\in I_{k,j}$, monotonicity of $F$ and \eqref{eq:F-dyadic} imply
\begin{align*}
    \frac{F(x)}{x^p}
    \ge\frac{F(j/2^k)}{((j+1)/2^k)^p}
    =\frac{U_j}{(j+1)^p},
\end{align*}
which gives the lower certificate. Minimizing $F(x)/x^p$ over the level-$k$ grid gives the upper certificate.

For each parent index $j$, the lower candidates of its two children satisfy
\begin{align*}
    \frac{U_{2j}}{(2j+1)^p}
    &>\frac{U_j}{(j+1)^p},\\
    \frac{U_{2j+1}}{(2j+2)^p}
    &>\frac{\rho U_j}{\rho(j+1)^p}
    =\frac{U_j}{(j+1)^p},
\end{align*}
where we used \eqref{eq:dyadic-scaling} and $U_{2j+1}=\Kop(U_j,U_{j+1})>\rho U_j$. Hence $\underline\phi_{k+1}\ge\underline\phi_k$. Since the level-$k$ grid is contained in the level-$(k+1)$ grid, $\overline\phi_{k+1}\le\overline\phi_k$.

Choose $j_k$ attaining $\underline\phi_k$. By Proposition \ref{prop:sharp-growth} and Bernoulli's inequality,
\begin{align*}
    0\le\overline\phi_k-\underline\phi_k
    &\le\frac{U_{j_k}}{j_k^p}
    \left[1-\left(\frac{j_k}{j_k+1}\right)^p\right]\\
    &\le\frac{p}{j_k+1}
    \le\frac{p}{2^k}.
\end{align*}
The nested bounds and vanishing gap force both endpoints to converge to $\phimin$. Finally, iterating the balanced recurrence computes all required $U_j$ through $2^{k+1}$ in $O(2^k)$ scalar operations; scanning the two finite candidate sets has the same order.
\end{proof}

At level $k=17$, the minimizing indices in \eqref{eq:phase-cert-def} and directed outward intervals are
\begin{center}
\begin{tabular}{ccc}
\toprule
Quantity & Index & Outward interval\\
\midrule
$\underline\phi_{17}$ & $181314$ & $(0.9928123698,\ 0.9928123700)$\\
$\overline\phi_{17}$ & $181324$ & $(0.9928193322,\ 0.9928193324)$\\
\bottomrule
\end{tabular}
\end{center}
Enlarging each side by $10^{-8}$ gives the certified enclosure
\begin{align}\label{eq:phase-numerical}
    0.99281236<\phimin<0.99281934.
\end{align}
The certificate requires only the balanced recurrence through $2^{18}$ and the exact finite minima in \eqref{eq:phase-cert-def}. The displayed level-17 decimal enclosure is certified by interval arithmetic, with details provided in the supplementary material.


\subsection{Marginal increments and fixed dyadic rays}\label{app:phase-consequences}

The derivative products also determine the first-order gain from adding one leaf at a fixed dyadic mantissa. If $x\in(1,2)$ is nondyadic and $N_k=\floor{2^kx}$, then
\begin{align}\label{eq:marginal-U}
    \lim_{k\to\infty}N_k^{1-p}
    \bigl(U_{N_k+1}-U_{N_k}\bigr)
    =x^{1-p}D(x).
\end{align}
Indeed, the slope of the level-$k$ cell containing $x$ is
\begin{align*}
    2^k\rho^{-k}\bigl(U_{N_k+1}-U_{N_k}\bigr).
\end{align*}
It converges to $D(x)$ by Proposition \ref{prop:subdivision-derivative}; multiplying by $(N_k/2^k)^{1-p}$ and using $N_k/2^k\to x$ proves \eqref{eq:marginal-U}. At a dyadic $x$, the same argument along the right or left adjacent cells replaces $D(x)$ by $F'_+(x)$ or $F'_-(x)$, respectively.

Exact scaling separately shows that a fixed finite seed selects a phase but cannot change the exponent. For every $m\ge1$ and $k\ge0$,
\begin{align*}
    \frac{U_{m2^k}}{(m2^k)^p}=\frac{U_m}{m^p}, \quad (m2^k)^ps_{m2^k} = m^ps_m.
\end{align*}
More generally, a seed with $m$ leaves and reciprocal scalar value $W>0$ has value $\rho^kW$ after $k$ symmetric self-compositions, so its normalized scalar and rate constants remain $W/m^p$ and $m^p/W$. For example,
\begin{align*}
    (2^k)^ps_{2^k}=1,
    \quad
    (3\cdot2^k)^ps_{3\cdot2^k}
    =\frac{3^p}{U_3}\approx1.006818956.
\end{align*}

\section{Asymmetric phase: proofs and strict support separation}\label{app:obsf-details}
The proof of Theorem~\ref{thm:asymmetric-main} has three parts: comparison with the
symmetric recursion, uniform convergence of the phase interpolants, and a support
deficit argument excluding both a constant phase and contact with the support bound.

\subsection{Reciprocal comparison and fixed-ray limits}\label{app:obsf-comparison}
The analysis of \citet[Lemma~10 and Theorem~5]{GrimmerShuWangMOR} also gives, for an explicit $c_{\mathrm{low}}>0$,
\begin{align}
    \eta^{\mathrm{F}}_{2N}\le\rho^{-1}\eta^{\mathrm{F}}_N, \quad \eta^{\mathrm{F}}_N\ge c_{\mathrm{low}}N^{-p}.
    \label{eq:obsf-imported}
\end{align}

Reciprocating \eqref{eq:obsf-dp} exposes the Bellman structure. By \eqref{eq:conpd-obsf-rate}, write $W_N\coloneqq (\eta^{\mathrm{F}}_N)^{-1}=2\one^\top h_\bullet^{(N-1)}+1$. Then
\begin{align*}
    W_1=1,
    \quad
    W_N=\max_{1\le m<N}\mathcal A(U_m,W_{N-m}),
\end{align*}
where
\begin{align*}
    \mathcal A(u,w)
    \coloneqq\frac{4u+w+\sqrt{w^2+8uw}}{2}.
\end{align*}
The map $\mathcal A$ is continuous, increasing, and positively homogeneous, but asymmetric in its two arguments: $U_m$ represents the optimized ConPP/OBS-S prefix, whereas $W_{N-m}$ represents the optimized ConPD/OBS-F continuation. Thus \eqref{eq:obsf-reciprocal-dp} is an ordered Bellman recursion, and the balanced rearrangement argument of Section \ref{sec:balanced} does not apply.

Before comparing the two recursions, note a simple consequence of the imported estimates. For each fixed $m\ge1$, the sequence $\{(m2^k)^p\eta^{\mathrm{F}}_{m2^k}\}_k$ is nonincreasing and bounded below by \eqref{eq:obsf-imported}, and hence converges to a positive limit, which is denoted by $C_m^F$. Shifting the ray index gives $C_{2m}^F=C_m^F$. We use these limits to identify the phase below.
\subsubsection{Local comparisons for the reciprocal sandwich}\label{app:obsf-local-comparisons}

\begin{lemma} \label{lem:obsf-local-comparisons}
For $x,y>0$,
\begin{align}
    \mathcal A(x,\rho y)&\le\rho\Kop(x,y),
    \label{eq:obsf-upper-local}\\
    \mathcal A(x,2y-1)&\ge2\Kop(x,y)-1,
    \quad x,y\ge1.
    \label{eq:obsf-lower-local}
\end{align}
\end{lemma}

\begin{proof}
Let $k=\Kop(x,y)$, so
\begin{align}\label{eq:obsf-K-quadratic-app}
    k^2=(x+y)k+xy.
\end{align}
For $x,y>0$, the value $\mathcal A(x,\rho y)$ is the larger root of
\begin{align*}
    Q(z)=z^2-(4x+\rho y)z+4x^2.
\end{align*}
Using \eqref{eq:obsf-K-quadratic-app} and $\rho^2=2\rho+1$ gives
\begin{align*}
    Q(\rho k)=x\bigl(4x+\rho^2y-(2\rho-1)k\bigr).
\end{align*}
Put $\theta=x/y$ and $s=\sqrt2$. After substituting the explicit formula for $k$, the nonnegativity of the factor in parentheses is equivalent to
\begin{align*}
    (7-2s)\theta+5+2s
    \ge(1+2s)\sqrt{\theta^2+6\theta+1}.
\end{align*}
Both sides are positive, and the left-hand side squared minus the right-hand side squared equals
\begin{align*}
    8(6-4s)\left(\theta-2-\frac{3s}{2}\right)^2\ge0.
\end{align*}
Hence $Q(\rho k)\ge0$. The smaller root of $Q$ is below $2x$, while $\rho k>2x$; thus $\rho k$ lies to the right of the larger root. This proves \eqref{eq:obsf-upper-local}.

For \eqref{eq:obsf-lower-local}, $\mathcal A(x,2y-1)$ is the larger root of
\begin{align*}
    \widetilde Q(z)=z^2-(4x+2y-1)z+4x^2.
\end{align*}
At $z=2k-1$, reduction by \eqref{eq:obsf-K-quadratic-app} gives the equivalent inequality
\begin{align*}
    (2x+1)\sqrt{x^2+6xy+y^2}
    \ge2x^2+2xy+3x+y.
\end{align*}
The difference of squares is $8x^2(2xy-x+y-1)\ge0$ for $x,y\ge1$. Therefore $\widetilde Q(2k-1)\le0$. Since $2k-1>2(x+y)-1>2x$ lies to the right of the smaller root, it does not exceed the larger root, proving \eqref{eq:obsf-lower-local}.
\end{proof}
\begin{proposition} \label{thm:obsf-sandwich}
For every $N\ge1$,
\begin{align*}
    2U_N-1\le W_N\le\rho U_N.
\end{align*}
Consequently, if $t=\{\log_2m\}$, then
\begin{align*}
    \frac{1}{\rho\Phi(t)}
    \le C_m^F
    \le\frac{1}{2\Phi(t)}.
\end{align*}
\end{proposition}

\begin{proof}
The local comparisons established in Lemma \ref{lem:obsf-local-comparisons} are
\begin{align*}
    \mathcal A(x,\rho y)&\le\rho\Kop(x,y),
    \quad x,y>0,\\
    \mathcal A(x,2y-1)&\ge2\Kop(x,y)-1,
    \quad x,y\ge1.
\end{align*}
Assume \eqref{eq:obsf-sandwich} below $N$. For every split,
\begin{align*}
    \mathcal A(U_m,W_{N-m})
    \le\mathcal A(U_m,\rho U_{N-m})
    \le\rho\Kop(U_m,U_{N-m}).
\end{align*}
Maximizing and using \eqref{eq:U-DP} gives $W_N\le\rho U_N$. For the lower bound, choose a balanced optimal split $N=a+b$. Since $U_N=\Kop(U_a,U_b)$,
\begin{align*}
    W_N
    \ge\mathcal A(U_a,W_b)
    \ge\mathcal A(U_a,2U_b-1)
    \ge2\Kop(U_a,U_b)-1
    =2U_N-1.
\end{align*}
The base case is immediate. Applying the sandwich along $N=m2^k$, dividing by $(m2^k)^p$, and using the exact phase law for $U_N$ gives the stated bounds on $C_m^F$.
\end{proof}
\subsection{Discrete modulus and interpolants}\label{app:obsf-discrete-modulus}
The fixed-ray limits do not by themselves define a continuous phase. One also needs uniform control of neighboring normalized values. Define
\begin{align*}
    D_N\coloneqq\frac{W_N}{N^p},
    \quad
    C_N\coloneqq N^p\eta^{\mathrm{F}}_N=\frac1{D_N},
\end{align*}
and recall that $\mathbb T=\R/\mathbb Z$ and $\phi_* = \min_{t\in\mathbb T}\Phi(t)>0$. Recall that $L_F$ is a Lipschitz constant for the canonical interpolation $F$ and set $L_U\coloneqq\max\{1,L_F\}$. The discrete estimates proved in Lemma \ref{lem:obsf-discrete-modulus} give
\begin{align*}
    0<U_m-U_{m-1}\le L_Um^{p-1},\quad
    0<W_{N+1}-W_N\le4L_UN^{p-1},\quad
    |C_{N+1}-C_N|\le\frac{L_{\mathrm{obs}}}{N},
\end{align*}
where $L_{\mathrm{obs}}\coloneqq(4L_U+\rho p)/\phi_*^2$. Since adjacent logarithmic phases are separated by $\log_2(1+1/N)\asymp N^{-1}$, the preceding estimate yields a uniform bound on the slopes of the phase interpolants.

For each $k\ge0$ and $0\le j\le2^k$, define
\begin{align*}
    t_{k,j}\coloneqq\log_2\frac{2^k+j}{2^k}.
\end{align*}
Let $\Psi_k:[0,1]\to\R$ be the unique continuous function satisfying
\begin{align*}
    \Psi_k(t_{k,j})=C_{2^k+j},
    \quad 0\le j\le2^k,
\end{align*}
and affine on each interval $[t_{k,j},t_{k,j+1}]$.
\begin{lemma} \label{lem:obsf-discrete-modulus}
Set $L_U\coloneqq\max\{1,L_F\}$, and put $\phi_*\coloneqq\min\Phi$. With $U_0=0$, for every $m,N\ge1$,
\begin{align*}
    0<U_m-U_{m-1}&\le L_Um^{p-1},\\
    0<W_{N+1}-W_N&\le4L_UN^{p-1},\\
    |C_{N+1}-C_N|&\le\frac{4L_U+\rho p}{\phi_*^2N}.
\end{align*}
\end{lemma}

\begin{proof}
The first estimate is immediate for $m=1$. If $m\ge2$, put $r=\floor{\log_2(m-1)}$. Exact interpolation gives
\begin{align*}
    U_j=\rho^rF(j/2^r),
    \quad j=m-1,m,
\end{align*}
and hence
\begin{align*}
    U_m-U_{m-1}
    \le L_F\rho^r2^{-r}
    =L_F2^{r(p-1)}
    \le L_Um^{p-1}.
\end{align*}
Strict positivity follows from Corollary \ref{cor:balanced-optimality}.

The join satisfies $\mathcal A(0,w)=w$ and
\begin{align*}
    \partial_u\mathcal A(u,w)
    =2+\frac{2w}{\sqrt{w^2+8uw}}
    \le4.
\end{align*}
The split $m=1$ in \eqref{eq:obsf-reciprocal-dp} gives $W_{N+1}>W_N$. Choose a maximizing pivot $m\in\{1,\ldots,N\}$ for $W_{N+1}$. If $m\ge2$, the split $(m-1,N+1-m)$ is admissible for $W_N$; if $m=1$, use $\mathcal A(U_0,W_N)=W_N$. Therefore
\begin{align*}
    W_{N+1}-W_N
    &\le\mathcal A(U_m,W_{N+1-m})
    -\mathcal A(U_{m-1},W_{N+1-m})\\
    &\le4(U_m-U_{m-1})
    \le4L_UN^{p-1}.
\end{align*}

The reciprocal sandwich and $\phi_*\le U_N/N^p\le1$ give $\phi_*\le D_N\le\rho$. Thus
\begin{align*}
    |D_{N+1}-D_N|
    &\le\frac{W_{N+1}-W_N}{(N+1)^p}
    +W_N\left(\frac1{N^p}-\frac1{(N+1)^p}\right)\\
    &\le\frac{4L_U}{N}
    +\rho\left[1-\left(\frac{N}{N+1}\right)^p\right]
    \le\frac{4L_U+\rho p}{N}.
\end{align*}
Since $C_N=D_N^{-1}$ and $D_N,D_{N+1}\ge\phi_*$, the final estimate follows.
\end{proof}
\subsection{Uniform phase convergence}\label{app:obsf-uniform-proof}
\begin{theorem} \label{thm:obsf-continuous-phase}
There exists a unique positive Lipschitz function $\Psi_F:\mathbb T\to(0,\infty)$ such that the interpolation functions $\Psi_k$ converge uniformly to $\Psi_F$ on $[0,1]$:
\begin{align*}
    \norm{\Psi_k-\Psi_F}_{L^\infty([0,1])}
    \to0.
\end{align*}
Equivalently, because the maximal interpolation mesh tends to zero and $\Psi_F$ is Lipschitz,
\begin{align*}
    \sup_{2^k\le N\le2^{k+1}}
    \left|N^p\eta^{\mathrm{F}}_N-\Psi_F(\{\log_2N\})\right|
    \to0.
\end{align*}
Thus
\begin{align*}
    \eta^{\mathrm{F}}_N=N^{-p}\Psi_F(\{\log_2N\})+o(N^{-p})
\end{align*}
as $N \to \infty$.
Moreover,
\begin{align*}
    \Clust(N^p\eta^{\mathrm{F}}_N)
    =\Psi_F(\mathbb T)
    =[\min\Psi_F,\max\Psi_F].
\end{align*}
Thus \eqref{eq:obsf-cluster-set} is the complete cluster interval of the normalized tight objective coefficient in \eqref{eq:obsf-objective-bound}. Combining this expansion with \eqref{eq:obsf-objective-bound} gives the following guarantee for the optimized ConPD/OBS-F schedules

\begin{align}\label{eq:obsf-objective-phase}
    f(x_{N-1})-f_*
    \le
    \left[N^{-p}\Psi_F(\{\log_2N\})+o(N^{-p})\right]
    \frac{L}{2}\norm{x_0-x_*}^2
\end{align}
where the $o(N^{-p})$ remainder has the same dyadic-block uniformity. One may take $2L_{\mathrm{obs}}\log 2$ as a Lipschitz constant for $\Psi_F$.
\end{theorem}

\begin{proof}
For adjacent interpolation nodes,
\begin{align*}
    \log_2\frac{N+1}{N}
    =\frac{\log(1+1/N)}{\log 2}
    \ge\frac{1}{(N+1)\log 2}.
\end{align*}
The discrete modulus therefore makes $\{\Psi_k\}$ equi-Lipschitz. The reciprocal sandwich gives common positive upper and lower bounds, so Arzel\`a--Ascoli yields subsequences converging uniformly on $[0,1]$.

Every subsequential limit has the same values on dyadic mantissas. Indeed, if $x=j/2^\ell\in[1,2]$ and $k\ge\ell$, then
\begin{align*}
    \Psi_k(\log_2x)=C_{j2^{k-\ell}},
\end{align*}
which converges to the fixed-ray limit $C_j^F$. Hence all uniformly convergent subsequences have limits agreeing on a dense set and therefore, by continuity, everywhere. Since the family is relatively compact and has only one possible limit, the full function sequence $\{\Psi_k\}$ converges uniformly on $[0,1]$. The identity $C_{2m}^F=C_m^F$ matches the endpoint values and gives a function on $\mathbb T$, proving \eqref{eq:obsf-uniform-phase} and uniqueness.

Compactness of the phase circle shows that every cluster point of $C_N$ lies in $\Psi_F(\mathbb T)$. Conversely, $N_k=\floor{2^{k+t}}$ approaches any prescribed phase $t$, so $C_{N_k}\to\Psi_F(t)$. Since the continuous image of the circle is a compact interval, \eqref{eq:obsf-cluster-set} follows.
\end{proof}
\subsection{Sharp homogeneous support}\label{app:obsf-support-proof}
\begin{lemma} \label{lem:obsf-support}
Let $\xi_*\in(0,1)$ be the unique interior solution of
\begin{align}\label{eq:obsf-xi-equation}
    \xi^{2q-1}(2-\xi)=1,
\end{align}
and define
\begin{align*}
    \tau_*\coloneqq\frac{1-\xi_*}{2\xi_*^2},
    \quad
    B_{\mathrm{sup}}\coloneqq2^q\frac{1-\xi_*^{2q}}{(1-\xi_*)^q}.
\end{align*}
Then $B_{\mathrm{sup}}>2^q$ and, for all $u,w\ge0$,
\begin{align}\label{eq:obsf-support}
    \mathcal A(u,w)^q\le B_{\mathrm{sup}}u^q+w^q.
\end{align}
For $u,w>0$, equality holds if and only if $u/w=\tau_*$. Equality also holds when $u=0$, whereas the inequality is strict when $w=0<u$.
\end{lemma}
\begin{proof}
By homogeneity, normalize $w=1$ and write $t=u>0$. Set
\begin{align*}
    z=\frac{1+\sqrt{1+8t}}{2}.
\end{align*}
Then $z^2-z-2t=0$, so $t=z(z-1)/2$ and $\mathcal A(t,1)=z^2$. With $x=1/z\in(0,1)$, the smallest admissible support coefficient is the supremum of
\begin{align*}
    h(x)
    \coloneqq\frac{\mathcal A(t,1)^q-1}{t^q}
    =2^q\frac{1-x^{2q}}{(1-x)^q}.
\end{align*}
Its logarithmic derivative is
\begin{align*}
    \frac{h'(x)}{qh(x)}
    =\frac{1}{1-x}-\frac{2x^{2q-1}}{1-x^{2q}},
\end{align*}
whose sign is that of $1-x^{2q-1}(2-x)$. Put $a=2q-1\in(0,1)$ and $g(x)=x^a(2-x)$. Since
\begin{align*}
    g'(x)=x^{a-1}\bigl(2a-(a+1)x\bigr),
\end{align*}
$g$ first increases and then decreases on $(0,1)$; moreover, $g(0)=0$, $g(1)=1$, and $g'(1)=a-1<0$. Thus $g(x)=1$ has exactly one interior solution $\xi_*$ in addition to the endpoint $1$, and $h$ has its unique maximum at $\xi_*$. The endpoint limits are
\begin{align*}
    \lim_{x\downarrow0}h(x)=2^q,
    \quad
    \lim_{x\uparrow1}h(x)=0,
\end{align*}
so $B_{\mathrm{sup}}=h(\xi_*)>2^q$. Back-substitution gives the unique positive equality ratio
\begin{align*}
    t=\frac{1-\xi_*}{2\xi_*^2}=\tau_*.
\end{align*}
Homogeneity yields the stated condition $u/w=\tau_*$. The boundary cases follow from $\mathcal A(0,w)=w$ and $\mathcal A(u,0)=2u$, together with $B_{\mathrm{sup}}>2^q$.
\end{proof}
Set
\begin{align}
    d_*\coloneqq B_{\mathrm{sup}}^p,
    \quad
    c_*\coloneqq d_*^{-1}=B_{\mathrm{sup}}^{-p},
    \quad
    r_*\coloneqq\tau_*^q,
    \quad
    \alpha_*\coloneqq\frac{B_{\mathrm{sup}}r_*}{1+B_{\mathrm{sup}}r_*}.
    \label{eq:star}
\end{align}
Numerically, $B_{\mathrm{sup}}\approx1.97530507031103$ and $c_*\approx0.420809352823451$. The support inequality \eqref{eq:obsf-support} recovers the lower barrier $C_N\ge c_*=c_{\mathrm{low}}$ established by \citet[Theorem~5]{GrimmerShuWangMOR}. Below we will establish a strict separation $c_*<\min_{t\in\mathbb T}\Psi_F(t)$, obtained from the deficit and phase-rigidity arguments. The exact deficit identity in Proposition \ref{prop:obsf-deficit} then shows what asymptotic contact with this barrier would require: along an optimizing right spine, every macroscopic pivot must simultaneously saturate the symmetric power bound and the support inequality. The latter selects the unique fraction $\alpha_*$, while the former forces a zero symmetric phase. The technical concentration and phase-circle arguments implementing this rigidity are collected in Appendix \ref{app:obsf-deficit-concentration} and Appendix \ref{app:obsf-global-separation}.
\subsection{Support deficit and contact geometry}\label{app:obsf-contact-location}

\begin{proposition} \label{prop:obsf-deficit}
For every $N\ge1$,
\begin{align*}
    W_N^q\le B_{\mathrm{sup}}N,
    \quad
    C_N=N^p\eta^{\mathrm{F}}_N\ge c_*.
\end{align*}
Define $E_N\coloneqq B_{\mathrm{sup}}N-W_N^q\ge0$. If $m$ is any maximizing pivot for $W_N$ and $r=N-m$, then
\begin{align}\label{eq:obsf-deficit-decomposition}
    E_N=B_{\mathrm{sup}}(m-U_m^q)+E_r+\sigma(U_m,W_r),
\end{align}
where
\begin{align*}
    \sigma(u,w)\coloneqq B_{\mathrm{sup}}u^q+w^q-\mathcal A(u,w)^q,
\end{align*}
and all three terms on the right-hand side of \eqref{eq:obsf-deficit-decomposition} are nonnegative.
\end{proposition}

\begin{proof}
The sharp symmetric power bound gives $U_m^q\le m$. Assuming $W_j^q\le B_{\mathrm{sup}}j$ below $N$, Lemma \ref{lem:obsf-support} gives, for every split $N=m+r$,
\begin{align*}
    \mathcal A(U_m,W_r)^q
    \le B_{\mathrm{sup}}U_m^q+W_r^q
    \le B_{\mathrm{sup}}m+B_{\mathrm{sup}}r=B_{\mathrm{sup}}N.
\end{align*}
Maximizing proves $W_N^q\le B_{\mathrm{sup}}N$; since $pq=1$, this is equivalent to $C_N\ge B_{\mathrm{sup}}^{-p}=c_*$. For a maximizing split, adding and subtracting $B_{\mathrm{sup}}U_m^q+W_r^q$ gives \eqref{eq:obsf-deficit-decomposition}.
\end{proof}

\begin{lemma} \label{lem:obsf-strict-S-phase}
For every $x\in(1,2)$,
\begin{align*}
    F(x)<x^p.
\end{align*}
Equivalently, $\Phi(t)<1$ for every nonzero phase $t\in(0,1)$.
\end{lemma}

\begin{proof}
Let an interior dyadic point $x$ first appear as the midpoint of a dyadic interval $[a,b]$, and set $H=\rho^{-1}\Kop$. The midpoint rule, endpoint power bounds, and monotonicity give
\begin{align*}
    F(x)=H(F(a),F(b))\le H(a^p,b^p)<\rho^{-1}(a+b)^p=x^p,
\end{align*}
where strictness follows from $a<b$ and Lemma \ref{lem:silver-power}. For nondyadic $x$, choose a sufficiently fine dyadic interval $[a,b]$ containing $x$ with $1<a<b<2$. With
\begin{align*}
    M=\max\left\{F(a)/a^p,F(b)/b^p\right\}<1,
\end{align*}
homogeneity and the same midpoint argument propagate $F(y)\le My^p$ to all dyadic descendants. Density and continuity yield $F(x)\le Mx^p<x^p$.
\end{proof}

\begin{lemma} \label{lem:obsf-contact-fraction}
Write $r_*$ and $\alpha_*$ as in \eqref{eq:star}. Set $x_*\coloneqq2\alpha_*$. The contact fraction satisfies
\begin{align}\label{eq:obsf-contact-identity}
    \alpha_*=1-\xi_*^{2q},
    \quad
    1-\alpha_*=\xi_*^{2q},
\end{align}
and
\begin{align}\label{eq:obsf-contact-interval}
    \frac14<1-\alpha_*<\frac12.
\end{align}
In particular, $1/2<\alpha_*<1$ and $x_*\in(1,2)$.
\end{lemma}

\begin{proof}
The definitions of $B_{\mathrm{sup}}$ and $\tau_*$ give
\begin{align*}
    B_{\mathrm{sup}}\tau_*^q
    =2^q\frac{1-\xi_*^{2q}}{(1-\xi_*)^q}
    \frac{(1-\xi_*)^q}{2^q\xi_*^{2q}}
    =\frac{1-\xi_*^{2q}}{\xi_*^{2q}}.
\end{align*}
Since $r_*=\tau_*^q$ and $\alpha_*=B_{\mathrm{sup}}r_*/(1+B_{\mathrm{sup}}r_*)$, this proves \eqref{eq:obsf-contact-identity}.

We next locate $\xi_*$. The inequalities $q<15/19$ and $2^{30}<3^{19}$ imply $2^{2q}<3$. Indeed, $q<15/19$ is equivalent to $\rho^{15}>2^{19}$, and
\begin{align*}
    \rho^{15}=275807+195025\sqrt2>2^{19}.
\end{align*}
Consequently,
\begin{align*}
    (1/2)^{2q-1}(2-1/2)=3/2^{2q}>1.
\end{align*}
Because $\xi_*$ is the first interior solution of \eqref{eq:obsf-xi-equation}, $\xi_*<1/2$. Since $2q>1$,
\begin{align*}
    1-\alpha_*=\xi_*^{2q}<\xi_*<\frac12.
\end{align*}

For the other direction, $q>11/14$ and $q<4/5$, equivalently $\rho^{11}<2^{14}$ and $\rho^4>2^5$. At $x=12/25$, these bounds give
\begin{align*}
    x^{2q-1}(2-x)
    <x^{4/7}\frac{38}{25}<1,
\end{align*}
where the last inequality is equivalent to $12^4 38^7<25^{11}$. Hence $\xi_*>12/25$. Since $2q<8/5$,
\begin{align*}
    \xi_*^{2q}>\xi_*^{8/5}>\left(\frac{12}{25}\right)^{8/5}>\frac14;
\end{align*}
the last inequality is equivalent to $12^8 4^5>25^8$. Together with \eqref{eq:obsf-contact-identity}, this proves \eqref{eq:obsf-contact-interval} and the remaining assertions.
\end{proof}

\subsection{Deficit concentration and nonconstancy}\label{app:obsf-deficit-concentration}

\begin{lemma} \label{lem:obsf-dyadic-strict}
The dyadic-ray limit satisfies
\begin{align*}
    C_1^F=\lim_{k\to\infty}(2^k)^p\eta^{\mathrm{F}}_{2^k}>c_*.
\end{align*}
\end{lemma}

\begin{proof}
Suppose instead that $C_1^F=c_*$. For $N=2^k$, this is equivalent to
\begin{align}\label{eq:obsf-dyadic-zero-deficit-app}
    \frac{W_N}{N^p}\to d_*=B_{\mathrm{sup}}^p,
    \quad
    \frac{E_N}{N}
    =B_{\mathrm{sup}}-\left(\frac{W_N}{N^p}\right)^q
    \to0.
\end{align}
Starting at $n_0=N$, choose a maximizing pivot $m_i$ for $W_{n_i}$ and set $n_{i+1}=n_i-m_i$ until $n_L=1$. Iterating \eqref{eq:obsf-deficit-decomposition} gives
\begin{align}\label{eq:obsf-spine-deficit-app}
    E_N
    =E_1+
    \sum_{i=0}^{L-1}
    \left[B_{\mathrm{sup}}(m_i-U_{m_i}^q)+\sigma(U_{m_i},W_{n_{i+1}})\right],
\end{align}
with every summand nonnegative.

Small pivots incur a linear support deficit. Since $q<1$ and $\mathcal A(t,1)=1+O(t)$ as $t\downarrow0$,
\begin{align*}
    \frac{\sigma(t,1)}{t^q}\to B_{\mathrm{sup}}.
\end{align*}
By homogeneity, there is $\delta>0$ such that
\begin{align*}
    0<\frac{u}{w}\le\delta
    \quad\Longrightarrow\quad
    \sigma(u,w)\ge\frac{B_{\mathrm{sup}}}{2}u^q.
\end{align*}
The sandwich and symmetric phase bounds give
\begin{align*}
    W_r\ge U_r\ge\phi_*r^p,
    \quad
    U_m\le m^p,
    \quad
    U_m^q\ge\phi_*^qm.
\end{align*}
Choose $\varepsilon>0$ so that
\begin{align*}
    \phi_*^{-1}\left(\frac{\varepsilon}{1-\varepsilon}\right)^p\le\delta.
\end{align*}
If $m_i/n_i\le\varepsilon$, then $n_{i+1}\ge(1-\varepsilon)n_i$ and $U_{m_i}/W_{n_{i+1}}\le\delta$. Thus
\begin{align}\label{eq:obsf-small-pivot-cost-app}
    \sigma(U_{m_i},W_{n_{i+1}})
    \ge c_0m_i,
    \quad
    c_0\coloneqq\frac{B_{\mathrm{sup}}}{2}\phi_*^q>0.
\end{align}

Let $j$ be the first index with $m_j/n_j>\varepsilon$. Such an index exists for all large $N$; otherwise \eqref{eq:obsf-spine-deficit-app}, \eqref{eq:obsf-small-pivot-cost-app}, and $\sum_i m_i=N-1$ would contradict \eqref{eq:obsf-dyadic-zero-deficit-app}. Moreover,
\begin{align*}
    \sum_{i<j}m_i\le\frac{E_N}{c_0}=o(N),
    \quad
    \frac{n_j}{N}\to1.
\end{align*}
Nonnegativity in \eqref{eq:obsf-spine-deficit-app} then gives
\begin{align}\label{eq:obsf-macro-zero-deficits-app}
    m_j-U_{m_j}^q=o(N),
    \quad
    E_{n_{j+1}}=o(N),
    \quad
    \sigma(U_{m_j},W_{n_{j+1}})=o(N).
\end{align}

Pass to a subsequence with $m_j/n_j\to\alpha\in[\varepsilon,1]$. The preceding relations imply
\begin{align*}
    \frac{U_{m_j}^q}{n_j}\to\alpha,
    \quad
    \frac{W_{n_{j+1}}^q}{n_j}\to B_{\mathrm{sup}}(1-\alpha).
\end{align*}
Because $\sigma$ is homogeneous of degree $q$ and $pq=1$, its vanishing limit becomes
\begin{align*}
    B_{\mathrm{sup}}\alpha+B_{\mathrm{sup}}(1-\alpha)
    -\mathcal A\bigl(\alpha^p,[B_{\mathrm{sup}}(1-\alpha)]^p\bigr)^q=0.
\end{align*}
The case $\alpha=1$ is impossible because $B_{\mathrm{sup}}>2^q$. Equality in Lemma \ref{lem:obsf-support} therefore forces
\begin{align*}
    \frac{\alpha^p}{[B_{\mathrm{sup}}(1-\alpha)]^p}=\tau_*,
    \quad\text{hence}\quad
    \alpha=\frac{B_{\mathrm{sup}}r_*}{1+B_{\mathrm{sup}}r_*}=\alpha_*.
\end{align*}
Thus every convergent subsequence has the same limit, and $m_j/N\to\alpha_*\in(1/2,1)$.

Since $N=2^k$, eventually $2^{k-1}<m_j<2^k$ and $2m_j/N\to x_*=2\alpha_*\in(1,2)$. Exact interpolation on $[2^{k-1},2^k]$ gives
\begin{align*}
    \frac{U_{m_j}^q}{m_j}
    =\frac{F(2m_j/N)^q}{2m_j/N}.
\end{align*}
The left-hand side tends to one by \eqref{eq:obsf-macro-zero-deficits-app}. Passage to the limit yields $F(x_*)^q=x_*$, or $F(x_*)=x_*^p$, contradicting Lemma \ref{lem:obsf-strict-S-phase}. Therefore $C_1^F>c_*$.
\end{proof}

\begin{lemma} \label{lem:obsf-constant-rigidity}
If the phase profile in Theorem \ref{thm:obsf-continuous-phase} were constant, then it would be identically equal to $c_*$.
\end{lemma}

\begin{proof}
Suppose $\Psi_F\equiv c$ and set $d=c^{-1}$. The support bound gives $d\le d_*$. Let $h \geq 0$ and $b\geq 1$ be arbitrary integers, and set $a=2^h$. For each $k\geq 0$, consider the Bellman recursion for $W_{(a+b)2^k}$ and  use the admissible pivot $a2^k$. Exact scaling of $U$ and the constant-profile asymptotics give, after division by $2^{kp}$ and passage to the limit,
\begin{align*}
    d(a+b)^p\ge\mathcal A(a^p,db^p).
\end{align*}
By homogeneity, with $\alpha=a/(a+b)$,
\begin{align}\label{eq:obsf-constant-bellman}
    d\ge\mathcal A(\alpha^p,d(1-\alpha)^p).
\end{align}
The set  $\{b/2^h: h \geq 0,b\geq 1\}$ is dense in $(0,\infty)$, so \eqref{eq:obsf-constant-bellman} holds for every $\alpha\in[0,1]$.

If $d^q<B_{\mathrm{sup}}$, choose $\alpha=d^qr_*/(1+d^qr_*)$. The two positive arguments in \eqref{eq:obsf-constant-bellman} then have equality ratio $\tau_*$. By Lemma \ref{lem:obsf-support}, the joined value has $q$-th power
\begin{align*}
    B_{\mathrm{sup}}\alpha+d^q(1-\alpha)>d^q,
\end{align*}
contradicting \eqref{eq:obsf-constant-bellman}. Hence $d^q\ge B_{\mathrm{sup}}$, so $d\ge d_*$. Therefore $d=d_*$ and $c=c_*$.
\end{proof}

The preceding two results imply that $\Psi_F$ is nonconstant: a constant profile would equal $c_*$ by Lemma \ref{lem:obsf-constant-rigidity}, contradicting its strictly larger value at the dyadic phase from Lemma \ref{lem:obsf-dyadic-strict}.

\subsection{Global strict separation}\label{app:obsf-global-separation}

\begin{lemma} \label{lem:obsf-global-strict}
For every $t\in\mathbb T$,
\begin{align*}
    \Psi_F(t)>c_*.
\end{align*}
Consequently, $\min_{t\in\mathbb T}\Psi_F(t)>c_*$.
\end{lemma}

\begin{proof}
By Proposition~\ref{prop:obsf-deficit} and the uniform phase convergence in Theorem~\ref{thm:obsf-continuous-phase}, $\Psi_F(t)\ge c_*$ for every $t\in\mathbb T$. Suppose, for a contradiction, that $\Psi_F(t_0)=c_*$ for some $t_0\in\mathbb T$. Choose integers $N_k\to\infty$ such that $\{\log_2N_k\}\to t_0$ in $\mathbb T$. Then
\begin{align*}
  \frac{W_{N_k}}{N_k^p}\to d_*,
  \quad
  \frac{E_{N_k}}{N_k}
  =B_{\mathrm{sup}}-
    \left(\frac{W_{N_k}}{N_k^p}\right)^q
  \to0.
\end{align*}

We first make explicit the concentration argument underlying estimate \eqref{eq:obsf-small-pivot-cost-app}, which applies to any integer sequence with $E_{N_k}/N_k\to0$. Along a maximizing right spine, write
\begin{align*}
  n_{k,0}=N_k,\quad
  n_{k,i+1}=n_{k,i}-m_{k,i},\quad n_{k,L_k}=1.
\end{align*}
The deficit identity of Proposition~\ref{prop:obsf-deficit} gives
\begin{align*}
 E_{N_k}=E_1+
 \sum_{i=0}^{L_k-1}
 \left[
 B_{\mathrm{sup}}\bigl(m_{k,i}-U_{m_{k,i}}^q\bigr)
 +\sigma\bigl(U_{m_{k,i}},W_{n_{k,i+1}}\bigr)
 \right],
\end{align*}
and every term is nonnegative. The small-pivot estimate \eqref{eq:obsf-small-pivot-cost-app} supplies constants $\varepsilon\in(0,1)$ and $c_0>0$, independent of the initial size, such that
\begin{align*}
 \frac{m_{k,i}}{n_{k,i}}\le\varepsilon
 \quad\implies\quad
 \sigma\bigl(U_{m_{k,i}},W_{n_{k,i+1}}\bigr)
 \ge c_0m_{k,i}.
\end{align*}
For all sufficiently large $k$, there must therefore be a first index $j_k$ with $m_{k,j_k}/n_{k,j_k}>\varepsilon$. Otherwise, $E_{N_k}\ge c_0(N_k-1)$, contradicting $E_{N_k}/N_k\to0$. Set
\begin{align*}
 b_k\coloneqq \sum_{i<j_k}m_{k,i},\quad
 s_k\coloneqq n_{k,j_k}=N_k-b_k,\quad
 m_k\coloneqq m_{k,j_k},\quad
 r_k\coloneqq n_{k,j_k+1}=s_k-m_k.
\end{align*}
In particular, $r_k$ is the actual continuation size after the first macroscopic pivot. Nonnegativity and the small-pivot estimate yield
\begin{align*}
 0\le b_k\le\frac{E_{N_k}}{c_0}=o(N_k),
 \quad \frac{s_k}{N_k}\to1.
\end{align*}
Applying the deficit identity at $s_k$, and observing that the deficit at any later spine node is at most $E_{N_k}$, also gives
\begin{align*}
 m_k-U_{m_k}^q=o(N_k),\quad
 E_{r_k}=o(N_k),\quad
 \sigma(U_{m_k},W_{r_k})=o(N_k).
\end{align*}

Consider any subsequence on which $m_k/s_k\to\beta\in[\varepsilon,1]$. Since $r_k=s_k-m_k$, the preceding estimates imply
\begin{align*}
 \frac{U_{m_k}^q}{s_k}\to\beta,
 \quad
 \frac{W_{r_k}^q}{s_k}
 =B_{\mathrm{sup}}\frac{r_k}{s_k}-\frac{E_{r_k}}{s_k}
 \to B_{\mathrm{sup}}(1-\beta).
\end{align*}
Using $pq=1$, continuity, and the degree-$q$ homogeneity of $\sigma$, we obtain
\begin{align*}
 0=B_{\mathrm{sup}}-
 \mathcal A\!\left(\beta^p,
          [B_{\mathrm{sup}}(1-\beta)]^p\right)^q.
\end{align*}
The case $\beta=1$ would give $B_{\mathrm{sup}}=2^q$, which is
excluded by Lemma~\ref{lem:obsf-support}. Thus both arguments of $\mathcal A$ are positive, and the equality characterization in that lemma forces
\begin{align*}
 \frac{\beta^p}{[B_{\mathrm{sup}}(1-\beta)]^p}=\tau_*,
 \quad
 \beta=\frac{B_{\mathrm{sup}}\tau_*^q}
             {1+B_{\mathrm{sup}}\tau_*^q}=\alpha_*.
\end{align*}
Every convergent subsequence of $m_k/s_k$ has this same limit. Consequently,
\begin{align*}
 \frac{m_k}{N_k}\to\alpha_*,\quad
 \frac{r_k}{N_k}\to1-\alpha_*,\quad
 \frac{U_{m_k}^q}{m_k}\to1,\quad
 \frac{E_{r_k}}{r_k}\to0.
\end{align*}
Here the last limit uses $1-\alpha_*>0$; in particular, $m_k,r_k\to\infty$.

The exact symmetric phase law now gives
\begin{align*}
 \frac{U_{m_k}^q}{m_k}
 =\Phi\!\left(\{\log_2m_k\}\right)^q\to1.
\end{align*}
By Lemma~\ref{lem:obsf-strict-S-phase}, $\Phi(t)=1$ only at the zero phase. Continuity and compactness of $\mathbb T$ therefore imply $\{\log_2m_k\}\to0$ in $\mathbb T$. Since
\begin{align*}
 \{\log_2m_k\}
 =\{\log_2N_k+\log_2(m_k/N_k)\},
\end{align*}
we conclude that
\begin{align*}
 t_0+\log_2\alpha_*=0\quad\text{in }\mathbb T.
\end{align*}
On the other hand, $E_{r_k}/r_k\to0$ implies
$W_{r_k}/r_k^p\to d_*$, while the actual continuation phases satisfy
\begin{align*}
 \{\log_2r_k\}\to
 t_1\coloneqq t_0+\log_2(1-\alpha_*)\quad\text{in }\mathbb T.
\end{align*}
Uniform phase convergence hence gives $\Psi_F(t_1)=c_*$. The same concentration argument, now applied to the sequence $r_k$, yields
\begin{align*}
 t_1+\log_2\alpha_*=0\quad\text{in }\mathbb T.
\end{align*}
Subtracting the two phase identities shows that $\log_2(1-\alpha_*)=0$ in $\mathbb T$. Thus $1-\alpha_*$ would be an integral power of two, contrary to $1/4<1-\alpha_*<1/2$ from Lemma~\ref{lem:obsf-contact-fraction}. This contradiction proves $\Psi_F(t)>c_*$ for every phase. Finally, $\Psi_F$ is continuous on the compact circle and attains its minimum, so $\min_{t\in\mathbb T}\Psi_F(t)>c_*$.
\end{proof}

\begin{proposition} \label{prop:obsf-phase-enclosure}
For every $t\in\mathbb T$,
\begin{align*}
    \Psi_F(t)\le\frac{c_*}{\phi_*}.
\end{align*}
Equivalently, with $\Delta_F(t)=1/\Psi_F(t)$, one has $\Delta_F(t)\ge d_*\phi_*$.
\end{proposition}

\begin{proof}
Let $t_0$ minimize $\Delta_F$, and choose $N_k\to\infty$ whose phases approach $t_0$. Fix $\alpha\in(0,1)$ and put $m_k=\floor{\alpha N_k}$ and $r_k=N_k-m_k$. The Bellman recursion, the minimum property of $\Delta_F(t_0)$, and $U_{m_k}/m_k^p\ge\phi_*$ give, after division by $N_k^p$ and passage to the limit,
\begin{align}\label{eq:obsf-limit-bellman-lower}
    \Delta_F(t_0)
    \ge\mathcal A\bigl(\phi_*\alpha^p,\Delta_F(t_0)(1-\alpha)^p\bigr).
\end{align}
If $\Delta_F(t_0)^q<B_{\mathrm{sup}}\phi_*^q$, choose
\begin{align*}
    \alpha=\frac{\Delta_F(t_0)^qr_*}{\phi_*^q+\Delta_F(t_0)^qr_*}.
\end{align*}
The two positive arguments in \eqref{eq:obsf-limit-bellman-lower} then have equality ratio $\tau_*$. Their joined value has $q$-th power
\begin{align*}
    B_{\mathrm{sup}}\phi_*^q\alpha+\Delta_F(t_0)^q(1-\alpha)>\Delta_F(t_0)^q,
\end{align*}
contradicting \eqref{eq:obsf-limit-bellman-lower}. Hence $\Delta_F(t_0)^q\ge B_{\mathrm{sup}}\phi_*^q$, so $\Delta_F(t_0)\ge d_*\phi_*$. Since $t_0$ minimizes $\Delta_F$, this bound holds at every phase; reciprocation proves the result.
\end{proof}
\subsection{Completion of the asymmetric phase theorem}\label{app:obsf-complete-proof}
\begin{theorem} \label{cor:obsf-complete-phase}
The phase $\Psi_F$ is nonconstant and satisfies
\begin{align*}
    c_*<\min_{t\in\mathbb T}\Psi_F(t)
    <\max_{t\in\mathbb T}\Psi_F(t)
    \le\frac{c_*}{\phi_*}.
\end{align*}
Consequently, the cluster interval in \eqref{eq:obsf-cluster-set} is nondegenerate and lies strictly above the sharp homogeneous support barrier. By \eqref{eq:obsf-objective-phase}, this is precisely the nondegenerate interval of asymptotic leading constants in the tight objective-gap guarantee for optimized ConPD/OBS-F schedules.
\end{theorem}

\begin{proof}
The support and deficit decomposition are established in Proposition \ref{prop:obsf-deficit}. The dyadic-ray separation and constant-phase rigidity in Lemma \ref{lem:obsf-dyadic-strict} and Lemma \ref{lem:obsf-constant-rigidity} imply that $\Psi_F$ is nonconstant. The global strict lower bound and the upper phase enclosure are proved in Lemma \ref{lem:obsf-global-strict} and Proposition \ref{prop:obsf-phase-enclosure}. These statements give \eqref{eq:obsf-final-enclosure}, while \eqref{eq:obsf-cluster-set} identifies the complete cluster interval.
\end{proof}
Together with the uniform convergence and cluster-set statement above,
this proves Theorem~\ref{thm:asymmetric-main}.

\section{Additional numerical phase illustrations}
\label{app:phase-plots}

For the symmetric family, figure~\ref{fig:phi-phase} displays the
underlying scalar phase $\Phi$, whereas the left panel of
figure~\ref{fig:phases} displays the corresponding objective-gap
leading phase $1/(2\Phi)$.
The detailed plot complements the exact phase law and the rigorous
enclosure of $\phi_*=\min\Phi$ established in
Appendix~\ref{app:certificates}.

\begin{figure}[htbp]
    \centering
    \includegraphics[width=0.96\linewidth]{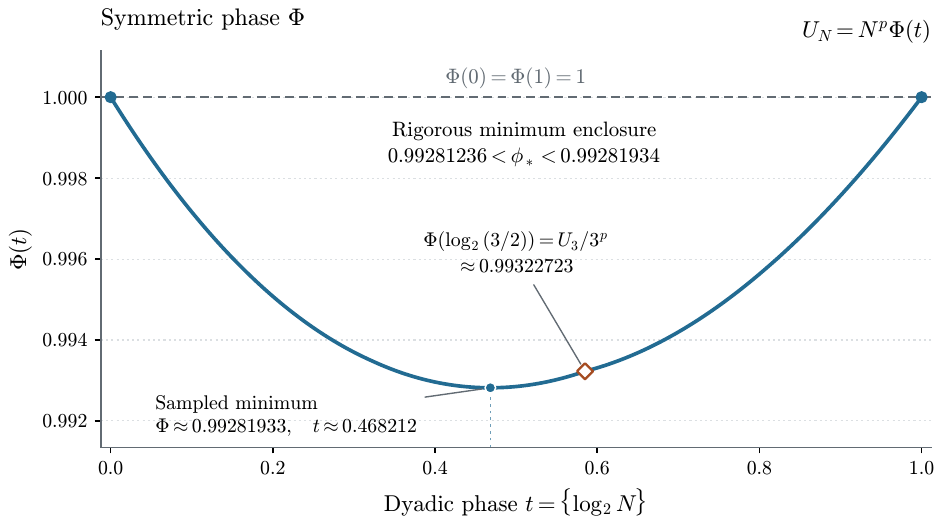}
    \caption{
    Detailed symmetric scalar-phase illustration.
    The annotations show an exact nonconstancy witness and a sampled
    minimum, together with the certified range of the phase minimum.
    }
    \label{fig:phi-phase}
\end{figure}

For the asymmetric family, figure~\ref{fig:psi-phase} displays the
finite-scale interpolant $\Psi_{18}$ over one dyadic block.
It provides a more detailed numerical view of the asymmetric phase
than the compact panel in figure~\ref{fig:phases}.
Theorem~\ref{thm:asymmetric-main} establishes the uniform convergence
$\Psi_k\to\Psi_F$ and gives the limiting interpretation of these
finite-scale profiles.

\begin{figure}[htbp]
    \centering
    \includegraphics[width=0.96\linewidth]{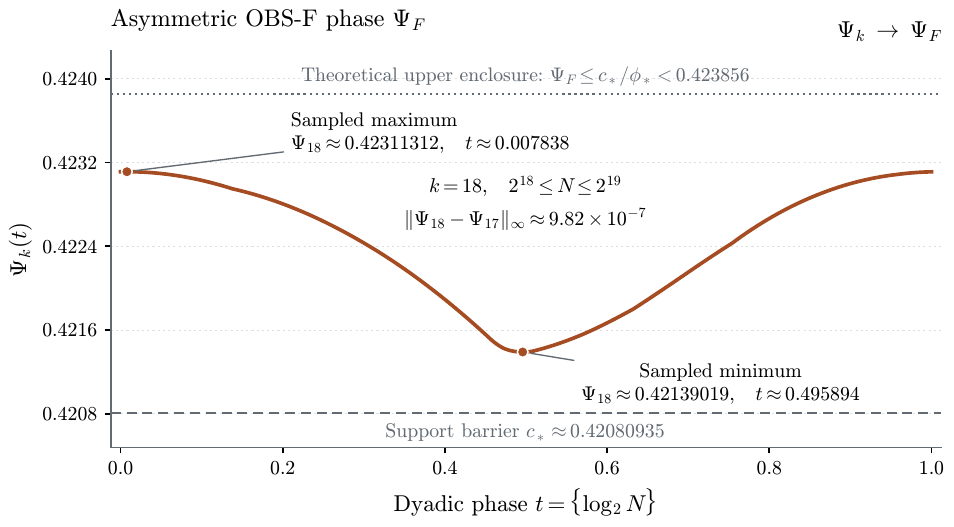}
    \caption{
    Detailed finite-scale asymmetric phase illustration.
    The curve is $\Psi_{18}$.
    The annotations show its sampled extrema, the sharp support barrier
    $c_*$, the theoretical upper enclosure for $\Psi_F$, and the
    successive-scale comparison with $\Psi_{17}$.
    }
    \label{fig:psi-phase}
\end{figure}

Implementation details, plotting data, reproduction commands,
interval-arithmetic checks, and software and hardware information are
provided in the supplementary material.

\end{document}